\documentclass[aap,MSNbibl,dvips]{arximspdf}
\usepackage[ruled,lined]{algorithm2e}
\usepackage{mathbh}
\usepackage{dcolumn}
\usepackage{graphicx}

\doi{10.1214/12-AAP863} 
\volume{23}
\issue{3}
\pubyear{2013}
\firstpage{1025}
\lastpage{1073}

\makeatletter

\newcolumntype{d}[1]{D{.}{.}{#1}}

\renewcommand{\overline}{\bar}

\newcommand{\underset}[2]{\displaystyle \mathop{#2}_{#1}}

\newcommand{\implies}{\Longrightarrow}
\newcommand{\WIS}{\operatorname{WIS}}
\newcommand{\AFF}{\operatorname{AFF}}

\newcommand{\R}{\mathbb{R}}
\newcommand{\D}{\mathbb{D}}
\newcommand{\E}{\mathbb{E}}
\newcommand{\N}{\mathbb{N}}
\newcommand{\Px}{\mathbb{P}}

\newcommand{\DM}{D^{\mathcal{M}}}
\newcommand{\DS}{D^{\mathcal{S}}}

\newcommand{\xcn}{\hat{X}^N}

\newtheorem{theorem}{Theorem}

\newtheorem{proposition}[theorem]{Proposition}

\newtheorem{corollary}[theorem]{Corollary}

\newtheorem{lemma}[theorem]{Lemma}

\newproclaim{remark}[theorem]{Remark}

\newcommand{\symm}{{\mathcal{S}_d(\mathbb R)}}
\newcommand{\nsing}{{\mathcal{G}_d(\mathbb R)}}
\newcommand{\genm}{{\mathcal{M}_d(\mathbb R)}}
\newcommand{\dpos}{{\mathcal{S}_d^{+,*}(\mathbb R)}}
\newcommand{\posm}{{\mathcal{S}_d^+(\mathbb R)}}

\newcommand{\Cpol}{\mathcal{C}^\infty_{\mathrm{pol}}(\D) }
\newcommand{\Cpolde}[1]{\mathcal{C}^\infty_{\mathrm{pol}}(#1) }

\newcommand{\Tr}{\operatorname{Tr}} 
\newcommand{\Rg}{\operatorname{Rk}} 
\newcommand{\adj}{\operatorname{adj}} 

\makeatother

\begin{document}
\begin{frontmatter}

\title{Exact and high-order discretization schemes for Wishart
processes and their affine extensions\thanksref{T1}}
\runtitle{Simulation schemes for Wishart processes}

\thankstext{T1}{Supported by the Credinext project from Finance
Innovation, of the
Eurostars E!5144-TFM project, and of the ``Chaire Risques Financiers''
of Fondation du Risque.}

\begin{aug}
\author[A]{\fnms{Abdelkoddousse} \snm{Ahdida}\ead[label=e1]{ahdidaa@cermics.enpc.fr}}
\and
\author[A]{\fnms{Aur\'elien} \snm{Alfonsi}\corref{}\ead[label=e2]{alfonsi@cermics.enpc.fr}}
\runauthor{A. Ahdida and A. Alfonsi}
\affiliation{Universit\'e Paris-Est}
\address[A]{CERMICS\\
Universit\'e Paris-Est\\
Project team MathFi ENPC-INRIA-UMLV\\
Ecole des Ponts\\
6-8 avenue Blaise Pascal\\
77455 Marne La Vall\'ee\\
France\\
\printead{e1}\\
\hphantom{E-mail: }\printead*{e2}} 
\end{aug}

\received{\smonth{4} \syear{2011}}
\revised{\smonth{3} \syear{2012}}

%
\begin{abstract}
This work deals with the simulation of Wishart processes and affine
diffusions on positive semidefinite matrices. To do so, we focus on the
splitting of the infinitesimal generator in order to use composition
techniques as did Ninomiya and Victoir [\textit{Appl. Math. Finance}
\textbf{15} (2008) 107--121] or Alfonsi [\textit{Math. Comp.}
\textbf{79} (2010) 209--237]. Doing so, we have found a remarkable
splitting for Wishart processes that enables us to sample exactly
Wishart distributions without any restriction on the parameters. It is
related but extends existing exact simulation methods based on
Bartlett's decomposition. Moreover, we can construct high-order
discretization schemes for Wishart processes and second-order schemes
for general affine diffusions. These schemes are, in practice, faster
than the exact simulation to sample entire paths. Numerical results on
their convergence are given.
\end{abstract}

%

\end{frontmatter}

\section*{Introduction}\label{Chapter1}
This paper focuses on simulation methods for Wishart processes
and more generally for affine diffusions on positive semidefinite
matrices. Before explaining our motivations and our main results, we
start with a short introduction to these processes. Even though we use
rather standard notation for matrices, they are recalled at the end of
the \hyperref[Chapter1]{Introduction}, and we invite the reader to first
give a quick look at it. Wishart processes have been initially
introduced by Bru~\cite{Bruthesis,Bru}. They are also named because
their marginal laws follow Wishart distributions. Very recently,
Cuchiero et al.~\cite{Teichmann} have introduced a general framework
for affine processes on positive semidefinite matrices $\posm$ that
embeds Wishart processes and includes possible jumps. In this paper, we
only consider continuous processes of this kind. Such processes solve
the following SDE:
%
\begin{equation}\label{EDSAFFINE}
X_t^x =x + \int_{0}^t \bigl( \overline{\alpha} + B(X_s^x)
\bigr)\,ds +
\int_{0}^t \bigl( \sqrt{X_s^x}\,dW_sa + a^T\,dW_s^T\sqrt{X_s^x}\bigr).
\end{equation}
Here, and throughout the paper, $(W_t,t \ge0)$ denotes a $d$-by-$d$
square matrix made of independent
standard Brownian motions and
%
\begin{equation}\label{paramaffine}
x, \bar{\alpha} \in\posm,\qquad a \in\genm
\quad\mbox{and}\quad B\in\mathcal{L}(\symm)
\end{equation}
is a linear mapping on $\symm$. Wishart processes
correspond to the case where
%
\begin{eqnarray}\label{paramwishart}
&&\exists\alpha\ge0, \qquad\bar
{\alpha}=\alpha a^T a \quad\mbox{and}\nonumber\\[-8pt]\\[-8pt]
&&\exists b \in
\genm, \forall x \in\symm\qquad B(x)=bx+xb^T.\nonumber
\end{eqnarray}
When $d=1$, (\ref{EDSAFFINE}) is simply
the SDE of the Cox--Ingersoll--Ross (CIR) process that has been broadly
studied, and
we will implicitly assume that $d\ge2$ throughout the paper. Weak and
strong uniqueness of
SDE (\ref{EDSAFFINE}) has been studied by Bru~\cite{Bru}, Cuchiero et
al.~\cite{Teichmann} and Mayerhofer, Pfaffel and Stelzer \cite
{STELZER}. Here we sum up their results.
%
\begin{theorem}
If $x \in\posm$, $\bar{\alpha}-(d-1)a^T a \in\posm$ and $B$
satisfies the following condition:
%
\begin{equation}\label{ReQuiredAssumption}
\forall x_1, x_2 \in\posm\qquad \Tr(x_1x_2)=0 \quad\implies\quad
\Tr(B(x_1)x_2) \geq0,
\end{equation}
there is a unique weak solution to the SDE (\ref{EDSAFFINE}) in $\posm
$. We denote
by $\AFF_d(x,\overline{\alpha},B,a)$ the law of $(X^x_t)_{t \geq0}$ and
$\AFF_d(x,\overline{\alpha},B,a;t)$ the marginal law of $X_t^x$. If we assume,
moreover, that $\bar{\alpha}-(d+1)a^T a \in\posm$ and $x\in\dpos
$, there is a unique strong solution to the
SDE (\ref{EDSAFFINE}).

Under the parametrization of Wishart processes (\ref{paramwishart}),
condition (\ref{ReQuiredAssumption}) is satisfied and weak uniqueness
holds as
soon as $\alpha\ge d-1$. In that case, we denote by $\WIS_d(x,\alpha
,b,a)$ the law of the Wishart process $(X_t^x)_{t\geq0}$ and $
\WIS_d(x,\alpha,b,\break a;t)$ the law of $X_t^x$.
\end{theorem}

Throughout the paper, when we use the notation $\AFF_d(x,\overline
{\alpha},B,a)$
or $\AFF_d(x$, $\overline{\alpha},B,a;t)$ [resp., $\WIS_d(x,\alpha,b,a)$
or $
\WIS_d(x,\alpha,b,a;t)$], we implicitly assume that $\bar{\alpha
}-(d-1)a^T a
\in\posm$ (resp., $\alpha\ge d-1$) and $B$
satisfies (\ref{ReQuiredAssumption}) so that weak
uniqueness holds.

In her Ph.D. thesis~\cite{Bruthesis}, Bru introduced Wishart
processes and used them in biology to study perturbed experimental
data. Recently, great attention has been paid to Wishart processes
for applications in finance. Namely, Gourieroux and Sufana~\cite{Gourieroux} and Da Fonseca, Grasselli and Tebaldi
\cite{Dafonseca} have suggested the use of these processes to model the
instantaneous covariance matrix of $d$ assets. It naturally extends
stochastic volatility models for only one asset like the Heston
model~\cite{Heston}. Obviously, processes on positive semidefinite
matrices are really interesting to model the evolution of a dependence
structure because they can describe a covariance matrix. However, when
dealing with applications, it is, in general, crucial to be able to
sample paths of such processes and make Monte Carlo algorithms.

To the best of our knowledge, there is minimal literature on simulation methods
for Wishart and general affine processes (\ref{EDSAFFINE}). Wishart
distributions have been intensively
studied in statistics when $\alpha\in\N$. In this case, exact simulation
methods have been proposed by Odell and Feiveson~\cite{Odell}, Smith and
Hocking~\cite{Hocking} and Gleser~\cite{Gleser}, to mention a few. Concerning
discretization schemes, the usual Euler--Maruyama scheme is not well defined
because of the square-root. This already happens for the CIR process
($d=1$). One has then to find specific schemes. Recently,
Benabid et al.~\cite{Harry} and Gauthier and Possamai~\cite{GP} have
proposed numerical approximations for Wishart processes that are
well defined under some
restrictions on the parameters. However, there is no result on the
accuracy of
their methods. Currently,
Teichmann~\cite{semBachelier} is working on dedicated schemes for general
affine processes by approximating their characteristic functions. Our study
here is only dedicated to the diffusion (\ref{EDSAFFINE}).

Initially, our goal was to find high-order discretization schemes for Wishart
processes by splitting operators and using scheme compositions. Indeed, this
approach has already proved to be very efficient for other affine
diffusions (see~\cite{Alfonsi}). The main difficulty here was to find a
splitting that involves infinitesimal generators of diffusions that are
well defined on $\posm$ and that can be simulated. Doing so, we
incidentally have found a remarkable
splitting for some canonical Wishart processes: the infinitesimal
generator of
$\WIS_d(x,\alpha,0,I^n_d)$ is the sum of commuting operators that are
associated to elementary SDEs that can be sampled exactly. With the
help of
a simple but useful law identity, this enables us to sample
exactly Wishart processes for any admissible parameter.
In particular, our result extends the Bartlett's
decomposition that is commonly used to sample central Wishart
distributions. This splitting is not only interesting for the exact
simulation method. It is also useful to construct high-order
discretization schemes for Wishart processes that are, in practice,
faster to
generate full paths. In fact, it allows us to
get a high-order scheme that preserves the domain $\posm$. We provide a
rigorous analysis of the weak error in this framework. Still, by using the
splitting technique, we also get a second-order scheme for any affine diffusion (\ref{EDSAFFINE}) without any
restriction on the parameters.

This paper is structured as follows. First, we present some general
results on
affine diffusions. We calculate their infinitesimal generator and obtain
interesting identities in law that are intensively used next for the
different simulation methods. Section~\ref{sec_exact_wish} is devoted
to the
exact simulation of Wishart processes. It exhibits\vadjust{\goodbreak} the remarkable
splitting of
the infinitesimal generator and
shows how it can be used to sample exactly any Wishart
distribution. Section~\ref{sec_high} deals with high-order schemes for affine
diffusions. Thanks to the remarkable splitting, we are able to
construct a
third-order scheme for Wishart processes and second-order schemes for affine
diffusions. Last, we give numerical illustrations of our convergence results
in Section~\ref{SecSimul}. We compare the time required by each
method and
also give a possible application of our results in finance.

\subsection*{Notation for real matrices}\label{Notationsmatrices}
\begin{itemize}
\item For $d,d'\in\N^*$, $\mathcal{M}_{d}(\mathbb R)$ denotes the
real $d$
square matrices and $\mathcal{M}_{d\times d'}(\mathbb R)$ the real matrices
with $d$ rows and $d'$ columns.
\item$\mathcal{S}_{d}(\mathbb R)$,
$\mathcal{S}_d^+(\mathbb R), \mathcal{S}_d^{+,*}(\mathbb R) $ and
$\mathcal{G}_d(\mathbb R)$ denote, respectively, the set of
symmetric, symmetric positive semidefinite, symmetric positive definite and
nonsingular matrices.
\item For $x \in\mathcal{M}_{d}(\mathbb R)$,
$x^T$, $\adj(x)$, $\det(x)$, $\Tr(x)$ and $\Rg(x)$ are,
respectively, the
transpose, the adjugate, the determinant, the trace and the rank of $x$.
\item For $x\in\mathcal{S}_d^+(\mathbb R)$, $\sqrt{x}$ denotes the unique
symmetric positive semidefinite matrix such that $(\sqrt{x})^2=x$.
\item The identity matrix is denoted by $I_d$ and we set for $n\le d$,
$I^n_d=(\mathbh{1}_{i=j\le n})_{1\le i,j \le d}$ and
$e^n_d=(\mathbh{1}_{i=j= n})_{1\le i,j \le d}$, so that $I^n_d=\sum_{i=1}^n
e^i_d$. We also set for $1\le i,j\le d$,
$e^{i,j}_d=(\mathbh{1}_{k=i,l=j})_{1\le k,l \le d}$.
\item For $x \in\symm$, we denote by $x_{\{i,j\}}$ the value of $x_{i,j}$,
so that
\[
x=\sum_{1 \le i\le j\le d} x_{\{i,j\}} (e^{i,j}_d+
\mathbh{1}_{i\not= j} e^{j,i}_d) .
\]
We use both notation in the paper:
notation $(x_{i,j})_{1 \le i,j \le d}$ is more convenient for matrix
calculations while $(x_{\{i,j\}})_{1 \le i \le j \le d}$ is
preferred to
emphasize that we work on symmetric matrices.
\item For $\lambda_1,\ldots,\lambda_d \in\R$,
$\operatorname{diag}(\lambda_1,\ldots,\lambda_d) $ denotes the diagonal matrix such
that $\operatorname{diag}(\lambda_1,\ldots,\lambda_d)_{i,i}=\lambda_i$.
\end{itemize}

\section{Some properties of affine processes on
positive semidefinite matrices}\label{chapter2}

\subsection{The infinitesimal generator on $\genm$ and $\symm$}

We start with a simple lemma. It is useful to calculate the
infinitesimal generator of processes on matrices.
%
\begin{lemma}\label{Invariantvol}
Let $(\mathcal{F}_t)_{t \ge0}$ denote the filtration generated by
$(W_t,t\ge0)$. We consider continuous
$(\mathcal{F}_t)$-adapted processes $(A_t)_{t \geq0}$, $(B_t)_{t\ge
0}$ and $(C_t)_{t \geq0}$,
respectively, valued in $\mathcal{M}_d(\mathbb R)$, $\mathcal
{M}_d(\mathbb
R)$ and $\mathcal{S}_d(\mathbb R)$, and a process $(Y_t)_{t \geq0}$
that admits the following
semimartingale decomposition:
%
\begin{equation}\label{semimg}
dY_t = C_t \,dt + B_t \,dW_t A_t + A_t^T \,dW_t^T B_t^T.
\end{equation}
Then, for $i,j,m,n\in\{1,\ldots,d\}$, the quadratic
covariation of
$(Y_t)_{i,j}$ and $(Y_t)_{m,n}$ is
%
\begin{eqnarray}
\label{Quadr}
&& d\langle(Y_t)_{i,j},(Y_t)_{m,n}\rangle\nonumber\\
&&\qquad =[
(B_tB_t^T)_{i,m}(A_t^TA_t)_{j,n}+(B_tB_t^T)_{i,n}(A_t^TA_t)_{j,m}\\
&&\qquad\quad\hspace*{2pt}{}
+(B_tB_t^T)_{j,m}(A_t^TA_t)_{i,n}+(B_tB_t^T)_{j,n}(A_t^TA_t)_{i,m}
]\,dt.\nonumber
\end{eqnarray}
\end{lemma}

It is worth noticing that the quadratic covariation given by (\ref{semimg})
depends on $A_t$ and $B_t$ only through the matrices $A_t^TA_t$ and
$B_tB_t^T$. Lemma~\ref{Invariantvol} enables us to easily calculate the
infinitesimal generator for the affine process (\ref{EDSAFFINE}) which is
defined by
\begin{eqnarray}
x\in\posm,\qquad L^{\mathcal{M}}f(x)=\lim_{t\rightarrow0^+}
\frac{\E[f(X^x_t)]-f(x)}{t} \qquad\qquad\nonumber\\
&&\eqntext{\mbox{for } f \in\mathcal{C}^2(\genm
,\R) \mbox{ with
bounded derivatives}.}
\end{eqnarray}
In fact, we get that the generator of $\AFF_d(x,\overline{\alpha
},B,a)$ is
given by
%
\begin{eqnarray}\label{InifitesimalGenerator}
L^{\mathcal{M}} &=& \Tr([\overline{\alpha} + B(x)]\DM)\nonumber\\
&&{}+ \tfrac
{1}{2}\{ 2\Tr(x\DM a^Ta \DM)+
\Tr(x({\DM})^Ta^Ta \DM)\\
&&\hspace*{113.5pt}{} + \Tr(x \DM a^Ta({\DM})^T)\},\nonumber
\end{eqnarray}
where $\DM=(\partial_{i,j})_{1 \le i,j\le d}$. Since we know that
the affine process $(X^x_t)_{t \geq0}$ takes values in $\posm\subset
\symm$, we can also look at the infinitesimal generator of this
diffusion on $\symm$, which is defined by
\begin{eqnarray}
x\in\posm,\qquad
L^{\mathcal{S}}f(x)=\lim_{t\rightarrow0^+} \frac{\E
[f(X^x_t)]-f(x)}{t} \qquad\nonumber\\
&&\eqntext{\mbox{for } f \in\mathcal{C}^2(\symm,\R)
\mbox{ with
bounded derivatives}.}
\end{eqnarray}
For $x \in\symm$, we denote by $x_{\{i,j\}}=x_{i,j}=x_{j,i}$ the value
of the coordinates $(i,j)$ and $(j,i)$, so that $x=\sum_{1\le i \le j
\le d} x_{\{i,j\}} (e^{i,j}_d+ \mathbh{1}_{i \not= j}e^{j,i}_d )$.
For\vspace*{1pt}
$f \in \mathcal{C}^2(\symm,\R)$, we then denote by
$\partial_{\{i,j\}}f$ its derivative with respect to $x_{\{i,j\}}$.
For $x \in \genm$, we set $\pi(x)=(x+x^T)/2$. It is such
that $\pi(x)=x$ for $x \in \symm$, and we have
\[
L^{\mathcal{S}}f(x)= L^{\mathcal{M}}f\circ\pi(x).
\]
By the chain rule, we have for $x\in\symm$,
$\partial_{i,j}f \circ\pi(x)= (\mathbh{1}_{i=j}+\frac{1}{2}\mathbh
1_{i\neq j})\partial_{\{i,j\}}f (x)$ and get from (\ref
{InifitesimalGenerator}) the following result.
%
\begin{proposition}\label{infiniGenS} The infinitesimal generator on
$\symm$ associated to
$\AFF_d(x$, $\overline{\alpha},B,a)$ is given by
%
\begin{equation}\label{GeneratorS}
L^{\mathcal{S}}= \Tr\bigl([\overline{\alpha}+ B(x)]D^{\mathcal{S}}\bigr)+
2\Tr(xD^{\mathcal{S}}a^TaD^{\mathcal{S}}),
\end{equation}
where $\DS$ is defined by $\DS_{i,j}= (\mathbh{1}_{i=j}+\frac
{1}{2}\mathbh
1_{i\neq j})\partial_{\{i,j\}}$, for $1\le i,j \le d$.
\end{proposition}

Of course, the generators $L^{\mathcal{M}}$ and $L^{\mathcal{S}}$ are
equivalent; one can be deduced from the other. However, $L^{\mathcal{S}}$
already embeds the fact that the process lies in $\symm$, which
reduces the
dimension from $d^2$ to $d(d+1)/2$ and gives, in practice, shorter
formulas. This is why we will mostly work in the sequel with infinitesimal
generators on $\symm$. Unless it is necessary to make the distinction with
$L^{\mathcal{M}}$, we will simply denote $L=L^{\mathcal{S}}$.

\subsection{The characteristic function of Wishart processes}

As for other affine processes, the characteristic function of affine processes
on positive semidefinite matrices can be obtained by solving two ODEs.
In the case
of Wishart processes, it is possible to solve explicitly these ODEs by solving
a matrix Riccati equation (see Levin~\cite{Levin}). Here, we give the closed
formula for the Laplace transform and a precise description of its set
of convergence.
%
\begin{proposition}\label{LaplaceGeneral}
Let $X^x_t\sim \WIS_d(x,\alpha,b,a;t)$, $q_t = \int_0^t
\exp(sb)a^Ta\exp(s\* b^T)\,ds$ and $m_t = \exp(tb)$. We introduce the set of
convergence of the Laplace transform of $X^x_t$,
$ \mathcal{D}_{b,a;t}=\{ v \in\symm, \E[\exp( \Tr(vX_t^x))]<
\infty\}$. This
is a convex open set that is given explicitly by
%
\begin{equation}\label{DomLapWis}\mathcal{D}_{b,a;t}=\{v \in\symm
, \forall s \in[0, t], I_d-2q_s v \in\nsing\}.
\end{equation}
Besides, the Laplace
transform of $X^x_t$ is well defined for $v=v_R+iv_I$ with $v_R \in
{\mathcal{D}}_{b,a;t},v_I \in\symm$
and is given by
%
\begin{equation}\label{CarWishart}
\mathbb E[\exp( \Tr(vX_t^x))] = \frac{\exp(\Tr
[v(I_d-2q_tv)^{-1}m_txm_t^T])}{\det(I_d-2q_tv)^{{\alpha}/{2}}}.
\end{equation}
\end{proposition}

The characteristic function corresponds to the case $v_R=0$ that clearly
belongs to $\mathcal{D}_{b,a;t}$. The proof of this result is given
in Appendix~\ref{proofLaplaceGeneral}. The formula (\ref
{CarWishart}) is
well known in the literature, and our contribution is to characterize precisely
the set of convergence. In particular, let us observe that $\rho I_d
\in
\mathcal{D}_{b,a;t}$ when $\rho>0$ is small enough, which will help
us to
study the Cauchy problem (Proposition~\ref{Flow}).

Last, let us remark here that
for $\tilde{X}^x_t\sim \WIS_d(x,\alpha,0,I^n_d;t)$, the formula above
becomes even
simpler and we have for $v=v_R+iv_I$ such that $v_R\in{\mathcal
{D}}_{b,a;t}, v_I \in\symm$,
%
\begin{equation}\label{CarWishartcan}
\mathbb E[\exp( \Tr(v\tilde{X}^x_t))] = \frac{\exp(\Tr
[v(I_d-2tI^n_dv)^{-1}x])}{\det(I_d-2tI^n_d v)^{{\alpha}/{2}}}.
\end{equation}

\subsection{Some identities in law for affine processes}

This section gives simple but interesting identities in law for affine
processes. First, we observe that their infinitesimal
generator (\ref{GeneratorS}) only depends on $a$ through $a^Ta$ and get
%
\begin{equation}\label{AFFidentite1}
\AFF_d(x,\bar{\alpha},B,a) \underset{\mathrm{Law}}{=}\AFF_d\bigl(x,\bar{\alpha
},B,\sqrt{a^Ta}\bigr).
\end{equation}
Also, it is natural to look at linear transformations of affine
processes. Let
$q \in\mathcal{G}_d(\mathbb R)$ and define $B_q\in\mathcal{L}(\symm)$
by $B_{q}(x)=(q^T)^{-1} B(q^T xq ) q^{-1}$. One easily has that $B$
satisfies (\ref{ReQuiredAssumption}) iff $B_q$
satisfies (\ref{ReQuiredAssumption}), and we get
%
\begin{equation}\label{AFFidentite2}
\AFF_d(x,\bar{\alpha},B,a) \underset{\mathrm{Law}}{=}q^T \AFF_d((q^{-1})^T x
q^{-1},(q^{-1})^T\bar{\alpha}q^{-1},B_q,aq^{-1})q,\hspace*{-20pt}
\end{equation}
since both processes solve the same martingale problem. An interesting
consequence is given in the following proposition: any affine
process can be obtained as a linear transformation of an affine process
for which
$\bar{\alpha}$ is a diagonal matrix and $a=I^n_d$. Since our main
goal here is to sample paths of such
processes, this says to us that it is sufficient to focus on this
special case.
%
\begin{proposition}\label{Canonform}
Let $n=\Rg(a)$ be the rank of $a^Ta$. Then, there exist a
diagonal matrix $\bar{\delta}$ and a nonsingular matrix $u\in\nsing
$ such that
$\bar{\alpha} = u^T\bar{\delta} u$ and $a^Ta=u^TI_d^nu$ and we have
\[
\AFF_d(x,\bar{\alpha},B,a)\underset{\mathit{Law}}{=}u^T \AFF_d
((u^{-1})^Txu^{-1},\bar{\delta},B_u,I_d^n)u,
\]
where $\forall y \in\symm, B_u(y)=(u^{-1})^TB(u^Tyu)u^{-1}$.
\end{proposition}

The proof of this result consists of algebraic arguments and is found in
Appendix~\ref{Proofprop5}. It gives, in particular, a general way to
compute $u$ and $\bar{\delta}$. Let us
notice, however, that in the case of Wishart processes, $u$ can
directly be
obtained by using a single extended Cholesky decomposition (Lemma \ref
{OuterProdDec}).

Up to now, we have stated identities for the law of affine processes. Thanks
to the explicit characteristic function of Wishart processes, we are
also able
to get another interesting identity on the marginal laws.
%
\begin{proposition}\label{Propidloiwis}
Let $t>0$, $a,b\in\genm$ and $\alpha\ge d-1$. Let $m_t=\exp(tb)$,
$q_t=\int_0^t \exp(sb)a^T
a \exp(sb^T)\,ds$ and $n=\Rg(q_t)$. Then there is $\theta_t \in\nsing
$ such that
$q_t=t \theta_t I^n_d\theta_t^T$, and we have
%
\begin{equation}\label{RedaWishCan}\WIS_d(x,\alpha,b,a;t)\underset
{\mathit{Law}}{=}\theta_t \WIS_d(\theta_t^{-1} m_txm_t^T(\theta_t^{-1})^T
,\alpha,0,I^n_d;t)\theta_t^T.
\end{equation}
\end{proposition}

This proposition plays a crucial role for the exact simulation of
Wishart processes. Thanks to (\ref{RedaWishCan}), we can sample any
Wishart distribution if we are able to simulate exactly the
distribution $\WIS_d(x ,\alpha,0,I^n_d;t)$ for any $x\in\posm$. In
Section~\ref{sec_exact_wish}, we focus on this and give a way to sample
exactly $\WIS_d(x,\alpha,0,I^n_d;t)$. Let us stress here that we can
compute the matrix $\theta_t$ by using the extended Cholesky
decomposition of $q_t/t$, as it is explained in the proof below.
\begin{pf*}{Proof of Proposition~\ref{Propidloiwis}}
We apply Lemma~\ref{OuterProdDec} to $q_t/t\in\posm$ and consider
$(p,c_n,k_n)$ an extended Cholesky\vspace*{1pt}
decomposition of $q_t/t$. We set $\theta_t= p^{-1} \bigl({{
c_n \atop k_n}\enskip
{0 \atop I_{d-n}}}\bigr)$. Then $\theta_t$ is
invertible and it is easy to check\vadjust{\goodbreak} that $q_t=t \theta_t I^n_d\theta_t^T$.
Now, let us observe that for $v \in\symm$,
\begin{eqnarray*}
&&
\det(I_d-2iq_tv)
= \det\bigl(\theta_t(\theta
_t^{-1}-2it I^n_d\theta_t^T
v)\bigr)=\det(I_d-2it I^n_d\theta_t^T v \theta_t),\\[-25pt]
\end{eqnarray*}
\begin{eqnarray*}
&&
\Tr[iv(I_d-2iq_tv)^{-1}m_txm_t^T]\\
&&\qquad=\Tr[i (\theta_t^{-1})^T \theta
_t^T v(\theta_t \theta_t^{-1} -2it \theta_t
I^n_d \theta_t^T v \theta_t \theta_t^{-1})^{-1}m_txm_t^T]\\
&&\qquad=\Tr[i \theta_t^T v \theta_t (I_d -2i t
I^n_d \theta_t^T v \theta_t )^{-1} \theta_t^{-1} m_txm_t^T (\theta
_t^{-1})^T].
\end{eqnarray*}
Let $X^x_t \sim \WIS_d(x,\alpha,b,a;t)$ and $\tilde{X}^x_t \sim
\WIS_d(x,\alpha,0,I^n_d;t)$. Then, from (\ref{CarWishart}) and (\ref
{CarWishartcan}), we get that
\begin{eqnarray*}
\E[\exp(i\Tr(v X^x_t)) ]&=&\E\bigl[\exp\bigl(i\Tr\bigl(\theta_t^T v \theta_t
\tilde{X}^{\theta_t^{-1}
m_txm_t^T(\theta_t^{-1})^T}_t\bigr)\bigr) \bigr]\\
&=&\E\bigl[\exp\bigl(i\Tr\bigl(v \theta_t \tilde
{X}^{\theta_t^{-1} m_txm_t^T(\theta_t^{-1})^T}_t\theta_t^T\bigr) \bigr)\bigr].
\end{eqnarray*}
\upqed\end{pf*}

Last, let us mention that (\ref{RedaWishCan}) extends a usual
identity between CIR and squared Bessel distribution. It gives when
$d=1$,
\[
\WIS_1(x,\alpha,b,a;t)\underset{\mathrm{Law}}{=}a^2\frac
{e^{2bt}-1}{2bt}\WIS_1\biggl(\frac{2btx}{a^2(1-e^{-2bt})},\alpha,0,1;t\biggr).
\]
In that case, this identity can also be obtained directly from the SDE. Let
$(X^x_t)_{t\ge0} \sim \WIS_1(x,\alpha,b,a)$. Then,\vspace*{1pt}
$Y_t=e^{-2bt}X^x_t/a^2$ is a time-changed
Bessel squared process since
$dY_t=\alpha(e^{-2bt}\,dt)+2\sqrt{Y_t}(e^{-bt}\,dW_t)$. We obtain
$\WIS_1(x,\alpha$, $b,a;t){\,\underset{\mathrm{Law}}{=}\,}a^2e^{2bt}\WIS_1(x/a^2,\alpha
,0,1;\frac{1-e^{-2bt}}{2b})$. A
linear time-change also gives that
$\WIS_1(x,\alpha,0,1;\lambda t)\underset{\mathrm{Law}}{=}\lambda
\WIS_1(x/\lambda,\alpha,0,1;t)$, which leads to (\ref
{RedaWishCan}) by taking $\lambda=(1-e^{-2bt})/(2bt)$.

\section{Exact simulation of Wishart processes}\label{sec_exact_wish}

In this section, we present a new method to simulate exactly a Wishart
process. To the best of our knowledge, this is the first exact
simulation method
for noncentral Wishart distributions that works for any $\alpha\ge
d-1$. Wishart
distributions have been thoroughly studied in statistics when $\alpha
\in
\N$ (which is then called the number of degrees of freedom). Exact simulation
methods have already been proposed in that case. For
instance, Odell and Feiveson~\cite{Odell} and Smith and
Hocking~\cite{Hocking} have proposed an exact simulation method for central
Wishart distributions based on the Bartlett's
decomposition. Gleser~\cite{Gleser} extends it to any (noncentral) Wishart
distribution. Bru~\cite{Bru} also explains, when $\alpha\in\N$, how
Wishart processes
can be obtained as a square of Ornstein--Uhlenbeck processes on matrices.

Here, our method relies on the identity in law (\ref{RedaWishCan}) that
enables us to focus on the case $b=0$, $a=I^n_d$. Then we show a remarkable
splitting of the infinitesimal generator as the sum of commuting operators.
These operators are associated to SDE that can be solved explicitly
on $\posm$,
which enables us to sample any Wishart distribution.

\subsection{\texorpdfstring{A remarkable splitting for $\WIS_d(x,\alpha,0,I^n_d)$}
{A remarkable splitting for WIS d(x,alpha,0,Ind)}}
\label{subsecremsplitting}

The following theorem explains how to split the infinitesimal generator of
$\WIS_d(x,\alpha,0,I^n_d)$ as the sum of commutative infinitesimal
generators. This result is the keystone of the paper and will play a crucial
role in the sequel both for the exact and discretization schemes.
%
\begin{theorem}\label{SplioperCan}
Let $L$ be the generator associated to the Wishart process
$\WIS_d(x,\alpha,0,I^n_d)$ and $L_{e^i_d}$ be the generator associated to
$\WIS_d(x,\alpha,0,e^i_d)$ for $i \in\{1,\ldots,d \}$.
Then, we have
%
\begin{equation}\label{SplitCan}
L= \sum_{i=1}^n L_{e^i_d}
\quad\mbox{and}\quad \forall i,j \in\{
1,\ldots,d \}\qquad L_{e^i_d}L_{e^j_d}= L_{e^j_d}L_{e^i_d}.
\end{equation}
\end{theorem}
\begin{pf}
From\vspace*{1pt} (\ref{GeneratorS}), we easily get that $L= \sum_{i=1}^n
L_{e^i_d}$ since
$I^n_d=\sum_{i=1}^n e^i_d$. The commutativity property comes from
a tedious but simple calculation.
\end{pf}

Beyond the commutativity property, two other features of (\ref
{SplitCan}) are
important to notice:
\begin{itemize}
\item The operators $L_{e^i_d}$ and $L_{e^j_d}$ are the same up to the
exchange of
coordinates $i$ and~$j$.
\item The processes $\WIS_d(x,\alpha,0,e^i_d)$ and $\WIS_d(x,\alpha,0,I^n_d)$
are well defined on $\posm$ under the same hypothesis, namely, $\alpha
\ge d-1$ and $x \in\posm$.
\end{itemize}
This second property makes possible the composition that we explain
now. Let us
consider $t>0$ and $x \in\posm$. We define, iteratively,
\begin{eqnarray*}
X^{1,x}_t &\sim& \WIS_d(x,\alpha,0,e^1_d;t), \\ X^{2,X^{1,x}_t}_t
&\sim&
\WIS_d(X^{1,x}_t,\alpha,0,e^2_d;t), \\
&\vdots& \\
X^{n,\ldots^{X^{1,x}_t}}_t &\sim&
\WIS_d\bigl(X^{{n-1},\ldots^{X^{1,x}_t}}_t,\alpha,0,e^n_d;t\bigr).
\end{eqnarray*}
Thus, conditionally to $X^{i-1,\ldots^{X^{1,x}_t}}_t$, $X^{i,\ldots
^{X^{1,x}_t}}_t$ is sampled according to the distribution at
time $t$ of a Wishart process starting from $X^{i-1,\ldots
^{X^{1,x}_t}}_t$ and
with parameters $(\alpha,0,e^i_d)$. We have the following result.
%
\begin{proposition}\label{Permut}
Let $X^{n,\ldots^{X^{1,x}_t}}_t$ be defined as above. Then
\[
X^{n,\ldots^{X^{1,x}_t}}_t \sim \WIS_d(x,\alpha,0,I^n_d;t).
\]
\end{proposition}

Thanks to this proposition, we can generate a sample according to
$\WIS_d(x,\break\alpha,0,I^n_d;t)$ as soon as we can simulate $\WIS_d(x,\alpha
,0,e^i_d;t)$. These laws are the same as $\WIS_d(x,\alpha,0,e^1_d;t)$, up
to the permutation of the first and $i$th coordinates. In the next
subsection, it is explained how to draw such random variables.

It is really easy to give a formal proof of\vspace*{1pt} Proposition
\ref{Permut}. Let $X^x_t\sim \WIS_d(x,\break\alpha,0,I^n_d;t)$ and $f$ be
a smooth function on $\posm$ such that the series below converge
absolutely. By\vspace*{1pt} iterating It\^o's formula, we have that
$\E[f(X^x_t)]=\sum_{k=0}^{\infty}t^k L^kf(x)/k!$. Similarly, we also
get by using the tower property of the conditional expectation that
%
\begin{eqnarray}\label{Exactcondexp}
\E\bigl[f\bigl(X^{n,\ldots^{X^{1,x}_t}}_t \bigr)\bigr]&=&\E\bigl[\E
\bigl[f\bigl(X^{n,\ldots^{X^{1,x}_t}}_t
\bigr)|X^{{n-1},\ldots^{X^{1,x}_t}}_t\bigr]\bigr]\nonumber\\[-8pt]\\[-8pt]
&=&\sum_{k_n=0}^{+\infty
}\frac{t^{k_n}}{k_n!}\E\bigl[L_{e^n_d}^{k_n}
f\bigl(X^{{n-1},\ldots^{X^{1,x}_t}}_t\bigr) \bigr].\nonumber
\end{eqnarray}
Simply by repeating this argument, we get that
%
\begin{eqnarray}\label{ExactCauchyprod}
\E\bigl[f\bigl(X^{n,\ldots^{X^{1,x}_t}}_t
\bigr)\bigr]&=&\sum_{k_1,\ldots,k_n=0}^{+\infty}\frac{t^{\sum_{i=1}^n k_i}}{k_1!
\cdots k_n!} L_{e^1_d}^{k_1}\cdots L_{e^n_d}^{k_n}f(x) \nonumber\\[-8pt]\\[-8pt]
&=& \sum
_{k=0}^\infty
\frac{t^k}{k!}(L_{e^1_d}+\cdots+L_{e^n_d})^k
f(x)=\E[f(X^x_t)].\nonumber
\end{eqnarray}
To get the second equality, we identify a Cauchy product and use that the
operators $L_{e^1_d},\ldots,L_{e^n_d}$ commute. To make this formal
proof correct, one has
to check that the series are well defined and can be switched with the
expectation. This check is made in the Appendix~\ref{proofPermut} for
our framework and
remains valid as soon as the operator $L_{e^i_d}$ and $L$ are of affine type.

\subsection{\texorpdfstring{Exact simulation for $\WIS_d(x,\alpha,0,e^1_d;t)$}
{Exact simulation for WIS d(x,alpha,0,e1d;t)}}
\label{subsecexactsimL1}

For the
sake of clarity, we start with the case of $d=2$ that avoids
complexities due to matrix decompositions. We deal with the general
case just after.

\subsubsection{The case $d=2$}

We start by writing explicitly the infinitesimal generator $L_{e^1_2}$
of $\WIS_2(x,\alpha,0,e^1_2)$. From (\ref{GeneratorS}), we get
%
\begin{eqnarray}\label{OpL1}
x &\in&\mathcal{S}_2^+(\R),\nonumber\\
L_{e^1_2}f(x) &=& \alpha\partial_{\{
1,1 \}}f(x) + 2 x_{\{1,1 \}} \partial_{\{1
,1 \}}^2f(x)\\
&&{}+ 2x_{\{1,2 \}}
\partial_{\{1,1 \}} \partial_{\{1,2 \}}f(x)
+ \frac{x_{\{2,2 \}}}{2}\partial_{\{1,2 \}}^2f(x).\nonumber
\end{eqnarray}
We now show that this operator is in fact associated to an SDE that can be
explicitly solved. We will denote by $(Z^1_t,t \ge0)$ and $(Z^2_t,t
\ge0)$
two independent standard Brownian motions in $\R$.

When $x_{\{2,2\}}=0$, we also have $x_{\{1,2\}}=0$ since $x$ is
nonnegative. In that case,
%
\begin{eqnarray}\label{Diff2SDE0}
X_0^x&=&x,\qquad d(X_t^x)_{\{1,1 \}} = \alpha \,dt + 2 \sqrt
{(X_t^x)_{\{1,1 \}}} \,dZ_t^1,\nonumber\\[-8pt]\\[-8pt]
d(X_t^x)_{\{1,2 \}}&=&0,\qquad d(X_t^x)_{\{2,2 \}}
= 0\nonumber
\end{eqnarray}
has the infinitesimal generator (\ref{OpL1}), which is one of a
CIR process (or of a squared Bessel process of dimension $\alpha$ to
be more precise). By using an algorithm that samples exactly a
noncentral chi-square distribution (see, e.g.,
Glasserman~\cite{Glasserman}), we can then sample $\WIS_2(x,\alpha,0,e^1_2;t)$
when \mbox{$x_{\{2,2\}}=0$}.

When $x_{\{2,2\}}>0$, it easy to check that the SDE
%
\begin{eqnarray}\label{Diff2SDE}\hspace*{28pt}
d(X_t^x)_{\{1,1 \}} &=& \alpha \,dt + 2 \sqrt
{(X_t^x)_{\{1,1 \}}-\frac{{((X_t^x)_{\{1,2 \}
})}^2}{(X_t^x)_{\{2,2 \}}}} \,dZ_t^1\nonumber\\
&&{} + 2\frac
{(X_t^x)_{\{1,2 \}}}{\sqrt{(X_t^x)_{\{2,2 \}
}}}\,dZ_t^2,\nonumber\\[-8pt]\\[-8pt]
d(X_t^x)_{\{1,2 \}}&=&\sqrt{(X_t^x)_{\{2,2 \}
}}\,dZ_t^2,\nonumber\\
d(X_t^x)_{\{2,2 \}}&=& 0,\nonumber
\end{eqnarray}
starting from $X_0^x=x$, has an infinitesimal generator equal
to $L_{e^1_2}$. To
solve (\ref{Diff2SDE}), we set
%
\begin{eqnarray}\label{UfctdeX2}
(U_t^u)_{\{1,1\}}&=&(X_t^x)_{\{1,1\}}-\frac{{ ( (X_t^x)_{\{1,2\}}
)}^2}{(X_t^x)_{\{2,2\}}},\nonumber\\[-8pt]\\[-8pt]
(U_t^u)_{\{1,2\}}&=&\frac{(X_t^x)_{\{1,2\}}}
{\sqrt{x_{\{2,2\}}}},\qquad
(U_t^u)_{\{2,2\}}=x_{\{2,2\}}.\nonumber
\end{eqnarray}
Here, $u$ stands for the initial condition, that is, $u=U^u_0$.
We get by using It\^{o} calculus that
%
\begin{eqnarray}\label{EDSU2}
d(U_t^u)_{\{1,1 \}} &=& (\alpha-1) \,dt + 2 \sqrt
{(U_t^u)_{\{1,1 \}}}
\,dZ_t^1,\nonumber\\[-8pt]\\[-8pt]
d(U_t^u)_{\{1,2 \}}&=&dZ_t^2\quad
\mbox{and}\quad d(U_t^u)_{\{2,2 \}} = 0.\nonumber
\end{eqnarray}
Therefore, $(U_t^u)_{\{1,2 \}}$ and $(U_t^u)_{\{1,1
\}}$ can be sampled,
respectively, by independent Gaussian and noncentral chi-square
variables. Then, we can get back $ X_t^x$ by inverting (\ref{UfctdeX2}),
%
\begin{eqnarray}\label{XfctdeU2}
(X_t^x)_{\{1,1 \}} &=&
(U_t^u)_{\{1,1 \}}+ (U_t^u)_{\{1,2 \}}^2,\nonumber\\[-8pt]\\[-8pt]
(X_t^x)_{\{1,2 \}} &=& (U_t^u)_{\{1,2 \}}\sqrt
{(U_t^u)_{\{2,2 \}}},\qquad
(X_t^x)_{\{2,2 \}}= (U_t^u)_{\{2,2 \}}.\nonumber
\end{eqnarray}

This result gives an interesting way to figure out the dynamics
associated to the operator $L_{e^1_2}$ by using a change of variable.
It is worth noticing that the CIR process $(U_t^u)_{\{1,1\}}$ is well
defined as soon as its degree $\alpha-1$ is nonnegative, which
coincides with the condition under which the Wishart process
$\WIS_2(x,\alpha,0,e^1_2)$ is well defined. Last, we notice that the
solution of the operator $L_{e^1_2}$ involves\vspace*{1pt} a CIR
process in the diagonal term and a Brownian motion in the nondiagonal
one. A~similar structure holds for larger $d$.

\subsubsection{The general case}

We now present a general way to sample exactly $\WIS_d(x,\alpha
,0,e^1_d;t)$. We
first write explicitly from (\ref{GeneratorS}) the infinitesimal
generator of $\WIS_d(x,\alpha,0,e^1_d)$ for $x \in\posm$
%
\begin{eqnarray}
L_{e^1_d} f(x) &=& \alpha\partial_{\{1,1 \}}f(x) +2
x_{\{1,1 \}}
\partial_{\{1,1 \}}^{2}f(x)\nonumber\\
&&{}+ 2 \mathop{\sum_{1\leq m
\leq d }}_{ m\neq
1} x_{\{1,m \}}\partial_{\{1,m \}}\partial
_{\{1,1 \}}f(x)\\
&&{}+ \frac{1}{2}
\mathop{\sum_{1 \leq m,l\leq d}}_{m \neq1,l \neq1} x_{\{ m,l
\}}
\partial_{\{1,m \}}\partial_{\{1,l \}}f(x).\nonumber
\end{eqnarray}
As for $d=2$, we will construct an SDE that has the same infinitesimal
generator $L_{e^1_d}$ and that can be solved explicitly. To do so
however, we need to use further matrix decomposition results. In the
case $d=2$, we have
already noticed that we choose different SDEs whether $x_{2,2}=0$ or
not. Here, the SDE will depend on the rank of the
submatrix $(x_{i,j})_{2\le i,j
\le d}$, and we set
\[
r=\Rg((x_{i,j})_{2\le i,j \le d}) \in\{0, \ldots, d-1\}.
\]

First, we consider the case where
%
\begin{eqnarray}\label{decompo1}
&\displaystyle \exists c_r\in
\mathcal{G}_r \mbox{ lower triangular},\qquad k_r\in\mathcal
{M}_{d-1-r\times
r}(\mathbb R),&\nonumber\\[-8pt]\\[-8pt]
&\displaystyle (x)_{2\leq i,j \leq d}=\pmatrix{
c_r & 0 \cr
k_r & 0
}\pmatrix{
c_r^T & k_r^T \cr
0 & 0
}=:cc^T.&\nonumber
\end{eqnarray}
With a slight abuse of notation, we consider that this decomposition also
holds when $r=0$ with $c=0$.
When $r=d-1$, $c=c_r$ is simply the usual
Cholesky decomposition of $(x_{i,j})_{2\le i,j \le d}$. As it is
explained in
Corollary~\ref{InvariantTrL1}, we can still get such a decomposition
up to a
permutation of the coordinates $\{2,\ldots,d\}$.
%
\begin{theorem}\label{StructureDyn}
Let us consider $x \in\posm$ such that (\ref{decompo1}) holds.
Let
$(Z_t^l)_{1\leq l\leq r+1}$ be a vector of independent standard Brownian
motions. Then, the following SDE [convention $\sum_{k=1}^r(\cdots
)=0$ when $r=0$]
%
\begin{eqnarray}\label{DiffStruMulti1}
d(X_t^x)_{\{1,1 \}}&=& \alpha \,dt + 2\sqrt
{(X_t^x)_{\{1,1 \}}-\sum_{k=1}^r \Biggl( \sum
_{l=1}^{r}(c_r^{-1})_{k,l}(X_t^x)_{\{1,l+1 \}}
\Biggr)^2}\,dZ_t^1\hspace*{-12pt}\nonumber\\
&&{} +2\sum_{k=1}^r\sum_{l=1}^{r}(c_r^{-1})_{k,l}(X_t^x)_{\{1
,l+1 \}}\,dZ_t^{k+1},\nonumber\\[-8pt]\\[-8pt]
d(X_t^x)_{\{1,i \}}&=& \sum_{k=1}^r c_{i-1,k}\,dZ_t^{k+1},\qquad
i=2,\ldots,d, \nonumber\\
d\bigl((X_t^x)_{\{ l,k \}}\bigr)_{2\leq k,l \leq d}&=&0\nonumber
\end{eqnarray}
has a unique strong solution starting from $x$. It takes
values in $\posm$ and has the infinitesimal generator $L_{e^1_d}$.
Moreover, this solution is
given explicitly by
%
\begin{eqnarray}\label{explicitsolL1}
X_t^x&=&\pmatrix{
1 & 0 &0 \cr
0 & c_r&0\cr
0 & k_r &I_{d-r-1}
}\nonumber\\
&&{}\times\pmatrix{\displaystyle
(U_t^u)_{\{1,1 \}}+\sum_{k=1}^{r} \bigl((U_t^u)_{\{1
,k+1 \}}\bigr)^2 &
\bigl((U_t^u)_{\{1,l+1 \}}\bigr)_{1\leq l \leq r}^T & 0\vspace*{2pt}\cr
\bigl((U_t^u)_{\{1,l+1 \}}\bigr)_{1\leq l \leq r} & I_{r}& 0\vspace*{1pt}\cr
0& 0& 0}\\
&&{}\times\pmatrix{
1 & 0 &0\cr
0 & c_r^T&k_r^T\cr
0 & 0 & I_{d-r-1}},\nonumber
\end{eqnarray}
where
%
\begin{eqnarray}\label{InitialCondition}
d(U_t^u)_{\{1,1 \}} &=& (\alpha-r)\,dt +
2\sqrt{(U_t^u)_{\{1,1 \}}}\,dZ^1_t,\nonumber\\
u_{\{1,1\}}&=&x_{\{1,1 \}}-\sum_{k=1}^r\bigl(u_{\{1,k+1
\}}\bigr)^2
\geq0,\nonumber\\[-8pt]\\[-8pt]
d\bigl((U_t^u)_{\{1,l+1 \}}\bigr)_{1\leq l \leq r} &=&
(dZ^{l+1}_t)_{1\leq
l\leq r},\nonumber\\
\bigl(u_{\{1,l+1 \}}\bigr)_{1\leq l \leq r}&=&c_r^{-1}\bigl(x_{\{1
,l+1 \}}\bigr)_{1
\leq l\leq
r}.\nonumber
\end{eqnarray}
\end{theorem}

Once again, we have made a slight abuse of notation when $r=0$,
and (\ref{explicitsolL1}) should be simply read as
\[
X_t^x=\pmatrix{
(U_t^u)_{\{1,1 \}} & 0 &0\cr
0 & 0&0\cr
0 & 0 & 0}
\]
in that case. In the statement above, it may seem weird that we use for $u$
and $U^u_t$ the same indexation as the one for symmetric matrices while we
only use its first row (or column). The reason is that we can, in fact,
see $X^x_t$ as a function of $U^u_t$ by setting
%
\begin{eqnarray}
(U^u_t)_{\{i,j\}}&=&u_{\{i,j\}}=x_{\{i,j\}} \qquad\mbox{for } i,j
\ge2 \quad\mbox{and}\nonumber\\[-8pt]\\[-8pt]
(U^u_t)_{\{1,i\}}&=&u_{\{1,i\}}=0 \qquad\mbox{for } r+1 \le i \le
d.\nonumber
\end{eqnarray}
Thus, $(c_r,k_r,I_{d-1})$ is an extended Cholesky decomposition of
$((U^u_t)_{i,j})_{2 \le i,j\le d}$ and can be seen as a function of
$U_t^u$. We
get from (\ref{explicitsolL1}) that
%
\begin{eqnarray} \label{XfctdeU}
X^x_t&=&h(U^u_t)\qquad\mbox{with }
h(u)=\sum_{r=0}^{d-1}
\mathbh{1}_{r=\Rg[(u_{i,j})_{2 \le i,j\le d}]}h_r(u) \quad\mbox{and}\\
h_r(u)&=&
\pmatrix{
1 & 0 &0 \cr
0 & c_r(u)&0\cr
0 & k_r(u) &I_{d-r-1}}\nonumber\\
&&{}\times\pmatrix{
\displaystyle u_{\{1,1 \}}+\sum^{r}_{k=1} \bigl(u_{\{1,k+1 \}
}\bigr)^2 & \bigl(u_{\{1,l+1 \}}\bigr)_{1\leq l \leq r}^T & 0\vspace*{2pt}\cr
\bigl(u_{\{1,l+1 \}}\bigr)_{1\leq l \leq r} & I_{r}& 0\vspace*{1pt}\cr
0& 0& 0
}\nonumber\\
&&{}\times
\pmatrix{
1 & 0 &0\cr
0 &c_r(u)^T&k_r(u)^T\cr
0 & 0 & I_{d-r-1}},\nonumber
\end{eqnarray}
where $(c_r(u),k_r(u),I_{d-1})$ is the extended Cholesky decomposition
of
$(u_{i,j})_{2 \le i,j\le d}$ given by some algorithm (e.g., Golub and
Van Loan~\cite{Golub},
Algorithm~4.2.4). Equation (\ref{XfctdeU}) will later play an important
role in analyzing discretization schemes.

The proof of Theorem~\ref{StructureDyn} is given in Appendix \ref
{proofStructureDyn}. It enables us to
simulate exactly the distribution $\WIS_d(x,\alpha,0,e^1_d;t)$ simply by
sampling one noncentral chi-square distribution for $(U_t^u)_{\{1
,1 \}}$
(see Glasserman~\cite{Glasserman}) and $r$ other independent Gaussian random
variables. As in the $d=2$ case, we notice that the condition which
ensures that the
CIR process $((U_t^u)_{\{1,1 \}},t\ge0)$ is well defined for
any $r \in\{0,\ldots,d-1\}$, namely, $\alpha-(d-1)\ge0$, is the
same as the one required for the
definition of $\WIS_d(x,\alpha,0,e^1_d)$.
%
\begin{remark}\label{RemarkOnh} From (\ref{explicitsolL1}), we
easily get by a calculation made
in (\ref{rangmat}) that $\Rg(X_t^x) = \Rg((x_{i,j})_{2 \le i,j \le
d})+\mathbh1_{(U_t^u)_{\{1,1\}}\neq0}$, and
therefore,
\[
\Rg(X_t^x)=\Rg((x_{i,j})_{2 \le i,j \le d})+1\qquad
\mbox{a.s.}
\]
In particular, $X_t^x$ is almost surely positive definite if $x\in
\dpos$.
\end{remark}

Theorem~\ref{StructureDyn} assumes that the initial value $x\in\posm$
satisfies (\ref{decompo1}). Now we explain why it is still possible,
up to a
permutation of the coordinates, to be in such a case. This relies on
the extended
Cholesky decomposition which is stated in Lemma~\ref{OuterProdDec}.
%
\begin{corollary} \label{InvariantTrL1}
Let $x\in\posm$ and $(c_r,k_r,p)$ be an
extended Cholesky decomposition of $(x_{i,j})_{2\le i,j \le d}$
(Lemma~\ref{OuterProdDec}). Then,
$\pi=\bigl({
{1\atop0}\enskip
{0\atop p}}\bigr)
$ is a \mbox{permutation} matrix, $ \WIS_d(x,\alpha,0,e^1_d) \underset
{\mathit{Law}}{=} \pi^T \WIS_d(\pi x
\pi^T,\alpha,0,e^1_d)\pi$ and
$ ((\pi x
\pi^T)_{i,j})_{2\le i,j \le d}= \bigl({
{c_r \atop k_r}\enskip {0 \atop
0}}\bigr)
\bigl({
{c_r^T \atop 0}\enskip {k_r^T\atop 0}}\bigr)$
satisfies (\ref{decompo1}).
\end{corollary}
\begin{pf}
The result comes directly from (\ref{AFFidentite2}), since $\pi
^T=\pi^{-1}$ and\break $\pi e^1_d\pi^T=e^1_d$.
\end{pf}

Therefore, by a combination of Corollary~\ref{InvariantTrL1} and
Theorem~\ref{StructureDyn}, we get a simple way to explicitly
construct a
process that has the infinitesimal generator $L_{e^1_d}$ for any
initial condition
$x\in\posm$. In particular, this enables us to sample exactly the Wishart
distribution $\WIS_d(x,\alpha,0,e^1_d;t)$. Algorithm~\ref{Algo1} below
sums up the
whole procedure.

\begin{algorithm}[t]
\KwIn{$x\in\posm$, $d$, $\alpha\ge d-1$ and $t>0$.}
\KwOut{$X$, sampled according to $\WIS_d(x,\alpha,0,e^1_d;t)$.}
Compute the extended Cholesky decomposition $(p,k_r,c_r)$ of
$(x_{i,j})_{2 \le i,j\le d}$ given by Lemma~\ref{OuterProdDec}, $r\in
\{0,\ldots,d-1\}$ (see Golub and Van Loan~\cite{Golub}
for an
algorithm);

Set $\pi=\bigl({{1\atop0}\enskip {0\atop p}}\bigr)$,
$\tilde{x}=\pi x\pi^T$, $ (u_{\{1,l+1 \}})_{1\leq l \leq
r}=(c_r)^{-1}(\tilde{x}_{\{1,l+1 \}})_{1 \leq l\leq
r}$ and $u_{\{1,1 \}}=\tilde{x}_{\{1,1 \}
}-\sum_{k=1}^r(u_{\{1,k+1 \}})^2
\geq0$;

Sample independently $r$ normal variables $G_2,\ldots,G_{r+1}\sim
\mathcal{N}(0,1)$ and
$(U_t^u)_{\{1,1 \}}$ as a CIR process at time $t$ starting
from $u_{\{1,1 \}}$ solving
$d(U_t^u)_{\{1,1 \}}=(\alpha-r)\,dt+2\sqrt
{(U_t^u)_{\{1,1 \}}}\,dZ^1_t$ (see Glasserman \cite
{Glasserman}).

Set $(U_t^u)_{\{1,l+1 \}}=u_{\{1,l+1
\}} +
\sqrt{t} G_{l+1}$;

\Return{\begin{eqnarray*}
X&=&\pi^T \pmatrix{
1 & 0 &0 \cr
0 & c_r&0\cr
0 & k_r &I_{d-r-1}}\\
&&{}\times\pmatrix{
(U_t^u)_{\{1,1 \}}+\displaystyle \sum_{k=1}^{r}\bigl((U_t^u)_{\{1,k+1
\}}\bigr)^2 &
\bigl((U_t^u)_{\{1,l+1 \}}\bigr)_{1\leq l \leq r}^T & 0\vspace*{2pt}\cr
\bigl((U_t^u)_{\{1,l+1 \}}\bigr)_{1\leq l \leq r} & I_{r}& 0\vspace*{1pt}\cr
0& 0& 0}\\
&&{}\times\pmatrix{
1 & 0 &0\cr
0 & c_r^T&k_r^T\cr
0 & 0 & I_{d-r-1}}
\pi.
\end{eqnarray*}}
\caption{Exact simulation $\WIS_d(x,\alpha,0,e^1_d;t)$}\label{Algo1}
\end{algorithm}

Let us now discuss the complexity of Algorithm~\ref{Algo1}. The number
of operations required by the extended Cholesky decomposition is of
order $O(d^3)$. From a computational point of view, the permutation is
handled directly and does not require any matrix multiplication so that
we can consider w.l.o.g. that $\pi=I_d$. Since $c_r$ is lower
triangular, the calculation of $u_{\{1,i\}}$, $i=1,\ldots,r+1$, only
requires $O(d^2)$ operations. Also, we do not perform in practice the
matrix product (\ref{explicitsolL1}), but only compute the values of
$X_{\{1,i\}}$ for $i=1,\ldots,d$, which also
requires $O(d^2)$ operations. Last, $d$ samples are at most required.
To sum up, it comes out that the complexity of Algorithm~\ref{Algo1} is
of order $O(d^3)$.

\subsection{Exact simulation for Wishart processes}

We have now shown all the mathematical results that enable us to give
an exact
simulation method for general Wishart processes. This is made in two steps.

First, we know how to sample exactly $\WIS_d(x,\alpha,0,e^1_d;t)$
thanks to
Theorem~\ref{StructureDyn} and Corollary~\ref{InvariantTrL1}. By a
simple permutation of the
first and $k$th coordinates, we are then also able to sample according
to $\WIS_d(x,\alpha,0,e^k_d;t)$ for $k \in\{1,\ldots,d \}$. Thus, we
get by
Proposition~\ref{Permut} an exact simulation method to
sample $\WIS_d(x,\alpha,0,I_d^n;t)$.\vadjust{\goodbreak} It is given explicitly in
Algorithm~\ref{Algo2}.
Then we get an exact simulation scheme for $\WIS_d(x,\alpha,b,a;t)$ by
using the
law identity (\ref{RedaWishCan}) (see Algorithm~\ref{Algo3}).

\begin{algorithm}[t]
\KwIn{$x\in\posm$, $n\le d$, $\alpha\ge d-1$ and $t>0$. }
\KwOut{$X$, sampled according to $\WIS_d(x,\alpha,0,I_d^n;t)$ }
$y=x$

\For{$k=1$ \KwTo$n$} {
Set $p_{k,1}=p_{1,k}=p_{i,i}=1$ for $i\notin\{1,k\}$ and $p_{i,j}=0$
otherwise (permutation of the
first and $k$th coordinates).

$y=pYp$ where $Y$ is sampled according
to $\WIS_d(pyp,\alpha,0,e^1_d;t)$ by using Algorithm~\ref{Algo1}.}

\Return{$X=y$.}
\caption{Exact simulation for $\WIS_d(x,\alpha,0,I_d^n;t)$}\label{Algo2}
\end{algorithm}


\begin{algorithm}[t]
\KwIn{$x\in\posm$, $\alpha\ge d-1$, $a,b \in\genm$ and $t>0$.}
\KwOut{$X$, sampled according to\vspace*{1pt} $\WIS_d(x,\alpha,b,a;t)$.}
Calculate $q_t=\int_0^t \exp(sb)a^T a \exp(sb^T)\,ds $ and
$(p,c_n,k_n)$ an extended Cholesky decomposition of $q_t/t$.

Set $\theta_t= p^{-1}
\bigl({
{c_n\atop k_n}\enskip {0
\atop I_{d-n}}}\bigr)$ and $m_t=\exp(tb)$.

\Return{$X=\theta_t Y \theta_t^T$}, where $Y \sim \WIS_d(\theta
_t^{-1} m_txm_t^T(\theta_t^{-1})^T
,\alpha,0,I^n_d;t)$ is sampled by Algorithm~\ref{Algo2}.
\caption{Exact simulation for $\WIS_d(x,\alpha,b,a;t)$}\label{Algo3}
\end{algorithm}

Let us analyze the overall complexity of Algorithm~\ref{Algo3}. Since
it basically runs $n$ times Algorithm~\ref{Algo1}, it requires a
complexity of order $O(nd^3)$ and therefore at most of order $O(d^4)$.
As we have seen, the ``bottleneck'' of Algorithm~\ref{Algo1} is the
extended Cholesky decomposition which is in $O(d^3)$. All the other
steps in Algorithm~\ref{Algo1} require at most $O(d^2)$ operations.
A~natural question for Algorithm~\ref{Algo2} is to wonder if we can reuse
the Cholesky decomposition between the loops instead of calculating it
from scratch. For example, if it were possible to get the Cholesky
decomposition of loop $k+1$ from the one of loop $k$ at a cost
$O(d^2)$, the complexity of Algorithms~\ref{Algo2}
and~\ref{Algo3} would then drop to $O(d^3)$. Despite our
investigations, we have not been able to do so up to now.\vadjust{\goodbreak}
%
\begin{remark}
When $\alpha\ge2d-1$, it is possible to sample $\WIS_d(x,\alpha,0,I_d^n;t)$
in $O(d^3)$ by another mean. If $X^1_t \sim \WIS_d(x,d,0,I_d^n;t)$ and
$X^2_t \sim \WIS_d(0,\alpha-d,0,I_d^n;t)$ are independent, we can check that
$X^1_t+X^2_t \sim \WIS_d(x,\alpha,0,I_d^n;t)$. Then, $X^1_t$ can be
sampled by using Proposition~\ref{squareOU} and $X^2_t$ by using
Bartlett's decomposition (\ref{eqBartlett}) since $X^2_t\underset
{\mathrm{Law}}{=} t \WIS_d(0,\alpha-d,0,I_d^n;1)$ from (\ref{CarWishartcan}).
\end{remark}


\subsection{The Bartlett's decomposition revisited}
\label{subsectionBartlett}

Now we would like to illustrate our exact simulation method on the particular
case $\WIS_d(0,\alpha,0,I_d^n;1)$, which is known in the literature as the
central Wishart distribution. In that case, we can perform explicitly the
composition $X^{n,\ldots^{X^{1,0}_1}}_1$ given by Proposition \ref
{Permut}. We will show by an induction on
$n$ that
%
\begin{equation}\label{eqBartlett}X^{n,\ldots^{X^{1,0}_1}}_1 =
\pmatrix{(L_{i,j})_{1
\le i,j\le n} & 0 \cr 0& 0}
\pmatrix{(L^T_{i,j})_{1 \le i,j\le n} & 0 \cr
0 & 0},
\end{equation}
where\vspace*{2pt} $(L_{i,j})_{1\leq j<i\leq d}$ and $L_{i,i}$
are independent random variables such that $L_{i,j} \sim\mathcal
{N}(0,1)$ and $(L_{i,i})^2 \sim\chi^2(\alpha
-i+1)$ and\vadjust{\goodbreak} $L_{i,j}=0$ for $i<j$. This result is known
as the Bartlett's decomposition and dates back to 1933 (see
Kshirsagar~\cite{Kshirsagar} or Kabe~\cite{Kabe}).

For $n=1$, we know from Theorem~\ref{StructureDyn} that
$(X^{1,0}_1)_{1,1} \sim \chi^2(\alpha)$ since\break
$d(X^{1,0}_t)_{1,1}=\alpha \,dt + 2 \sqrt{(X^{1,0}_t)_{1,1}}\,dZ^1_t$
with $(X^{1,0}_0)_{1,1}=0$, and all the other elements are equal
to~$0$. Let us assume now that the induction hypothesis is satisfied
for $n-1$. Then, we can apply once again Theorem~\ref{StructureDyn} (up
to the permutation of the first and $n$th coordinates). We have
$\Rg(X^{n-1,\ldots^{X^{1,0}_1}}_1)=n-1$, a.s., and the
Cholesky\vspace*{1pt} decomposition is directly given by $(L_{i,j})_{1
\le i,j \le n-1}$. Then,\vspace*{2pt} we get from (\ref{explicitsolL1}) that there
are independent variables $L_{n,n}^2 \sim\chi^2(\alpha -n+1)$ and
$L_{n,i} \sim\mathcal{N}(0,1)$ for $i \in\{1,\ldots,n-1 \} $ such that
\begin{eqnarray*}
X^{n,\ldots^{X^{1,0}_1}}_1&=&\pmatrix{
(L_{i,j})_{1 \le i,j \le n-1} & 0 &0 \cr
0 & 1 &0\cr
0 & 0 &I_{d-n}}\\
&&{}\times\pmatrix{
I_{n-1} & (L_{n,i})_{1 \le i \le n-1} & 0\cr
(L_{n,i})_{1 \le i \le n-1}^T & \displaystyle \sum_{i=1}^{n} L_{n,i}^2 & 0\cr
0& 0& 0}\\
&&{}\times\pmatrix{
(L_{i,j})^T_{1 \le i,j \le n-1} & 0 &0 \vspace*{2pt}\cr
0 & 1 &0\vspace*{2pt}\cr
0 & 0 &I_{d-n}}.
\end{eqnarray*}
Since
\begin{eqnarray*}
\pmatrix{
I_{n-1} & (L_{n,i})_{1 \le i \le n-1} & 0\vspace*{2pt}\cr
(L_{n,i})_{1 \le i \le n-1}^T &
\displaystyle \sum_{i=1}^{n} L_{n,i}^2 & 0\vspace*{2pt}\cr
0& 0& 0}
&=&
\pmatrix{
I_{n-1} & 0 & 0\vspace*{2pt}\cr
(L_{n,i})_{1 \le i \le n-1}^T & L_{n,n} & 0\vspace*{2pt}\cr
0& 0& 0}\\
&&{}\times\pmatrix{
I_{n-1} & (L_{n,i})_{1 \le i \le n-1} & 0\cr
0 & L_{n,n} & 0\cr
0& 0& 0},
\end{eqnarray*}
we conclude by induction on $n$.



\section{High-order discretization schemes for Wishart and
semidefinite positive affine processes}\label{sec_high}

In this section, we switch from exact sampling to approximate schemes.
First, this
will enable us to simulate not only Wishart processes, but also general affine
processes. More importantly, the discretization schemes that we
introduce are in practice
faster than the exact simulation scheme, especially if one has to sample
entire paths. This will be illustrated in Section~\ref{SecSimul}.

When dealing with discretization schemes,
splitting operators is a powerful technique to construct schemes for SDEs
from other schemes obtained on simpler SDEs. This idea of splitting originates
from the seminal work of Strang~\cite{Strang} in the field of ODEs. As pointed
out by Ninomiya and Victoir~\cite{NV} or Alfonsi~\cite{Alfonsi}, it
is rather easy
to analyze the weak error (i.e., the error made on marginal
distributions) of
schemes obtained by splitting. Indeed, this can be done simply by using
the same arguments as Talay and Tubaro~\cite{Talay} for the Euler--Maruyama
scheme. Nonetheless, when we use the splitting technique for SDEs that are
defined on a given domain [$\posm$ in our case], one has to be careful
that the discretization
scheme remains in it. For example, in the case of the CIR diffusion
(i.e., $d=1$),
general splitting methods such as Ninomiya and Victoir~\cite{NV} fail to
preserve the domain $\R^+$. It is, in fact, only well defined for
$\alpha\ge1$,
while the CIR process exists for any $\alpha\ge0$ (see
Alfonsi~\cite{Alfonsi}). Of course, the same remark holds for Wishart and
affine processes. This is why we will use the ad hoc
splitting (\ref{SplioperCan}) instead of general splitting methods, which
enables us to get schemes that preserve $\posm$ and are defined
without any restriction on the
parameters.

The analysis of the strong error of our schemes is beyond the scope of this
paper. In fact, behind the term ``strong error'' we have in mind here
two different
things. First, it can be the error made on pathwise expectations
between the
discretization scheme and the exact scheme. This kind of error is illustrated
numerically in the next section (Figure~\ref{vfWishartSup}) and seems
to be of
the same order as the weak error, even though we are not at all able to
mathematically show this result. Second, ``strong error'' can also mean the
pathwise error between the discretization scheme and the exact solution
for a
given Brownian motion $(W_t,t \ge0)$. The rate of convergence for this kind
of error has been analyzed for the CIR in Alfonsi~\cite{Alfonsi2} and is
really low. This is mainly due to the fact that the square root is not
Lipschitz near $0$. Fortunately, discretization schemes are mostly used
to compute
expectations with a Monte Carlo algorithm. In this context, pathwise error
is not so relevant.\looseness=-1

To our knowledge, there are very few papers in the literature that deal
with discretization schemes for Wishart processes. Recently,
Benabid, Bensusan and Karoui~\cite{Harry} have proposed a Monte Carlo
method to calculate
expectations on Wishart processes which is based on a Girsanov change of
probability. Gauthier and Possamai~\cite{GP} introduce a moment-matching
scheme for Wishart processes. Both methods are well defined under some
restrictions on the parameters, and there is no theoretical result on their
accuracy. Currently, Teichmann~\cite{semBachelier} is working on dedicated
schemes for general affine processes by approximating their characteristic
functions.

This section is structured as follows. First, we recall basic results on
the splitting technique to get discretization schemes for SDEs. We will
take the same
framework as Alfonsi~\cite{Alfonsi} since it is somehow designed for affine
processes. Then we will explain how to get high-order schemes for
$\WIS_d(x,\alpha,0,e^1_d)$ from the construction given by
Theorem~\ref{StructureDyn}.\vadjust{\goodbreak} The remarkable
splitting (\ref{SplitCan}) will then enable us to get high-order schemes
for $\WIS_d(x,\alpha,\break 0,I^n_d)$. From this result, we will be able to get
a second-order scheme for any semidefinite positive affine processes
and a
third-order scheme for Wishart processes.

\subsection{Weak error analysis and splitting methods}

Let us start with some notation. We consider a
time horizon $T>0$ and the regular time grid defined by $t_i^N=iT/N$,
$i=0,\ldots,N$. When considering a Markovian process on a domain~$\D$,
a discretization scheme
is a way to sample the value at a given time step $t>0$, starting from
the current
value $x\in\D$. It is thus described by a probability measure
$\hat{p}_x(t)(dz)$ on $\D$, and we denote by $\hat{X}^x_t$ a random
variable that follows
this law. Then the full discretization on the
regular time grid associated to this scheme from $x \in\D$ is simply a
sequence $(\xcn_{t^N_i}, 0 \le i \le N)$ of random variables
such that:
\begin{itemize}
\item$\xcn_{t^N_0}=x$,
\item the law of $\xcn_{t^N_{i+1}}$ is sampled according to $\hat
{p}_{\xcn_{t^N_i}} (T/N)(dz)$
independently from the previous samples, that is,
$\E[f(\xcn_{t^N_{i+1}})|(\xcn_{t^N_j}, 0 \le j \le i)]=\break\int_{\D}
f(z)\hat{p}_{\xcn_{t^N_i}} (T/ N)(dz)$ for any bounded measurable
function $f\dvtx\D\rightarrow\R$.
\end{itemize}

Now we focus on the analysis of the weak error
$\E[f(X^x_T)]-\E[f(\xcn_{t^N_N})]$. There is a huge literature on
this topic.
Talay and Tubaro~\cite{Talay} have obtained an expansion error for
Euler--Maruyama and Milstein schemes. This error has also been studied
on other
schemes: we cite the articles of Kusuoka~\cite{Kusuoka},
Lyons and Victoir~\cite{LyonsVictoir}, Ninomiya and Victoir~\cite{NV},
and Ninomiya and Ninomiya~\cite{Ninomiya2}, to mention a few. However,
to our
knowledge, most of these papers make regularity assumptions on the SDE
coefficients that are not satisfied by affine diffusions. Typically, they
assume that these coefficients are $\mathcal{C}^\infty$ with bounded
derivatives. This is not satisfied by general affine diffusions because
of the square root
diffusion term. For this reason, Alfonsi~\cite{Alfonsi} introduced a
framework that allows us to rigorously analyze the weak error for affine
diffusions. In this paper, we will naturally work under this
framework. Unfortunately, this requires us to introduce some
definitions, and we
present here only the main ones.

We consider a domain $\D\subset\R^\zeta$, $\zeta\in\N^*$, and $L$ an
operator associated to an SDE defined on $\D$. Mainly (but not
only), we consider in this paper $\mathbb D=\posm\subset\symm\simeq
\R^{d(d+1)/2}$.
For $\gamma=(\gamma_1, \ldots,\gamma_\zeta) \in\N^{\zeta}$, we
define
$\partial_{\gamma}=\partial_{1}^{\gamma_1},\ldots,\partial_{\zeta
}^{\gamma_{\zeta}}$
and $|\gamma|=\sum_{i=1}^{\zeta}\gamma_i$ and set
\begin{eqnarray*}
\Cpol&=&\{ f \in\mathcal C^{\infty}(\mathbb D,\mathbb R),
\forall\gamma\in\mathbb N^\zeta, \exists C_{\gamma} >0,
e_{\gamma} \in\mathbb N^*,\\
&&\hspace*{42pt} \forall x \in\mathbb D, |\partial
_{\gamma}f(x)| \leq C_{\gamma}(1 + \|x\|^{e_{\gamma}})\},
\end{eqnarray*}
where \mbox{$\|\cdot\|$} is a norm on $\R^\zeta$. We say that
$(C_\gamma,e_\gamma)_{\gamma\in\N^\zeta}$ is a \textit{good
sequence} for $f\in
\Cpol$ if one has $|\partial_{\gamma}f(x)| \leq
C_{\gamma}(1 + \|x\|^{e_{\gamma}}) $. The operator $L$ is said to
satisfy the
\textit{required assumptions} if it can be written as $L=\sum_{0<|\gamma
|\le
2}a_\gamma(x) \partial_\gamma$, with $a_\gamma\in\Cpol$. This property
holds for affine diffusions since any $a_\gamma$ is an affine function.
We will say that $\hat{X}^x_t$ is a
\textit{potential weak $\nu$th-order scheme for the operator $L$}
if for any function $f
\in\Cpol$ with a good sequence $(C_\gamma,e_\gamma)_{\gamma\in\N
^\zeta} $, there
exist\vspace*{2pt} positive constants $C, E$ and $\eta$ depending only on
$(C_\gamma,e_\gamma)_{\gamma\in\N^\zeta}$ such that
%
\begin{eqnarray}\label{defpotential}
&&\forall t \in(0, \eta)\nonumber\\[-8pt]\\[-8pt]
&&\qquad\Biggl| \E[f(\hat
{X}^x_t)]-\Biggl[ f(x)+ \sum_{k=1}^\nu
\frac{1}{k!}t^k L^k
f(x)\Biggr] \Biggr| \le Ct^{\nu+1}(1+\| x\|^{E}).\nonumber
\end{eqnarray}

Roughly speaking, this is the main assumption that a discretization scheme
should satisfy to get a weak error of order $\nu$. This is precised by the
following theorem given in~\cite{Alfonsi} that relies on the idea
developed by
Talay and Tubaro~\cite{Talay} for the Euler--Maruyama scheme.
%
\begin{theorem}\label{Thmweak} Let $L$ be an operator satisfying the
required assumptions on~$\D$. We assume that:
\begin{longlist}[(1)]
\item[(1)] $\hat{X}^x_t$ is a
potential weak $\nu$th-order scheme for $L$, and the scheme has
uniformly bounded moments, that is,
%
\begin{equation} \label{momentsbornes}
\exists n_0 \in\N^*, \forall q \in\mathbb{N}^*\qquad \underset{N \ge
n_0, 0\le i \le N}{\sup} \E[\|
\xcn_{t^N_i}\|^q]< \infty;
\end{equation}

\item[(2)] $f\dvtx\D\rightarrow\R$ is a function such that $u(t,x)=\E
[f(X^x_{T-t})]$
is defined on $[0,T] \times\D$, $\mathcal{C}^\infty$, solves
$\forall t \in
[0,T], \forall x \in\D, \partial_t u(t,x)=-Lu(t,x)$ and
satisfies
%
\begin{eqnarray}\label{contderivees}
&&
\forall l \in\N, \gamma\in\N^\zeta,\exists C_{l,\gamma
},e_{l,\gamma}>0, \forall x \in
\D, t\in[0,T]\nonumber\\[-8pt]\\[-8pt]
&&\qquad|\partial_t^l \partial_\gamma u (t,x)| \le
C_{l,\gamma}(1+\|x\|^{e_{l,\gamma}}).\nonumber
\end{eqnarray}
\end{longlist}
Then, there is $K>0$, $N_0 \in\N$, such that $|\E[f(\xcn
_{t^N_N})]-\E[f(X^{x}_T)]|
\le K /N^\nu$ for $N \ge N_0$.
\end{theorem}

It is really important to notice that only assumption (1) depends on the
discretization scheme. Assumption (2) just depends on the underlying
diffusion. Since we only have a hold over the discretization scheme, this
means from a numerical point of view that we mainly have to focus on
assumption (1) to construct an accurate scheme. From a mathematical point
of view, the regularity of the Cauchy problem which is required by
assumption (2)
is a tough problem that is interesting in its own. General results have been
obtained in Talay and Tubaro~\cite{Talay} when $b$ and $\sigma$ are
$\mathcal{C}^\infty$ with bounded derivatives. In the case of Wishart
processes, we are able to get (\ref{contderivees}) when
$f \in\Cpolde{\symm}$.\vadjust{\goodbreak}
%
\begin{proposition}\label{Flow}
Let $(X_t^x)_{t \geq0} \sim \WIS_d(x,\alpha,b,a)$ and $L$ the associated
generator. Let $f\in\Cpolde{\symm}$,
$x\in\posm$ and $T>0$. Then, $\tilde{u} (t,x)=\mathbb E[f(X_t^x)]$ is
$\mathcal{C}^{\infty}$ on $[0,T] \times\posm$, solves $\partial
_t\tilde{u} (t,x)=L\tilde{u} (t,x)$ and its derivatives satisfy
%
\begin{eqnarray}\label{derflowform}
&&\forall l \in\mathbb N,\forall n \in\mathbb{N}^{{d(d+1)}/{2}},
\exists C_{l,n},e_{l,n}>0, \forall x \in\posm, \forall
t\in[0,T]\nonumber\\[-8pt]\\[-8pt]
&&\qquad \biggl|\partial_t^l\prod_{1\leq i\leq j \leq
d}\partial_{\{ i,j \}}^{n_{\{i,j\}}} \tilde{u} (t,x)
\biggr|\leq
C_{l,n}(1+\|x\|^{e_{l,n}}).\nonumber
\end{eqnarray}
\end{proposition}

The proof of this result is made in Appendix~\ref{AppFlow}. It relies
on the
explicit formula of the characteristic function (\ref{CarWishart})
and, more
exactly, on the property stated in Lemma~\ref{Remarkable}.
Unfortunately, we
have not been able to show an analogous result for general affine processes
$\AFF_d(x,\bar{\alpha},B,a)$. We deem that (\ref{derflowform}) also
holds in that case, but this remains an open question.

Let us now turn to assumption (1) of Theorem~\ref{Thmweak}. Usually, the
boundedness of moments is not a big issue and requires, in general, tedious
calculations. This basically holds when the drift and the diffusion
coefficients have a sublinear growth, which is the case here.
Conversely, it is
much more difficult to find a scheme which is a potential $\nu$-order scheme
and stays at the same time in the domain $\posm$. For example, the
Euler--Maruyama scheme is, generally speaking, a potential first-order
scheme. However, it does not stay in $\posm$ even for the CIR case
($d=1$). Still, for the CIR process, higher-order schemes such as
Ninomiya and
Victoir~\cite{NV} or Ninomiya and Ninomiya~\cite{Ninomiya2} stay
in $\R^+$
only under additional restrictions on the parameters. To solve this problem
and get high-order schemes that remain in $\posm$, we will construct
ad hoc discretization schemes by taking advantage of the remarkable
splitting (\ref{SplitCan}). In
fact, the property of being a potential $\nu$th-order schemes is
really easy
to handle by scheme composition, especially when $\nu=2$. This kind of
result dates back to
Strang~\cite{Strang} in the field of ODEs. In our framework, we recall
a result
that is stated in~\cite{Alfonsi}.
%
\begin{proposition}\label{composchemas}
Let $L_1, L_2$ be the generators of SDEs defined on $\D$
that satisfy the required assumption on $\D$. Let
$\hat{X}^{1,x}_t$ and $\hat{X}^{2,x}_t$ denote, respectively, two
potential weak $\nu$th-order
schemes on $\D$ for $L_1$ and $L_2$.
\begin{longlist}[(2)]
\item[(1)] If $L_1L_2=L_2L_1$, $\hat{X}^{2,\hat{X}^{1,x}_t}_t$ is a potential
weak $\nu$th-order discretization scheme for $L_1+L_2$.
\item[(2)] Let $B$ be an independent Bernoulli variable of
parameter $1/2$. If $\nu\ge2$,
\[
\mbox{\textup{(a)}}\quad B \hat{X}^{2,\hat{X}^{1,x}_t}_t+(1-B)\hat{X}^{1,\hat{X}^{2,x}_t}_t
\quad\mbox{and}\quad \mbox{\textup{(b)}}\quad \hat{X}^{2,\hat{X}^{1,\hat
{X}^{2,x}_{t/2}}_t}_{t/2}
\]

are potential weak second-order
schemes for $L_1+L_2$.
\end{longlist}
\end{proposition}

Let us explain the notation above. The composition $\hat{X}^{2,\hat
{X}^{1,x}_{t_1}}_{t_2}$ means that we first use
the scheme 1 with time step $t_1$ and then, conditionally to
$\hat{X}^{1,x}_{t_1}$, we sample the scheme 2 with initial value
$\hat{X}^{1,x}_{t_1}$ and time step $t_2$. To be explicit, it has the law
$\int_{\D} \hat{p}^2_y(t_2)(dz) \hat{p}^1_x(t_1)(dy)$, where
$\hat{p}^{i}_x(t_i)(dz)$ denotes the law of $\hat{X}^{i,x}_{t_i}, i=1,2$.

\subsection{High-order schemes for Wishart processes}

In this paragraph, we will give a way to get weak $\nu$th-order schemes
for any Wishart processes. The construction is similar to
the one used for the exact scheme. First, we obtain a $\nu$th-order
scheme for
$\WIS_d(x,\alpha,0,e^1_d)$. Then, we get a $\nu$th-order scheme for
$\WIS_d(x,\alpha,0,I_d^n)$ from the splitting (\ref{SplitCan}) and
Proposition~\ref{composchemas}. Last, we use the identity in
law (\ref{RedaWishCan}) to get a weak $\nu$th-order scheme for any
Wishart process.\looseness=-1

Let us start then by introducing a potential weak $\nu$th-order scheme
for $\WIS_d(x,\alpha,0,e^1_d)$. Roughly speaking, we
obtain this scheme from the exact scheme given by
Theorem~\ref{StructureDyn} and Corollary~\ref{InvariantTrL1} by
replacing the
Gaussian random variables with moment matching
variables and the exact CIR distribution with a sample according to a potential
weak $\nu$th-order scheme for the CIR.
%
\begin{theorem}\label{WeakPotenL1}
Let $x\in\posm$ and $(c_r,k_r,p)$ be an extended Cholesky
decomposition of
$(x_{i,j})_{2 \le i,j \le d}$. We set $\pi=
\bigl({1 \atop 0}\enskip{0 \atop p}\bigr)
$
and $\tilde{x}=\pi x \pi^T$, so that
$(\tilde{x}_{i,j})_{2 \le i,j \le d}=
\bigl({c_r \atop k_r}\enskip{0 \atop 0}\bigr)
\bigl({c_r^T \atop 0}\enskip{k_r^T \atop 0}\bigr)
$. As in Theorem~\ref{StructureDyn}, we have
\[
u_{\{1,1 \}}=\tilde{x}_{\{1,1 \}}-\sum
_{k=1}^r\bigl(u_{\{1,k+1 \}}\bigr)^2
\geq0,
\]
where
\[
\bigl(u_{\{1,l+1 \}}\bigr)_{1\leq l \leq
r}=c_r^{-1}\bigl(\tilde{x}_{\{1,l+1 \}}\bigr)_{1 \leq l\leq
r},
\]
and we set $u_{\{1,i\}}=0 \mbox{ if } r+2 \le i\le d$ and $ u_{\{i,j\}
}=\tilde{x}_{\{i,j\}} \mbox{ if } i,j \ge2$.
Let $(\hat{G}^i)_{1\leq i \leq r }$ be a sequence of independent real
variables with finite moments of any order such that
\[
\forall i \in\{1,\ldots,r\}, \forall k \leq
2\nu+1\qquad \mathbb E[(\hat{G}^i)^k]=\mathbb E[G^k]\qquad\mbox{where }G
\sim
\mathcal{N}(0,1).
\]
Let $h_r$ be the function defined by (\ref{XfctdeU}). Let $(\hat
{U}_t^u)_{\{1,1 \}}$ be sampled independently according to a
potential weak $\nu$th-order scheme
for the CIR process
$d(U_t^u)_{\{1,1\}}=(\alpha-r)\,dt+2\sqrt{(U_t^u)_{\{1,1\}}}\,dZ^1_t$ starting
from $u_{\{1,1\}}$. 
We set
\begin{eqnarray*}
(\hat{U}_t^u)_{\{1,i\}}&=&u_{\{1,i\}}+\sqrt{t} \hat{G}^i,\qquad 2
\le i\le r+1,\\
(\hat{U}_t^u)_{\{1,i\}}&=&0,\qquad
r+2 \le i\le d,\\
(\hat{U}_t^u)_{\{i,j\}}&=&u_{\{i,j\}} \qquad\mbox{if } i,j \ge2.
\end{eqnarray*}
Then, the scheme $\hat{X}^x_t=\pi^T h_r(\hat{U}_t^u) \pi$ is a potential
$\nu$th-order scheme for $L_{e^1_d}$ and takes values in $\posm$.
\end{theorem}

Let us give the idea of the proof. By construction, we have $\hat
{X}^x_t \in\posm$ since an analogous formula to (\ref{explicitsolL1})
holds for $\hat {X}^x_t$. The tedious part is to check that it is a
potential $\nu$th-order scheme. We know from Theorem
\ref{StructureDyn}, equation\vspace*{1pt} (\ref{XfctdeU}) and
Corollary~\ref{InvariantTrL1} that we have $X_t^x=\pi^T h_r(U^u_t) \pi
$. It is easy to check that $\hat{U}_t^u$ is a potential $\nu$th-order
scheme for the operator associated to the diffusion $U^u_t$. Let us
suppose for a while that $h_r(u)\in\Cpolde{\symm}$. Then, $u\mapsto
f(\pi^T h_r(u) \pi)$ is also in $\Cpolde{\symm}$, and for any $f\in
\Cpolde{\posm}$, there are constants\vspace*{2pt} $C,E,\eta>0$ depending only on a
good sequence of $f$ such that
\[
|\E[f(\pi^T h_r(\hat{U}_t^u) \pi)] -\E[f(X^x_t)] |\le
Ct^{\nu+1}(1+\|x\|^E),
\]
which basically gives the desired result. Unfortunately, $h_r$ is not in
$\Cpolde{\symm}$. In fact, $h_r$ is only smooth with respect to the
coefficients
of the first row and the first columns. However, these coefficients are
also the only
ones that are changed by $\hat{U}_t^u$ [the submatrix $((\hat
{U}_t^u)_{i,j})_{2\le
i,j\le d}=(u_{i,j})_{2\le i,j\le d}$ is constant], and it comes out that
the regularity on $h_r$ is sufficient to get a potential $\nu$th-order
scheme\vspace*{1pt}
for~$L_{e^1_d}$. This is shown rigorously in the preprint version of
this paper at the cost of additional
technical definitions such as the ``immersion property'' that we do not
reproduce here.

Now we briefly comment on the practical implementation of Theorem \ref
{WeakPotenL1}. Second and third-order schemes for the CIR process
satisfying can be found in Alfonsi~\cite{Alfonsi}. We can\vspace*{1pt}
therefore get second (resp., third) order schemes for $L_{e^1_d}$ by
taking any variables that matches the five (resp.,\vspace*{1pt} the seven) first
moments of $\mathcal{N}(0,1)$. This can be obtained by taking
%
\begin{equation}
\label{VariableY}
\Px\bigl(\hat{G}^i=\sqrt{3}\bigr) = \mathbb P\bigl(\hat{G}^i=-\sqrt{3}\bigr)
=\tfrac
{1}{6} \quad\mbox{and}\quad
\mathbb P(\hat{G}^i=0)=\tfrac{2}{3},
\end{equation}
respectively,
\begin{eqnarray}
\label{VariableY5}
\mathbb P\bigl(\hat{G}^i=\varepsilon
\sqrt{3+\sqrt{6}}\bigr)&=&\frac{\sqrt{6}-2}{4\sqrt{6}},\nonumber\\[-8pt]\\[-8pt]
\mathbb
P \bigl(\hat{G}^i=\varepsilon
\sqrt{3-\sqrt{6}} \bigr)&=&\frac{1}{2}-\frac{\sqrt{6}-2}{4\sqrt
{6}}, \qquad \varepsilon
\in\{-1,1\}.\nonumber
\end{eqnarray}

We focus now on the construction of a potential weak $\nu$th-order scheme
for $\WIS_d(x,\alpha,0,I^n_d)$. Let $\hat{X}^{1,x}_t$ denote a
potential weak
$\nu$th-order scheme for $\WIS_d(x,\alpha,0,e^1_d)$. For $i\in\{
2,\ldots,d\}$,
$\WIS_d(x,\alpha,0,e^i_d)$ and $\WIS_d(x,\alpha,0,e^1_d)$ have the same
law up
to the permutation of the first and $i$th
coordinate. Let $\pi^{1 \leftrightarrow i}$ denote the associated permutation
matrix. Then, we easily get that
\[
\hat{X}^{i,x}_t=\pi^{1 \leftrightarrow i}\hat{X}^{1,\pi^{1
\leftrightarrow i}x\pi^{1 \leftrightarrow i}}_t\pi^{1 \leftrightarrow
i}
\]
is a potential $\nu$th-order scheme for $\WIS_d(x,\alpha,0,e^i_d)$.
Last, we get from Theorem~\ref{SplioperCan} and the point 1 of
Proposition~\ref{composchemas} that
%
\begin{equation}\label{potnucan}\qquad
\hat{X}^{n,\ldots^{\hat{X}^{1,x}_t}}_t \mbox{ is a potential weak
$\nu$th-order
scheme for $\WIS_d(x,\alpha,0,I_d^n)$}.
\end{equation}

Now we are in position to construct a scheme for any
Wishart process $\WIS_d(x,\alpha,b,a)$ thanks to the
identity (\ref{RedaWishCan}). Let $\theta_t\in\nsing$ be such as in
Proposition~\ref{Propidloiwis} and $\hat{Y}^y_t$ denote a potential
weak $\nu$th-order scheme for $\WIS_d(y,\alpha,0,I^n_d)$. Then we
consider the
following scheme for $\WIS_d(x,\alpha,b,a)$:
%
\begin{equation}\label{schordrenu}\hat{X}^x_t=\theta_t \hat
{Y}^{\theta_t^{-1}
m_txm_t^T(\theta_t^{-1})^T}_t\theta_t^T.
\end{equation}
Unfortunately, we need to make some technical
restrictions on $a$ and $b$ [namely, $a \in\nsing$ or $ba^T a=a^T
ab$] to show that we get like this a potential
$\nu$th-order scheme. We, however, believe that this is rather due to
our analysis of
the error and that the scheme converges as well without this
restriction. In addition, we
mention that we give in the next section a second-order
scheme based on Proposition~\ref{Canonform} for which we can make our error
analysis for any parameters.
%
\begin{proposition}\label{propweaknu}
Let $t>0$, $a,b\in\genm$ and $\alpha\ge d-1$.
Let $m_t=\exp(tb)$, $q_t=\int_0^t \exp(sb)a^T a \exp(sb^T)\,ds$ and
$n=\Rg(a^Ta)$.
We assume that either $a \in\nsing$ or $b$ and $a^T a$ commute. We
define:
\begin{itemize}
\item if $n=d$, $\theta_t$ as the (usual) Cholesky decomposition
of $q_t/t$,
\item if $n<d$, $\theta_t=\sqrt{\frac{1}{t}\int_0^t \exp(sb) \exp
(sb^T) \,ds } p^{-1}
\bigl({c_n \atop k_n}\enskip{0 \atop I_{d-n}}\bigr)
$ where $(c_n,k_n,p )$ is the extended Cholesky
decomposition of $a^Ta$ otherwise.
\end{itemize}
In both cases, $\theta_t \in\nsing$ and the scheme (\ref{schordrenu})
is a potential weak $\nu$th-order scheme for $\WIS_d(x,\alpha,b,a)$.
\end{proposition}

The proof of Proposition~\ref{propweaknu} is left in
Appendix~\ref{Apppropweaknu}. From Theorem~\ref{Thmweak}, we
finally get
the following result by using Propositions~\ref{Flow},~\ref{propweaknu}.
%
\begin{theorem}\label{thirdordertheorem}
Let $(X^{x}_t)_{t\ge0} \sim \WIS_d(x,\alpha,b,a)$ such that either
$a\in
\nsing$ or $a^Tab=ba^Ta$ and $f \in\Cpolde{\symm}$. Let $(\xcn
_{t^N_i}, 0 \le i \le
N)$ be sampled with the scheme defined by Proposition \ref
{propweaknu} and Theorem~\ref{WeakPotenL1}
with the third-order scheme for the CIR given in~\cite{Alfonsi}. Then,
\[
\exists C,N_0>0, \forall N \ge N_0\qquad
|\E[f(\xcn_{t^N_N})]-\E[f(X^{x}_T)]|\le C/N^3.
\]
\end{theorem}

\subsection{Second-order schemes for affine diffusions on $\posm$}\label{subsecsecondordaff}

In this part, we present a potential second-order scheme for
$\AFF_d(x,\overline{\alpha},B,a)$. Thanks to
Proposition~\ref{Canonform}, there is $u \in\nsing$ and a diagonal matrix\vadjust{\goodbreak}
$\overline{\delta}$ such that $ \overline{\alpha}=u^T \overline
{\delta} u$,
$a^T a = u^T I^n_d u$ and we have
\begin{eqnarray}
\bigl(u^T Y^{(u^{-1})^T x u^{-1} }_tu\bigr)_{t
\geq0} \sim \AFF_d(x,\overline{\alpha},B,a)\nonumber\\
&&\eqntext{\mbox{where }
(Y^y_t)_{t\ge0}\sim \AFF_d(y,\overline{\delta},B_u,I^n_d).}
\end{eqnarray}
Using the same linear transformation, we can get a potential $\nu$th-order
scheme for $\AFF_d(x,\overline{\alpha},B,a)$ from a potential $\nu$th-order
scheme for $\AFF_d(y,\overline{\delta},\break B_u,I^n_d)$ as stated below.
%
\begin{lemma}\label{lempotaff} If $\hat{Y}^y_t$ is a potential $\nu
$th-order scheme for
$\AFF_d(y,\overline{\delta},B_u,I^n_d)$, then $u^T \hat
{Y}^{(u^{-1})^T x u^{-1}
}_tu$ is a potential $\nu$th-order scheme for $\AFF_d(x,\overline
{\alpha},B,a)$.
\end{lemma}
\begin{pf}
Let $f \in\Cpolde{\posm}$. We then have $x\mapsto f(u^Tx u) \in
\Cpolde{\posm}$. Since $u$ is fixed, there are constants $C,\eta, E$
depending only on a
good sequence of $f$ such that for $t\in(0, \eta)$, $|\E[f( u^T \hat
{Y}^{(u^{-1})^T x u^{-1}
}_tu)]-\E[f(X^x_t)]|= |\E[f( u^T \hat{Y}^{(u^{-1})^T x u^{-1}
}_tu)]-\E[f(u^T Y^{(u^{-1})^T x u^{-1} }_tu)]| \le Ct^{\nu+1}(1+ \|
(u^{-1})^T\vspace*{1pt} x\times\break u^{-1}\|^E)\le C't^{\nu+1}(1+ \|x\|^E)$, for some
constant $C'>C$.
\end{pf}

We now focus on finding a scheme for $\AFF_d(y,\overline{\delta},B_u,I^n_d)$,
and we will construct it from the second-order scheme for
$\WIS_d(x,\alpha,0,I^n_d)$ obtained in (\ref{potnucan}).
Since $\overline{\delta}$ is a
diagonal matrix such that $ \overline{\delta} - (d-1)I^n_d \in\posm
$, we have
\[
\delta_{\min}:=\min_{1\le i \le n} \overline{\delta}_{i,i} \ge
d-1.
\]
We rewrite the infinitesimal generator of $Y_t^y$ as follows:
%
\begin{eqnarray}\label{splitaff}\quad
L&=& \Tr\bigl([\overline{\delta}+
B_u(x)]D^{\mathcal{S}}\bigr)+ 2\Tr(xD^{\mathcal{S}}I^n_d
D^{\mathcal{S}})\nonumber\\[-8pt]\\[-8pt]
&=&{\underset{L_{\mathrm{ODE}}}{\underbrace{\Tr\bigl([\overline
{\delta}-\delta_{\min}I^n_d +
B_u(x)]D^{\mathcal{S}}\bigr)}}} +
{\underset{L_{\WIS_d(x,\delta_{\min},0,I^n_d)}}{\underbrace{\delta
_{\min} \Tr(I^n_d D^{\mathcal{S}})+ 2\Tr(xD^{\mathcal{S}}I^n_d
D^{\mathcal{S}})}}}.\nonumber
\end{eqnarray}
It is the sum of the infinitesimal generator of $\WIS_d(x,\delta_{\min
},0,I^n_d)$
and of the generator of the affine ODE
\[
dX^{\mathrm{ODE},x}_t=[\overline{\delta}-\delta_{\min
}I^n_d+B_u(X^{\mathrm{ODE},x}_t)]\,dt,\qquad X^{\mathrm{ODE},x}_0=x \in
\posm.
\]
We know by Lemma~\ref{MBCAFFINEODE} that $X^{\mathrm{ODE},x}_t
\in\posm$ for any $t\ge0$ since assumption (\ref{ReQuiredAssumption})
holds for $B_u$ and $\overline{\delta}-\delta_{\min}I^n_d \in\posm$.
Besides, this ODE can be solved explicitly [see formula
(\ref{ODEexplicitesol})]. Let $\hat{X}^x_t$ denote the potential
second-order scheme for $\WIS_d(x,\delta_{\min},0,I^n_d)$ obtained by
(\ref{potnucan}) that uses the nested second-order scheme for the CIR
given in~\cite{Alfonsi}. By using Proposition~\ref{composchemas}, the
schemes
%
\begin{equation}\label{schemacompoode} \hat{Y}^x_t=X^{\mathrm{ODE},\hat
{X}^{X^{\mathrm{ODE},x}_{t/2}}_t}_{t/2}
\quad\mbox{or}\quad
\hat{Y}^x_t=(1-B)\hat{X}^{X^{\mathrm{ODE},x}_t}_t + BX^{\mathrm{ODE},\hat{X}^x_t}_t
\end{equation}
are potential second-order schemes for
$\AFF_d(x,\overline{\delta},B_u,I_d^n)$. In the numerical
experiments in Section~\ref{SecSimul}, we have used
$X^{\mathrm{ODE},\hat{X}^{X^{\mathrm{ODE},x}_{t/2}}_t}_{t/2}$ even\vspace*{1pt} though the other scheme
would have worked as well; it is, in fact, a computational trade-off between
solving a deterministic ODE and drawing a Bernoulli variable. Thanks to
Lemma~\ref{lempotaff}, Proposition~\ref{Flow} and Theorem \ref
{Thmweak}, we finally get the following result.
%
\begin{theorem}\label{secondorderthm}
The scheme defined by Lemma~\ref{lempotaff} and
equation (\ref{schemacompoode}) is a potential second-order scheme for
$\AFF_d(x,\overline{\alpha},B,a)$. In the Wishart case (\ref
{paramwishart}), we
have for $f \in\Cpolde{\symm}$,
\[
\exists C,N_0>0, \forall N \ge N_0\qquad
|\E[f(\xcn_{t^N_N})]-\E[f(X^{x}_T)]|\le C/N^2.
\]
\end{theorem}

\subsection{\texorpdfstring{A faster second-order scheme for $\AFF_d(x,\overline{\alpha},B,a)$ when $\bar{\alpha}-da^Ta \in\posm$}
{A faster second-order scheme for AFF d(x,alpha,B,a) when alpha$-da^Ta in \posm$}}\label{subsecfaster}

In this section, we focus on the complexity of the discretization
schemes with respect to the dimension $d$. Up to now, the
discretization schemes that we have considered in
Theorems~\ref{thirdordertheorem} and~\ref{secondorderthm} have a
complexity of $O(d^4)$. Indeed, both schemes rely on the
construction (\ref{potnucan}) to sample $\WIS_d(x,\alpha,0,I^n_d)$,
which requires $n$ Cholesky decompositions, like the exact sampling. This
requires at most $O(d^4)$ operations. Here, we present a second-order
scheme whose
complexity is $O(d^3)$, provided that $\bar{\alpha}-da^Ta \in\posm$
or $\alpha
\ge d$ in the Wishart case. The practical relevance of such a scheme
will be
illustrated in Section~\ref{SecSimul}.

To do so, we use the same construction as in
Section~\ref{subsecsecondordaff}, and we remark that different
splitting from (\ref{splitaff}) are
possible. In fact, we could have chosen instead $L=\Tr([\overline
{\delta}-\beta I^n_d +
B_u(x)]D^{\mathcal{S}})+ \beta\Tr(I^n_d D^{\mathcal{S}})+ 2\Tr
(xD^{\mathcal{S}}I^n_d
D^{\mathcal{S}})$ for any $\beta\in[d-1,\delta_{\min}]$: the first
part is
the operator of an affine ODE which is well defined on $\posm$ by
Lemma~\ref{MBCAFFINEODE} while the second part is the generator of
$\WIS_d(x,\beta,0,I^n_d)$. When $\delta_{\min} \ge d$, which is
equivalent to
$\overline{\alpha}-da^Ta \in\posm$, the following splitting
obtained with $\beta=d$
%
\begin{equation}\label{splitaff2}
L={\underset{\tilde{L}_{\mathrm{ODE}}}{\underbrace{\Tr\bigl([\overline{\delta}-d
I^n_d +
B_u(x)]D^{\mathcal{S}}\bigr)}}} + {\underset
{L_{\WIS_d(x,d,0,I^n_d)}}{\underbrace{d \Tr(I^n_dD^{\mathcal{S}})+
2\Tr(xD^{\mathcal{S}}I^n_d
D^{\mathcal{S}})}}}
\end{equation}
is really interesting. Indeed it is known from Bru~\cite{Bru} that Wishart
processes with $\alpha\in\N$ can be seen as the square of an
Ornstein--Uhlenbeck process on
matrices and can be simulated very efficiently. More precisely, we will use
the following result that is shown in Appendix~\ref{AppsquareOU}.
%
\begin{proposition}\label{squareOU}
Let $x\in\posm$ and $c\in\mathcal{M}_d(\R)$ be such that $c^Tc=x$.
We have
\[
\bigl((c+W_tI^n_d)^T(c+W_tI^n_d),t\ge0\bigr) \underset{\mathit{Law}}{=}
\WIS_d(x,d,0,I^n_d).
\]
If $\hat{G}$ denote a $d$-by-$d$ matrix with independent elements sampled
according to~(\ref{VariableY}),
$\hat{X}^x_t=(c+\sqrt{t}\hat{G}I^n_d)^T(c+\sqrt{t}\hat{G}I^n_d)$
is a
potential second-order scheme for $\WIS_d(x,d,0,I^n_d)$.
\end{proposition}

To compute $\hat{X}^x_t$, one has to sample $d^2$ random variables and
to make
one matrix product, which requires $O(d^3)$ operations. This is faster than
the scheme obtained by (\ref{potnucan}). Then we follow the same
line as in
Section~\ref{subsecsecondordaff} and set
\[
d\tilde{X}^{\mathrm{ODE},x}_t=[\overline{\delta}-\delta_{\min
}I^n_d+B_u(\tilde{X}^{\mathrm{ODE},x}_t)]\,dt,\qquad
\tilde{X}^{\mathrm{ODE},x}_0=x \in
\posm.
\]
This ODE is well defined on $\posm$ and can be solved explicitly. By
Proposition~\ref{composchemas},
%
\begin{equation}\label{schemacompoode2} \hat{Y}^x_t=\tilde
{X}^{\mathrm{ODE},\hat{X}^{\tilde{X}^{\mathrm{ODE},x}_{t/2}}_t}_{t/2}
\quad\mbox{or}\quad \hat{Y}^x_t=(1-B)\hat{X}^{\tilde{X}^{\mathrm{ODE},x}_t}_t + B\tilde
{X}^{\mathrm{ODE},\hat{X}^x_t}_t
\end{equation}
is a potential second-order scheme for
$\AFF_d(x,\overline{\delta},B_u,I_d^n)$ that have still an $O(d^3)$ complexity.
Thanks to
Lemma~\ref{lempotaff}, Proposition~\ref{Flow} and Theorem \ref
{Thmweak}, we get
a similar result to Theorem~\ref{secondorderthm}.
%
\begin{theorem}\label{fastsecondorderthm}
Let us assume that $\bar{\alpha}-da^Ta \in\posm$. The scheme
defined by Lemma~\ref{lempotaff} and
equation (\ref{schemacompoode2}) is a potential second-order scheme for
$\AFF_d(x,\overline{\alpha},B,a)$ that requires at most $O(d^3)$
operations. In the Wishart case (\ref{paramwishart}), we
have for $f \in\Cpolde{\symm}$,
\[
\exists C,N_0>0, \forall N \ge N_0\qquad
|\E[f(\xcn_{t^N_N})]-\E[f(X^{x}_T)]|\le C/N^2.
\]
\end{theorem}

\section{Numerical results on the simulation methods}\label{SecSimul}

The scope of this section is to compare the different simulation
methods given
in this paper. We still consider a time horizon $T$ and the regular time-grid
$t^N_i=iT/N$, for $i=0,\ldots,N$. In addition, we want to compare our schemes
to a standard one, and we will consider the following
corrected Euler--Maruyama scheme for $\AFF_d(x,\overline{\alpha},B,a)$:
%
\begin{eqnarray}\label{EulerC}
\hat{X}_{t_{0}^N}^N&=&x,\nonumber\\
\hat{X}_{t_{i+1}^N}^N&=&\hat{X}_{t_{i}^N}^N +\bigl(
\overline{\alpha}+B(\hat{X}_{t_{i}^N}^N)\bigr)\frac{T}{N} +
\sqrt{(\hat{X}_{t_{i}^N}^N)^+}(W_{t_{i+1}^N}-W_{t_{i}^N})a\\
&&{} + a
^T(W_{t_{i+1}^N}-W_{t_{i}^N})^T \sqrt{(\hat{X}_{t_{i}^N}^N)^+},\qquad 0\le
i\le N-1.\nonumber
\end{eqnarray}
Here, $x^+$ denotes the matrix that has the same eigenvectors as $x$
with the
same eigenvalue if it is positive and a zero eigenvalue otherwise.
Namely, we set $x^+=o
\operatorname{diag}(\lambda_1^+,\ldots,\lambda_d^+) o^T$ for
$x=o \operatorname{diag}(\lambda_1,\ldots,\lambda_d) o^T$. Thus, $x^+$ is by
construction a
positive semidefinite matrix and its square root is well defined. Without
this positive part, the scheme above is not well defined for any realization
of $W$.

First, we compare the time required by the different schemes and the
exact simulation. Then, we present numerical results on the
convergence of the different schemes. Last, we give an application of our
scheme to the Gourieroux--Sufana model in finance.

\subsection{Time comparison between the different algorithms}
In this paragraph, we compare the time required by the different
schemes given in this paper. As it has already been mentioned, the
complexity of the exact scheme as well as the one of the second-order
scheme (given by Theorem~\ref{secondorderthm}) and the third-order
scheme (given by Theorem~\ref{thirdordertheorem}) is in $O(d^4)$ for
one time-step. To be more precise, they require $O(d^4)$ operations
that mainly correspond to $d$ Cholesky decompositions, $O(d^2)$
generations of Gaussian (or moment-matching) variables and $O(d)$
generations of noncentral chi-square distributions (or second or
third-order schemes for the CIR). The time saved by the second and
third-order schemes with respect to the exact scheme only comes from
the generation of random variables. For example, the generation of the
moment-matching variables (\ref{VariableY}) and (\ref{VariableY5}) is
$2.5$ faster than the generation of $\mathcal{N}(0,1)$ on our computer.
The gain between the second or third-order schemes for the CIR given in
Alfonsi~\cite{Alfonsi} and the exact sampling of the CIR given by
Glasserman~\cite{Glasserman} is much greater, but it depends on the
parameters of the CIR. When the dimension $d$ gets larger, the absolute
gain in time between the discretization schemes and the exact scheme
is, of course, increased. However, the relative gain instead decreases
to $1$, because more and more time is devoted to matrix operations and
Cholesky decompositions that are the same in both cases. Let us now
quickly analyze the complexity of the other schemes. The second-order
scheme given by Theorem~\ref{fastsecondorderthm} (called ``second-order
bis'' later) has a complexity in $O(d^3)$ operations for one Cholesky
decomposition and matrix multiplications, with $O(d^2)$ generations of
Gaussian variables. The complexity of the corrected Euler scheme is of
the same kind. At each time-step, $O(d^3)$ operations are needed for
matrix multiplications and for diagonalizing the matrix in order to
compute the square root of its positive part. However, diagonalizing a
symmetric matrix is, in practice, much longer than computing a Cholesky
decomposition even though both algorithms are in $O(d^3)$. Also, one
has to sample $O(d^2)$ Gaussian variables for the Brownian increments.

\begin{table}
\caption{$\E[\exp(-\Tr(i{v \hat{X}_{t^N_N}^N}))]$ calculated by a Monte
Carlo with $10^6$ samples for a Wishart process with $a=I_d$, $b=0$,
$x= 10I_d$, $v=0.09I_d$ and $T=1$. The starred numbers are those for
which the exact value is outside the $95\%$ confidence interval, and
$\Delta_R$ (resp., $\Delta_I$) gives the two standard deviations~value
on the real (resp., imaginary) part} \label{ResultTable}
\begin{tabular*}{\tablewidth}{@{\extracolsep{\fill}}ld{2.7}
d{2.7}d{4.0}d{2.7}
d{2.7}d{4.0}@{}}
\hline
&\multicolumn{3}{c}{$\bolds{N=10}$} &
\multicolumn{3}{c@{}}{$\bolds{N=30}$}\\[-4pt]
&\multicolumn{3}{c}{\hrulefill} &
\multicolumn{3}{c@{}}{\hrulefill}\\
\multicolumn{1}{@{}l}{\textbf{Schemes}}
& \multicolumn{1}{c}{\textbf{R. value}}
& \multicolumn{1}{c}{\textbf{Im. value}}
& \multicolumn{1}{c}{\textbf{Time}}
& \multicolumn{1}{c}{\textbf{R. value}}
& \multicolumn{1}{c}{\textbf{Im. value}}
& \multicolumn{1}{c@{}}{\textbf{Time}} \\
\hline
\multicolumn{7}{c@{}}{$\alpha = 3.5$, $d=3,\Delta_R=10^{-3},\Delta
_{\mathit{Im}}=10^{-3}$,}\\
\multicolumn{7}{c@{}}{exact value $\mbox{R.} = -0.527090$ and
$\mbox{Im.}= -0.228251$}\\
[4pt]
Exact (1 step) & -0.526852 & -0.227962 & 12 & & & \\
2nd-order bis & -0.526229 & -0.228663 & 41 & -0.526486 &
-0.229078 &  125\\
2nd order & -0.526577 & -0.228923 & 76 & -0.526574 &
-0.228133 & 229\\
3rd order & -0.527021 & -0.227286 & 82 & -0.527613 &
-0.228376 & 244\\
Exact ($N$ steps) & -0.526963 & -0.228303 & 123 & -0.526891 &
-0.227729 & 369\\
Corrected Euler & -0.525627^* & -0.233863^* & 225  &
-0.525638^* &-0.231449^* & 687\\
[4pt]
\multicolumn{7}{c@{}}{$\alpha=2.2$, $d=3,\Delta_R=0.9 \times
10^{-3},\Delta_{\mathit{Im}}=1.3 \times10^{-3}$,}\\
\multicolumn{7}{c@{}}{exact value $\mbox{R.} = -0.591411$
and $\mbox{Im.}= -0.036346$}\\
[4pt]
Exact (1 step) & -0.591579 & -0.037651 & 12 & & & \\
2nd order & -0.590444 & -0.037024 & 77 & -0.590808 &
-0.036487 & 229\\
3rd order & -0.591234 & -0.034847 & 82  & -0.590818 &
-0.036210 & 246\\
Exact ($N$ steps) &-0.591169 & -0.036618 & 174 & -0.592145 &
-0.037411 & 920\\
Corrected Euler & -0.589735^* & -0.042002^* & 223 & -0.590079^*
&-0.039937^* & 680\\
[4pt]
\multicolumn{7}{c}{$\alpha=10.5,d=10,\Delta_R=1.4 \times
10^{-3},\Delta_{\mathit{Im}}=1.3 \times10^{-3}$,}\\
\multicolumn{7}{c@{}}{exact value $\mbox{R.} = 0.063960$
and $\mbox{Im.}= -0.063544$}\\
[4pt]
Exact (1 step) & 0.062712 & -0.063757 & 181 & & & \\
2nd-order bis & 0.064237 & -0.063825 & 921 & 0.064573 &
-0.062747 & 2762 \\
2nd order & 0.064922 & -0.064103 & 1431 & 0.063534 &
-0.063280 &4283 \\
3rd order & 0.064620 & -0.064543 & 1446 & 0.064120 &
-0.063122 & 4343\\
Exact ($N$ steps) & 0.063418 & -0.064636 & 1806 & 0.063469 &
-0.064380 &5408\\
Corrected Euler & 0.068298^* & -0.058491^* & 2312 & 0.061732^*
& -0.056882^* &7113\\
[4pt]
\multicolumn{7}{c}{$\alpha=9.2,d=10,\Delta_R=1.4 \times
10^{-3},\Delta_{\mathit{Im}}=1.4\times10^{-3}$,}\\
\multicolumn{7}{c@{}}{exact value $\mbox{R.} = -0.036064$
and $\mbox{Im.}= -0.093275$}\\
[4pt]
Exact (1 step) & -0.036869 & -0.094156 & 177 & & & \\
2nd order & -0.036246 & -0.094196 & 1430 & -0.035944
&-0.092770 &4285\\
3rd order & -0.035408 & -0.093479 & 1441
&-0.036277&-0.093178& 4327\\
Exact ($N$ steps) & -0.036478 &-0.092860 & 1866 & -0.036145 &
-0.093003 & 6385\\
Corrected Euler & -0.028685^* & -0.094281^* & 2321 &-0.030118^*
&-0.088988^*& 7144\\
\hline
\end{tabular*}
\end{table}

In Table~\ref{ResultTable}, we have calculated by a Monte Carlo method
one value of the characteristic function of a Wishart process. It is
also known analytically thanks to~(\ref{CarWishart}), and we have
indicated in each case the exact value. We have considered dimensions
$d=3$ and $d=10$. We have given in each case an example where
$\alpha\ge d$ and another one where $d-1 \le\alpha< d$. We have used
the different algorithms\vadjust{\goodbreak} presented in this paper: ``$2$nd-order bis''
stands for the scheme given by Theorem~\ref{fastsecondorderthm} [with
the moment-matching variables (\ref{VariableY})], ``$2$nd order''
stands for the scheme given by Theorem~\ref{secondorderthm} (with
(\ref{VariableY}) and the second-order scheme for the CIR given by
\cite{Alfonsi}), ``$3$rd order'' stands for the scheme given by Theorem
\ref{thirdordertheorem} (with (\ref{VariableY5}) and the third-order
scheme for the CIR given by~\cite{Alfonsi}) and ``Corrected Euler''
stands for the corrected Euler--Maruyama scheme (\ref{EulerC}). For the
exact scheme, we have considered both the cases with one time-step $T$
and $N$ time-steps $T/N$. Of course, the first case is sufficient to
calculate an expectation that only depends on $X_T$, but the second
case allows us to also compute pathwise expectations. For each method,
we have given the value obtained and the time needed (in seconds) on
our computer (3000 MHz~CPU).

First, let us mention that the exact value is in each case in the confidence
interval except for the corrected Euler
scheme. As one can expect, the exact method with one time-step is by
far the quickest
method to compute an expectation that only depends on the final value.
We put
aside this case and focus now on the generation of the whole path. We
see from
Table~\ref{ResultTable} that the second and the third-order schemes require
roughly the same computation time. As expected, the second-order scheme
bis is
much faster when it is defined (i.e., when $\alpha\ge d$). On the contrary,
the Euler scheme is much slower than the second and third-order scheme.
This is due to the
cost of the matrix diagonalization. Let us mention that the time
required by
the discretization schemes is proportional to $N$ and do not depend on
the parameters when the dimension
is given. On the contrary, the time needed by the exact scheme may change
according to $\alpha$ and can increase considerably when $\alpha$ is close
to $d-1$. To be more precise, the exact simulation method for the CIR
given by
Glasserman~\cite{Glasserman} uses a rejection sampling when the degree of
freedom is lower than $1$, which corresponds to the case $d-1 \le
\alpha
< d$. The rejection rate can in fact be rather high, notably when the
time-step gets smaller. For $N=30$, $d=3$ and $\alpha=2.2$, the exact
scheme is four times slower than the second-order scheme and $2.5$
slower than
the exact scheme with $\alpha=3.5$.

Let us draw a conclusion from this time comparison between the different
schemes. Obviously, we recommend the use of the exact scheme when
calculating expectations that depend on one or few dates. Instead, when
calculating pathwise
expectations of affine processes by Monte Carlo, we would recommend the
use of, in general, the second-order bis scheme when $\alpha\ge d$ and the second order (or third
order for
Wishart processes) when $d-1\le\alpha< d$.

\subsection{Numerical results on the convergence}\label{subsecconv}

Now we want to illustrate the theoretical results of convergence
obtained in
this paper for the different schemes. To do so, we have plotted for each
scheme $\E[\exp(-\Tr(i{v \hat{X}_{t^N_N}^N}))]$ in function of the time
step $T/N$. This expectation is calculated by a Monte Carlo method. As
for the time comparison, we
illustrate the convergence for $d=3$ in Figure~\ref{vfWishart3O1} and
$d=10$ in
Figure~\ref{Wishart10}. Each time, we consider a case where $\alpha
\ge d$ and a case
where $d-1 \le\alpha< d$, which is in general tougher. In these
figures:
\begin{itemize}
\item scheme $1$ denotes the value obtained by the exact scheme with one
time-step,
\item scheme $2$ stands for the second-order scheme given by
Theorem~\ref{secondorderthm},
\item scheme $3$ denotes the third-order scheme given by Theorem \ref
{thirdordertheorem},
\item scheme $4$ is the corrected Euler scheme (\ref{EulerC}).
\end{itemize}
Here, we have not plotted the convergence of the second-order (bis) scheme
given by Theorem~\ref{fastsecondorderthm} because it
would have given almost the same convergence as the other second-order
scheme.

\begin{figure}

\includegraphics{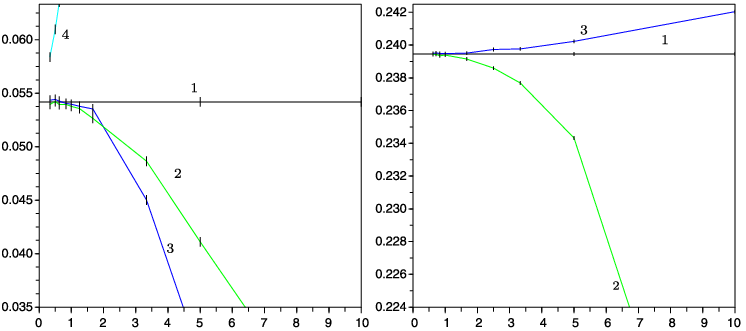}

\caption{$d=3$, $10^7$ Monte Carlo samples, $T=10$. The real value of
$\E[\exp(-\Tr(i{v \hat{X}_{t^N_N}^N}))]$ in function of the time-step
$T/N$. Left: $v=0.05I_d$ and Wishart parameters $x=0.4I_d$,
$\alpha=4.5$, $a=I_d$ and $b=0$. Exact value: $0.054277$. Right: $v=
0.2I_d+0.04q$ and Wishart parameters $x= 0.4I_d+0.2q$, $\alpha=2.22$,
$a=I_d$ and $b=-0.5I_d$. Exact value: $0.239836$. Here, $q$ is the matrix
defined by: $q_{i,j}=\mathbh{1}_{i \not= j}$. The width of each point
represents the $95\%$ confidence interval.}
\label{vfWishart3O1}
\end{figure}

\begin{figure}

\includegraphics{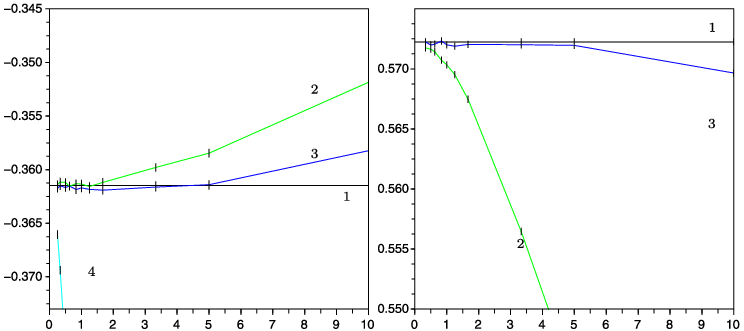}

\caption{$d=10$, $10^7$ Monte Carlo samples, $T=10$. Left:
imaginary value of $\E[\exp(-\Tr(i{v \hat{X}_{t^N_N}^N}))]$ with
$v=0.009I_d$ in function of the time-step
$T/N$. Wishart parameters: $x=0.4I_d$, $\alpha=12.5$, $b=0$ and
$a=I_d$. Exact value: $-0.361586$. Right: real value of $\E[\exp(-\Tr
(i{v \hat{X}_{t^N_N}^N}))]$ with
$v=0.009I_d$ in function of $T/N$. Wishart parameters: $x=0.4I_d$,
$\alpha=9.2$, $b=-0.5I_d$ and $a=I_d$. Exact value $0.572241$. The
width of each point
represents the $95\%$ confidence interval.}
\label{Wishart10}
\end{figure}

As expected, we observe in both Figures~\ref{vfWishart3O1} and~\ref{Wishart10}
convergences that fit our theoretical results. Namely, scheme 2
converges in
$O(1/N^2)$ and scheme 3 converges faster in $O(1/N^3)$. In some cases,
such as
Figure~\ref{Wishart10}, scheme 3 already matches the exact value from
$N=2$. Even though it seems to converge at an $O(1/N)$ speed, the corrected
Euler scheme is clearly not competitive with respect to
the other schemes. In the tough case $d-1 \le\alpha\le d$, the values
obtained by the Euler scheme are in fact outside the figures, and we
have put
the corresponding values in Table~\ref{TableEuler}.

We want to conclude this section by numerically testing the convergence
of our
schemes when we calculate pathwise expectations. Of course, our theoretical
results only bring on the weak error, but we may hope that our schemes
converge also quickly when considering more intricate expectations. In
Figure~\ref{vfWishartSup}, we approximate $\E[\max_{0\le t\le T} \Tr
(X^x_t)]$
with the different schemes by computing the maximum on the time-grid. The
convergence seems to be roughly in $O(1/\sqrt{N})$ for all the schemes
(see Figure~\ref{vfWishartSup}, left), including the exact scheme.
However, the
main error seems to come from the approximation of $\max_{0\le t\le T}
\Tr(X^x_t)$ by $\max_{0\le k\le N} \Tr(X^x_{t^N_k}) $. In fact,\vspace*{-1pt} we
have plotted
in Figure~\ref{vfWishartSup} (right) the difference between $\E[\max
_{0\le k\le N
}\Tr(\hat{X}_{t_k^N}^N)]$ and $\E[\max_{0\le k\le N
}\Tr(X_{t_k^N}^x)]$. Then, we find convergences that are very similar
to those obtained for the weak error: schemes 2 and 3 converge at a
speed which
is, respectively, compatible with $O(1/N^2)$ and $O(1/N^3)$. Scheme 4
seems also
to give an $O(1/N)$ convergence. It would be hasty to draw a global conclusion
from this simple example. Nonetheless, the convergence of schemes 2 and
3 is
really encouraging on pathwise expectations, if we put aside the
problem of
approximating a function of $(X^x_t,0\le t\le T)$ by a function of
$(X^x_{t^N_k},0\le k\le N)$.

\begin{table}[b]
\caption{Values obtained by the Euler scheme in the numerical
experiments of
Figures \protect\ref{vfWishart3O1} and \protect\ref
{Wishart10}}\label{TableEuler}
\begin{tabular*}{\tablewidth}{@{\extracolsep{\fill}}ld{2.6}d{2.6}lllc@{}}
\hline
$\bolds{N}$ & \multicolumn{1}{c}{\textbf{2}}
& \multicolumn{1}{c}{\textbf{4}} & \multicolumn{1}{c}{\textbf{8}}
& \multicolumn{1}{c}{\textbf{10}} & \multicolumn{1}{c}{\textbf{16}}
& \multicolumn{1}{c@{}}{\textbf{30}} \\
\hline
Figure~\ref{vfWishart3O1}, right & -0.000698 & 0.000394 & 0.033193 &
0.111991 & 0.185128 & 0.210201 \\
Figure~\ref{Wishart10}, right & 0.494752 & -0.464121 & 0.657041 &0.643042
& 0.637585 & 0.619553 \\
\hline
\end{tabular*}
\end{table}
%

\begin{figure}

\includegraphics{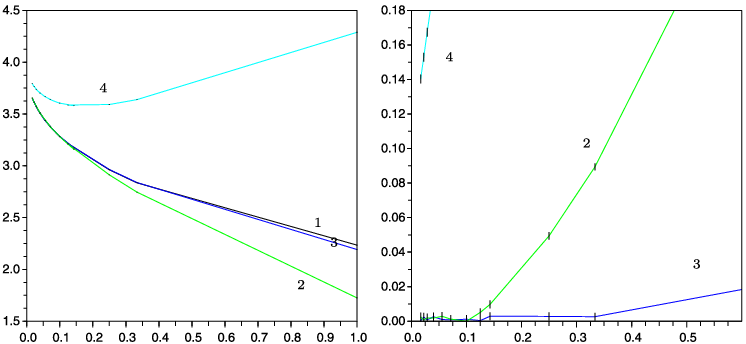}

\caption{$d=3$, $10^7$ Monte Carlo samples, $T=1$. Wishart
parameters $x=0.4I_d+0.2q$ with $q_{i,j}=\mathbh{1}_{i \not= j}$,
$\alpha=2.2$, $b=0$ and
$a=I_d$. Left: $\E[\max_{0\le k\le N
}\Tr(\hat{X}_{t_k^N}^N)]$.\vspace*{-1pt} Right: $\E[\max_{0\le k\le N
}\Tr(\hat{X}_{t_k^N}^N)]-\E[\max_{0\le k\le N
}\Tr(X_{t_k^N}^x)]$ in function of
$T/N$. The width of each point gives the precision up to two standard
deviations.}
\label{vfWishartSup}
\end{figure}

\subsection{An application in finance to the Gourieroux and Sufana model}

In this paragraph, we want to give a possible application of our
schemes in
finance. More precisely, we will consider the model introduced by Gourieroux
and Sufana~\cite{Gourieroux}. This is a model for $d$ risky assets
$S^1_t,\ldots,S^d_t$. Let $(B_t,t\ge0)$ denote a standard Brownian motion
on $\R^d$ that is independent from $(W_t,t \ge0)$. Then, we consider the
following dynamics for the assets:
%
\begin{equation}\label{sdegour}
t\ge0, 1\le l\le d,\qquad S_t^l=S_0^l+r\int_0^t S_u^l\,du +\int_0^t S_u^l
\bigl(\sqrt{X_u}\,dB_u\bigr)_{l},
\end{equation}
where $X_t =X_0 + \int_{0}^t (
\alpha a^Ta + bX_u + X_ub^T )\,du +
\int_{0}^t ( \sqrt{X_u}\,dW_ua + a^T\,dW_u^T\sqrt{X_u})$ is a
Wishart process. Here, $(\sqrt{X_u}\,dB_u)_{l}$ is simply the $l$th
coordinates of the vector $\sqrt{X_u}\,dB_u$. We can easily check that the
instantaneous quadratic covariation matrix between the log-prices of
the assets
is $X_t$. Last, $r$ denotes the instantaneous interest rate.

To simulate both assets and the Wishart matrix, we proceed as follows. We
observe that the generator of $(S_t,X_t)$ can be written as
\[
L=L^S+L^X\qquad\mbox{where } L^S= \sum_{i=1}^d rs_i \partial_{s_i}+
\frac{1}{2}\sum_{i,j=1}^d s_i s_j x_{i,j} \partial_{s_i}\partial_{s_j},
\]
and $L^X$ is the generator of the Wishart process $\WIS_d(x,\alpha
,b,a)$. The operator $L^S$ is associated to the SDE
$dS^l_t=rS^l_t+S_t^l (\sqrt{x}\,dB_t)_{l}$ that\vspace*{1pt} can be
solved explicitly. We have indeed $S^l_t=S^l_0 \exp[ (r-x_{l,l}/2)t
+(\sqrt{x}B_t)_l ]$. Let us also remark that $\sqrt{x}B_t
\underset{\mathrm{Law}}{=} c B_t$ if we have $cc^T=x$; both are
centered Gaussian vectors with the same covariance matrix. In practice,
it is more efficient to use $S^l_t=S^l_0 \exp[ (r-x_{l,l}/2)t +(cB_t)_l
]$ where $c$ is computed with an extended Cholesky decomposition of $x$
rather\vadjust{\goodbreak} than calculating $\sqrt{x}$, which requires a diagonalization.
Then we consider the scheme given by {2}(a) in Proposition~\ref{composchemas}, 
where we take the second-order scheme for
$\WIS_d(x,\alpha,b,a)$ and the exact scheme for $L^S$. This
construction is known to preserve the second-order convergence. To be
consistent with Section~\ref{subsecconv}, this scheme will be denoted
by scheme 2 in this paragraph. To compare this scheme with a more basic
one, we consider the Euler--Maruyama scheme defined by (\ref{EulerC})
and
\begin{eqnarray*}
\hat{S}^{l,N}_{t^N_{0}}&=& S^l_0,\\[-2pt]
\hat{S}^{l,N}_{t^N_{i+1}}&=&\hat{S}^{l,N}_{t^N_{i}}
\bigl(1+rT/N+\bigl(\sqrt{(\hat{X}^N_{t^N_{i}})^+}
(B_{t^N_{i+1}}-B_{t^N_{i}} )\bigr)_{l}
\bigr),\qquad 0\le i\le N-1.
\end{eqnarray*}
It is denoted by scheme 4 as in Section~\ref{subsecconv}.

We have plotted in Figure~\ref{vGourieroux} the price of a put option
on the
maximum of two risky assets ($d=2$). The Gourieroux and Sufana model is an
affine model, and the characteristic function of $S_t$ is explicitly known
(see~\cite{Gourieroux}). Thus, it is possible to
adapt the method proposed by Carr and Madan~\cite{CM} and to calculate by
numerical integration (which is possible for small dimensions) the
%
\begin{figure}

\includegraphics{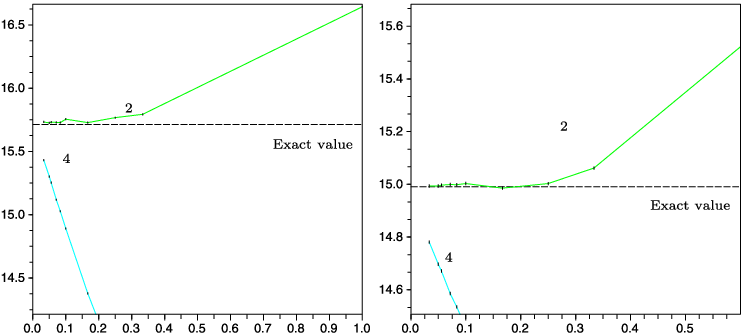}

\caption{$\E[e^{-rT}(K-\max{(\hat{S}^{1,N}_{t_N^N},\hat
{S}^{2,N}_{t_N^N})})^+]$
in function of $T/N$. $d=2$, $T=1$, $K=120$, $S^1_0=S^2_0=100$ and
$r=0.02$. Wishart parameters:
$x=0.04I_d+0.02q$ with $q_{i,j}=\mathbh{1}_{i \not= j}$, $a=0.2I_d$,
$b=0.5I_d$ and $\alpha=4.5$ (left), $\alpha=1.05$ (right). The width
of each point gives the precision up to two standard
deviations ($10^6$ Monte Carlo samples).} \label{vGourieroux}\vspace*{-3pt}
\end{figure}
value of this put option. We have given in Figure~\ref{vGourieroux} the
exact value obtained by this method. As one might have guessed, we
observe a quadratic convergence for scheme $2$ and a linear convergence
for scheme $4$. The benefit of using scheme $2$ is clear since it
already fits with the exact value from $N=5$ in both cases; its
convergence is really satisfactory.\vspace*{-2pt}

\section{Conclusion and prospects}

Let us draw a brief summary of this paper. Thanks to a remarkable
splitting of the infinitesimal generator of Wishart processes,\vadjust{\goodbreak} we have
been able to sample exactly any Wishart distribution. We have also
proposed a third-order scheme for Wishart processes and a second-order
scheme for general affine diffusions. We have confirmed these rates of
convergence with numerical tests and analyzed the time complexity of
each method. It comes out that we recommend to use the exact scheme to
compute expectations that depend on one (or few) times. To calculate
pathwise expectations, we instead recommend generally to use
discretization schemes. More precisely, the second-order scheme given
by Theorem~\ref{fastsecondorderthm} has to be preferred when $\alpha\ge
d$. Otherwise, we recommend to use the third-order scheme given by
Theorem~\ref{thirdordertheorem} for Wishart processes or the
second-order scheme given by Theorem~\ref{secondorderthm} for general
affine diffusions.

Let us give now some prospects of this work. As a possible continuation of
this paper, it is natural to study how it is possible to extend our
schemes to
affine diffusions on positive semidefinite matrices
that include jumps (see Cuchiero et al.~\cite{Teichmann}). From a modeling
point of view, we believe that Wishart processes could be used in a
wide range
of applications. In fact, they can be used as soon as one has to model
dependence dynamics. Thus, we hope that the possibility of sampling such
processes will stimulate different kinds of dependence models.

\begin{appendix}\label{app}

\section{The extended Cholesky decomposition}\label{Resmatrices}

\begin{lemma}\label{OuterProdDec}
Let $q \in\mathcal{S}^+_{d}(\mathbb R)$ be a matrix with rank $r$. Then
there is a permutation matrix $p$, an invertible lower triangular
matrix $c_r
\in\mathcal{G}_r(\mathbb R)$ and $k_r\in
\mathcal{M}_{d-r\times r}(\mathbb R)$ such that
\[
pqp^T = c c^T,\qquad
c = \pmatrix{
c_r & 0\cr
k_r & 0}.
\]
The triplet $(c_r,k_r,p)$ is called an extended Cholesky decomposition of
$q$.
Besides, $\tilde{c}=
\bigl({c_r \atop k_r}\enskip{0 \atop I_{d-r}}\bigr)
\in\nsing$, and we have
\[
q=(\tilde{c}^T p)^T I^r_d \tilde{c}^Tp.
\]
\end{lemma}

The proof and a numerical procedure to get such a decomposition
can be found in Golub and Van Loan (\cite{Golub}, Algorithm 4.2.4).
When $r=d$, we can take $p=I_{d}$, and $c_r$ is the usual
Cholesky decomposition.


\section{\texorpdfstring{Proofs of Section \lowercase{\protect\ref{chapter2}}}{Proofs of Section 1}}
%
%
\subsection{\texorpdfstring{Proof of Proposition \protect\ref{LaplaceGeneral}}{Proof of Proposition 4}}
\label{proofLaplaceGeneral}

We will need in the proof the following basic lemma.
%
\begin{lemma}\label{det0} Let $b,c\in\symm$. If either $b \in\posm
$ or $c
\in\posm$, then $I_d+ibc$ is invertible. In particular, if $b \in
\dpos$,
$b+ic$ is invertible.
\end{lemma}
%
%
\begin{pf}
Let $v\in\symm$ such that $ \forall
s \in[0, t], I_d-2q_s v \in\nsing$. As it is usual for affine
diffusions, the Laplace\vadjust{\goodbreak} transform can be formulated with ODE
solutions. Namely, we will show that
$ \mathbb E[\exp( \Tr(vX_t^x))] =
\exp[\phi(t,v)+ \Tr(\psi(t,v)x)]$,
where $\psi$ and $\phi$ solve the following ODEs (see, e.g., Cuchiero
et al.~\cite{Teichmann}):
\begin{eqnarray*}
\partial_t \psi(t,v) &=& \psi(t,v)b + b^T\psi(t,v) +2\psi(t,v)a^T
a\psi(t,v);\qquad \psi(0,v)=v,\\
\partial_t \phi(t,v)&=& \alpha\Tr(\psi(t,v));\qquad \phi(0,v)=0.
\end{eqnarray*}
The function $\psi$ solves an usual matrix Riccati ODE. As shown by
Levin~\cite{Levin}, $\psi$~can be obtained explicitly by the mean of an
exponential matrix, and we get
%
\[
\psi(t,v)=\exp(tb^T)(I_d-2v q_t)^{-1}v \exp(tb),
\]
%
provided that $I_d-2q_sv$ is invertible for $s\in[0,t]$, which holds
by assumption. Therefore we get, for $x \in\symm$,
\begin{eqnarray*}
\Tr(\psi(t,v)x)&=&\Tr\bigl( (I_d-2vq_t)^{-1}v\exp(tb)x \exp(tb^T)
\bigr)\\
&=&\Tr\bigl( v (I_d-2q_tv)^{-1} \exp(tb)x \exp(tb^T)\bigr),
\end{eqnarray*}
%
since $v (I_d-2q_tv)^{-1}=(I_d-2vq_t)^{-1}v$. As explained by Grasselli and
Tebaldi (\cite{Tebaldi}, Section 4.2), $\phi$ can also be calculated
explicitly by the mean of the exponential matrix above, and we get
\[
\phi(t,v)=-\frac{\alpha}{2} \Tr\bigl( \log[
(I_d-2vq_t)\exp(tb^T) ] -t \Tr(b) \bigr).
\]
By using that $\exp(\operatorname{Tr}(\operatorname{log}(A)))= \det
(A)$ for $A
\in
\nsing$, we deduce then that
\begin{eqnarray*}
\exp(\phi(t,v))&=&\exp\biggl(\frac{\alpha}{2} t \operatorname
{Tr}(b)\biggr)\bigl(\operatorname{det}\{(I_d-2vq_t)\}
\operatorname{det}\{\exp(tb^T)\}\bigr)^{{-\alpha
}/{2}}\\
&=&\frac
{1}{\det(I_d-2q_tv)^{{\alpha}/{2}}}.
\end{eqnarray*}
%


Now it remains to show that (\ref{CarWishart}) indeed holds. By It\^
{o} calculus, we get
that for $s \in(0,t)$,
%
\begin{eqnarray} \label{Itocar}
&&d\exp\bigl[\phi(t-s,v)+\Tr\bigl(\psi
(t-s,v)X^x_s\bigr)\bigr]\nonumber\\
&&\qquad=\exp\bigl[\phi(t-s,v)+\Tr\bigl(\psi(t-s,v)X^x_s\bigr)\bigr]\\
&&\qquad\quad{}\times\Tr\bigl[\psi(t-s,v)\bigl(\sqrt{X_s^x}\,dW_sa +
a^T\,dW_s^T\sqrt{X_s^x}\bigr)\bigr].\nonumber
\end{eqnarray}
Thus, $\exp[\phi(t-s,v)+\Tr(\psi(t-s,v)X^x_s)]$ is a positive local
martingale and therefore a supermartingale, which gives that
$\E[\exp(\Tr(vX^x_t))]\le\exp[\phi(t,v)+\Tr(\psi(t,v)x)]<\infty
$, that is, $\mathcal{D}_{b,a;t} \subset\tilde{\mathcal
{D}}_{x,\alpha,b,a;t}, \mbox{ where } $
\[
\mathcal{D}_{b,a;t}:= \{v \in\symm, \forall s \in[0, t], I_d-2q_s
v \in\nsing\}
\]
and
\[
\tilde{\mathcal
{D}}_{x,\alpha,b,a;t}:=\{v
\in\symm, \E[\exp(\Tr(vX^x_t))]< \infty\}.
\]
On the other hand, when $-v \in\dpos$, we can check that
$\exp[\phi(t-s,v)+\Tr(\psi(t-s,v)X^x_s)]\le1$ by observing that
$\det(I_d-2q_tv)=\det(I_d+2\sqrt{-v}\*q_t\sqrt{-v})\ge1$ and
$\Tr( v (I_d-2q_tv)^{-1} \exp(tb)x \exp(tb^T))=-\Tr
( \sqrt{-v}
(I_d+2\*\sqrt{-v}q_t\sqrt{-v})^{-1}\sqrt{-v} \exp(tb)x \exp
(tb^T))\le0
$. In that case, $\exp[\phi(t-s,v)+\Tr(\psi(t-s,v)X^x_s)]$ is a martingale
from (\ref{Itocar}), and (\ref{CarWishart}) holds.

Let us now observe that $\mathcal{D}_{b,a;t}$ is
convex. In fact, we have $\det(I_d-2q_s v)=\det(I_d-2\sqrt{q_s}
v\sqrt{q_s})$,
and therefore, $\mathcal{D}_{b,a;t}= \{v \in\symm, \forall s \in
[0, t], I_d-\break 2 \sqrt{q_s} v \sqrt{q_s} \in\dpos\}$ which is obviously
convex. The Laplace transform $v \mapsto\E[\exp(\Tr(vX^x_t))]$
is an analytic function on $\mathcal{D}_{b,a;t}$ (see, e.g., \cite
{Filipovic}, Lemma~10.8).
The right-hand side of (\ref{CarWishart}) is also analytic on
$\mathcal{D}_{b,a;t}
$ and coincides with the Laplace
transform when $-v \in\dpos$. Therefore, (\ref{CarWishart}) holds
for $v \in
\mathcal{D}_{b,a;t}$ since $\mathcal{D}_{b,a;t}$ is convex. Now, we
can extend
to complex values of $v$. Indeed, the right-hand side of (\ref
{CarWishart}) is well defined for $v=v_R+iv_I$ with $v_R \in\mathcal
{D}_{b,a;t}$, thanks to
Lemma~\ref{det0}. Since both-hand sides
are analytic functions of $v$, (\ref{CarWishart}) holds for $v=v_R+iv_I$.

Last, we want to show that
$\mathcal{D}_{b,a;t}=\tilde{\mathcal{D}}_{x,\alpha,b,a;t}$. We
first consider
the case $b=0$ and assume by a way of contradiction that there is $v
\in
\tilde{\mathcal{D}}_{x,\alpha,0,a;t} \setminus\mathcal{D}_{0,a;t}$
for some
$x$, $\alpha$, $a$ and $t>0$. Let $\tilde{t}= \min
\{s \in[0,t], I_d-2 q_s v \notin\nsing\} \in(0,t]$. On the
one hand, we have $v \notin\mathcal{D}_{0,a; \tilde t}$ and $v \in
\mathcal{D}_{0,a; s}$ for $s \in[0, \tilde t)$. On the other hand, we
have, by Jensen's inequality
\[
s\in[0,t],\qquad \exp\bigl(\alpha(t-s) \Tr(va^T a )\bigr) \exp( \Tr(vX^x_{s}))
\le\E[\exp( \Tr(vX^x_t)) |\mathcal{F}_{s}],
\]
which gives $s \in[0, t] \mapsto
\exp(-\alpha s \Tr(v a^Ta)) \E[\exp(\Tr(vX^x_{ s}))]$ is
nondecreasing and finite.\vspace*{1pt} Since (\ref{CarWishart}) holds for
$s<\tilde t$, we get that
$\E[\exp(\Tr(vX^x_{ \tilde t}))]=+\infty$, which leads to a contradiction.
Let us now consider the case $b \not= 0$. From Proposition \ref
{Propidloiwis} (which is a
consequence of the characteristic function obtained above), we have
\begin{eqnarray*}
v \in\tilde{\mathcal{D}}_{x,\alpha,b,a; t} \quad&\iff&\quad\theta_t^T v
\theta_t
\in\mathcal{D}_{0,I^n_d;t} \\
&\iff&\quad\forall s \in[0,t ]\qquad \det\bigl(I_d -2(s/t)
q_t v\bigr) \not= 0.
\end{eqnarray*}
In particular, $\tilde{\mathcal{D}}_{x,\alpha,b,a; t}$ is an open
set. For $v
\in\nsing$, we have $\det(I_d -2(s/t)
q_t v) \not= 0 \iff\det(v^{-1} -2(s/t)
q_t ) \not= 0$ [resp., $\det(I_d -2 q_s v) \not= 0 \iff\det(v^{-1}
-2 q_s )\not=
0 $]. Since $s q_t \le s'q_t$ (resp., $q_s \le q_{s'}$) for $s \le s'$,
we know
from Theorem 8.1.5 in~\cite{Golub} that the (real) eigenvalues of
$v^{-1} -2(s/t)
q_t$ (resp., $v^{-1} -2 q_s$) are nonincreasing w.r.t. $s$. Since they
are also
continuous, and $v^{-1} -2(s/t) q_t=v^{-1} -2 q_s$ for $s\in\{0,t \}$, we
get\vspace*{1pt} that $\forall s \in[0,t ], \det(v^{-1} -2(s/t)
q_t ) \not= 0 \iff\forall s \in[0,t ], \det(v^{-1} -2 q_s ) \not=
0$ and
thus $\tilde{\mathcal{D}}_{x,\alpha,b,a; t} \cap
\nsing=\mathcal{D}_{b,a;t}\cap\nsing$. Let $v \in
\tilde{\mathcal{D}}_{x,\alpha,b,a; t}$. Since
$\tilde{\mathcal{D}}_{x,\alpha,b,a; t}$ is an open set, there is
$\varepsilon>0$ such that $v\pm\varepsilon I_d \in
\tilde{\mathcal{D}}_{x,\alpha,b,a; t}\cap\nsing$. Since $\mathcal
{D}_{b,a;t}$
is convex, $v=(v+\varepsilon I_d+v-\varepsilon I_d)/2 \in\mathcal{D}_{b,a;t}$.
\end{pf}

\subsection{\texorpdfstring{Proof of Proposition \protect\ref{Canonform}}{Proof of Proposition 5}}
\label{Proofprop5}
Once $u$ is given, the identity in law comes directly
from (\ref{AFFidentite2}). We now give a constructive proof of
the existence of $u$, which takes back the arguments given by Golub and Van
Loan (\cite{Golub}, Theorem 8.7.1). Nonetheless, we explain it
entirely since
it gives a practical way to get $u$.

Let us consider $\bar{\alpha}+a^Ta \in\posm$. From the extended
Cholesky decomposition given in Lemma~\ref{OuterProdDec} there is a
matrix $v
\in\nsing$ such that $v^T\bar{\alpha}v+v^Ta^Tav= I^r_d$, where
$r=\Rg(\bar{\alpha}+a^Ta)$. Since
$v^T\bar{\alpha}v \in\posm$, $v^Ta^Tav \in\posm$ and $z^T I^r_d
z=0$ for $z
\in\R^d$ such that $z_1=\cdots=z_r=0$, there are $s_1,s_2 \in
\mathcal{S}^+_{n}(\R)$ such that
\[
v^T\bar{\alpha}v =\pmatrix{
s_1 & 0\cr
0 & 0}
\quad\mbox{and}\quad v^Ta^Tav=\pmatrix{
s_2 & 0\cr
0 & 0}.
\]
Let $o_2$ be an orthogonal matrix such that $o_2^Ts_2o_2$ is a
diagonal matrix. We assume without loss of generality that only the first
$n$ elements of this diagonal are positive: $o_2^Ts_2o_2=\operatorname{diag}(\eta
_1,\ldots,\eta_n,0,\ldots,0)$. We set $o=
\bigl({o_2 \atop 0}\enskip{0 \atop I_{d-r}}\bigr)
$ and get $I^r_d=o^Tv^T\bar{\alpha}vo+o^Tv^Ta^Tavo, $ which
gives that $o^Tv^T\bar{\alpha}vo$ is a diagonal matrix. Thus, we get the
desired result by taking $u\!=\!\operatorname{diag}(\sqrt{\eta_1},\ldots,\sqrt{\eta
_n},1,\ldots,1) o^{-1} v^{-1}$.

\section{\texorpdfstring{Proofs of Section \lowercase{\protect\ref{sec_exact_wish}}}{Proofs of Section 2}}
\subsection{\texorpdfstring{Proof of Proposition \protect\ref{Permut}}{Proof of Proposition 8}}\label{proofPermut}
Let $X^x_t\sim \WIS_d(x,\alpha,0,I^n_d;t)$. We will check that for any
polynomial function $f$ of the matrix
elements, we have
$\E[f(X^x_t)]=\E[f(X^{n,\ldots^{X^{1,x}_t}}_t)]$. Let us consider a
polynomial function $f$ of degree $m$,
\[
x \in\symm,\qquad f(x)= \sum_{\gamma\in\N^{d(d+1)/2}, |\gamma| \le m}
a_\gamma\bar{x}^\gamma,
\]
where $|\gamma|=\sum_{1\le i \le j \le d} |\gamma_{\{i,j\}}|$ and
$\bar{x}^\gamma=\prod_{1\le i \le j \le d}
x_{\{i,j\}}^{\gamma_{\{i,j\}}}$. Since\vspace*{1pt} the operators are
affine, it is easy to check that $Lf(x)$ and $L_{e^i_d}f(x)$ are also
polynomial functions of degree $m$. We set
\[
\|f\|_{\mathbb P}=\sum_{\gamma\in\N^{d(d+1)/2}, |\gamma| \le m}
|a_\gamma| \quad\mbox{and}\quad |L|=\max_{\gamma\in\N^{d(d+1)/2}, |\gamma
| \le
m} \|L \bar{x}^\gamma\|_{\mathbb P},
\]
so that $\|L^k f\|_{\mathbb P} \le|L|^k \|f\|_{\mathbb P}$ for any $k
\in\N$. Therefore, the series $\sum_{k=0}^{\infty} t^k L^kf(x)/\break k!$ converges
absolutely. By using $l+1$ times It\^{o}'s formula, we get
\[
\E[f(X_t^{x})]=\sum_{k=0}^l\frac{t^{k}}{k!}L^{k}f(x)+\int_0^t\frac
{(t-s)^l}{l!}\E[L^{l+1}f(X_s^{x})]\,ds.
\]
Wishart processes have bounded moments since the drift and\break
diffusion
coefficients have a sublinear growth. Thus, $C=\break \max_{\gamma\in\N
^{d(d+1)/2}, |\gamma| \le
m} \sup_{s\in[0,t] }\E[|\overline{X}_s^{x^\gamma}|]< \infty$ and
we obtain that\break
$|{\int_0^t}\frac{(t-s)^l}{l!}\E[L^{l+1}f(X_s^{x})]\,ds|\le C\|f\|
_{\mathbb P}(t|L|)^{l+1}/(l+1)!
\underset{l\rightarrow+ \infty} {\rightarrow} 0 $. Thus, we have
$\E[f(X^x_t)]=\sum_{k=0}^{\infty} t^k L^kf(x)/k!$ and similarly we
get that
\[
\E\bigl[f\bigl(X^{n,\ldots^{X^{1,x}_t}}_t
\bigr)|X^{{n-1},\ldots^{X^{1,x}_t}}_t\bigr]=\sum_{k_n=0}^{+\infty}\frac
{t^{k_n}}{k_n!}L_{e^n_d}^{k_n}
f\bigl(X^{{n-1},\ldots^{X^{1,x}_t}}_t\bigr).
\]
Now, we remark that $\tilde{C}=\max_{\gamma\in\N^{d(d+1)/2},
|\gamma| \le
m} \sup_{s\in[0,t]
}\max(\E[|\overline{X}^{1,x^\gamma}_t|],\ldots,\break\E[|\overline
{X}^{n,\ldots^{X^{1,x^\gamma}_t}}_t|])<\infty$
by using once again that Wishart processes have bounded moments. Since
$\E[|L_{e^n_d}^{k_n}
f(X^{{n-1},\ldots^{X^{1,x}_t}}_t)|]\le\tilde{C} \|f\|_{\mathbb P}
|L_{e^n_d}|^{k_n}
$, we can switch the expectation with the series\vspace*{1pt} and
get (\ref{Exactcondexp}). Then, since $L_{e^n_d}^{k_n}f(x)$ are polynomial
function of degree $m$, we can iterate this argument and finally
get (\ref{ExactCauchyprod}), which gives the result.

\subsection{\texorpdfstring{Proof of Theorem \protect\ref{StructureDyn}}{Proof of Theorem 9}}
\label{proofStructureDyn}

The proof is divided into two parts. First, we prove that the
SDE (\ref{DiffStruMulti1}) has a unique strong solution which is given
by (\ref{explicitsolL1}) and is well defined on $\posm$. Second, we
show that its
infinitesimal generator is equal to the operator $L_{e^1_d}$ defined
in (\ref{OpL1}).

\textit{First step.} Let us assume that $(X^x_t)_{t\ge0}$ is a solution
to (\ref{DiffStruMulti1}). We use the matrix decomposition of
$(x_{i,j})_{2\le i,j\le d}$ given by (\ref{decompo1}) and set
\begin{eqnarray*}
(U_t)_{\{1,l+1 \}} &=& \sum
_{i=1}^r(c^{-1}_r)_{l,i}(X^x_t)_{\{1,i+1 \}},\qquad l\in\{ l,\ldots,r\}, \\
(U_t)_{\{1,1 \}}&=& (X^x_t)_{\{1,1 \}}-\sum
_{l=1}^r \Biggl(
\sum_{i=1}^{r}(c^{-1}_r)_{l,i}(X^x_t)_{\{1,i+1 \}}
\Biggr)^2\\
&=&(X^x_t)_{\{1,1 \}}-\sum_{l=1}^r\bigl((U_t)_{\{1,l+1
\}}\bigr)^2.
\end{eqnarray*}

We get by using Lemma~\ref{DemonstrationLemmaDesompo} that
\begin{eqnarray*}
\hspace*{-5pt}&&
\pmatrix{
1 & 0 &0 \cr
0 & c_r&0\cr
0 & k_r &I_{d-r-1}
}\\
\hspace*{-5pt}&&\quad{}\times
{\fontsize{10.35pt}{11pt}\selectfont{\pmatrix{\displaystyle
(U_t)_{\{1,1 \}}+\sum_{k=1}^r\bigl((U_t)_{\{1,k+1
\}}\bigr)^2 & \bigl((U_t)_{\{1,l+1 \}}\bigr)_{1\leq l \leq r}^T &
0\vspace*{2pt}\cr
\bigl((U_t)_{\{1,l+1 \}}\bigr)_{1\leq l \leq r} & I_{r}& 0\vspace*{2pt}\cr
0& 0& 0
}}}\\
\hspace*{-5pt}&&\quad{}\times
\pmatrix{
1 & 0 &0\cr
0 & c_r^T&k_r^T\cr
0 & 0 & I_{d-r-1}
}
\\
\hspace*{-5pt}&&\qquad={\fontsize{10.35pt}{11pt}\selectfont{\pmatrix{
\displaystyle (U_t)_{\{1,1 \}}+\sum_{k=1}^r\bigl((U_t)_{\{1,k+1
\}}\bigr)^2 & \bigl((U_t)_{\{1,l+1 \}}\bigr)_{1\leq l \leq r}^T
c_r^T & \bigl((U_t)_{\{1,l+1 \}}\bigr)_{1\leq l \leq r}^T k_r^T\vspace*{2pt}\cr
c_r \bigl((U_t)_{\{1,l+1 \}}\bigr)_{1\leq l \leq r} & c_rc_r^T &
c_rk_r^T\vspace*{2pt}\cr
k_r \bigl((U_t)_{\{1,l+1 \}}\bigr)_{1\leq l \leq r} & k_r c_r^T& 0}}
}\\
\hspace*{-5pt}&&\qquad= X^x_t.
\end{eqnarray*}
Since
\[
\pmatrix{
1 & 0 &0 \cr
0 & c_r&0\cr
0 & k_r &I_{d-r-1}
}
\]
is invertible, $X^x_t \in\posm$ if, and only if
%
\begin{eqnarray} \label{rangmat}
&&\forall z \in\R^d\nonumber\\
&&\qquad
z^T
{\fontsize{10.35pt}{11pt}\selectfont{\pmatrix{
\displaystyle (U_t)_{\{1,1 \}} + \sum_{i=1}^r\bigl((U_t)_{\{1,i+1
\}}\bigr)^2&\bigl((U_t)_{\{1,l \}}\bigr)_{2\leq l \leq r+1 }&0\vspace*{2pt}\cr
\bigl((U_t)_{\{ l,1 \}}\bigr)_{2\leq l \leq r+1 }&I_r&0\vspace*{2pt}\cr
0&0&0
}}}z \nonumber\\[-8pt]\\[-8pt]
&&\qquad\qquad =z_1^2(U_t)_{\{1,1 \}}+\sum
_{i=1}^r\bigl(z_{i+1}+(U_t)_{\{1,i+1 \}}z_1\bigr)^2\nonumber\\
&&\qquad\qquad\ge0 \quad\iff\quad(U_t)_{\{1,1 \}} \ge0. \nonumber
\end{eqnarray}
In particular, we get that $(U_0)_{\{1,1 \}}=u_{\{1,1\}}
\ge0$
since $x
\in\posm$. Now, by It\^{o} calculus, we get from (\ref
{DiffStruMulti1}) that
\[
d(U_t)_{\{1,l+1 \}} = \sum_{i=1}^r\sum_{k=1}^r (c^{-1}_r)_{l,i}
(c_r)_{i,k}\,dZ^{k+1}_t=
dZ_t^{l+1}
\]
and
\begin{eqnarray*}
d(U_t)_{\{1,1 \}}&=&(\alpha-r) \,dt + 2\sqrt
{(U_t)_{\{1,1 \}}}\,dW_t^1\\
&&{}+2\sum_{l=1}^r\sum
_{k=1}^r(c_r^{-1})_{l,k}(X_t)_{\{1,k+1 \}}\,dW_t^{l+1}\\
&&{}-\sum_{l=1}^r2\bigl((U_t)_{\{1,l+1 \}}\bigr)\,dW_t^{l+1}\\
&=& (\alpha-r) \,dt + 2\sqrt{(U_t)_{\{1,1 \}}}\,dW_t^1.
\end{eqnarray*}
Thus, the solution $(X^x_t)_{t\ge0}$ is necessarily the one given
by (\ref{explicitsolL1}) [pathwise uniqueness holds for
$((U^u_t)_{\{1,l \}})_{1\le l \le r+1}$, and especially
for the CIR diffusion
$(U^u_t)_{\{1,1 \}}$ since $\alpha\ge d-1\ge r$].
Reciprocally, it
is easy to check by It\^o calculus
that (\ref{explicitsolL1}) solves (\ref{DiffStruMulti1}).

\textit{Second step.} Now we want to show that $L_{e^1_d}$ is the infinitesimal
operator associated to the process $(X^x_t)_{t \geq0}$. It is
sufficient to
compare the drift and the quadratic covariation of the process $X_t^x$ with
$L_{e^1_d}$. Since the drift part of $(X^x_t)_{t \geq0}$ clearly
corresponds to the
first order of $L_{e^1_d}$, we study directly the quadratic part.
From (\ref{DiffStruMulti1}), we have for $i,j\in\{2,\ldots
,d\}^2$,
\begin{eqnarray*}
&&d\bigl\langle(X^x_t)_{\{1,1 \}},(X^x_t)_{\{1,1 \}
}\bigr\rangle\\
&&\qquad=
4\Biggl((X^x_t)_{\{1,1 \}}-\sum_{k=1}^r \Biggl[
\sum_{l=1}^{r}(c_r^{-1})_{k,l}(X^x_t)_{\{1,l+1 \}}\Biggr]^2
\\
&&\qquad\quad\hspace*{53.7pt}{}
+\sum_{k=1}^r \Biggl[ \sum_{l=1}^{r}(c_r^{-1})_{k,l}(X^x_t)_{\{1,l+1
\}}\Biggr] ^2\Biggr)\\
&&\qquad= 4 (X^x_t)_{\{1,1 \}} \,dt,\\
&&d\bigl\langle(X^x_t)_{\{1,i \}},(X^x_t)_{\{1,j \}
}\bigr\rangle\\
&&\qquad= \sum
_{k=1}^r(c_r)_{i-1,k}(c_r)_{j-1,k} \,dt
= (cc^T)_{i-1,j-1}\,dt\\
&&\qquad= (X^x_t)_{\{ i,j \}}\,dt, \\
&&d\bigl\langle(X^x_t)_{\{1,1 \}},(X^x_t)_{\{1,i \}
}\bigr\rangle\\
&&\qquad= 2
\sum_{k=1}^r\sum_{l=1}^r
(c_r)_{i-1,k}(c_r^{-1})_{k,l}(X^x_t)_{\{1,l+1 \}
}\,dt\\
&&\qquad=2(X^x_t)_{\{1,i \}}\,dt
\qquad\mbox{if } i \le r+1,
\end{eqnarray*}
\begin{eqnarray*}
&&d\bigl\langle(X^x_t)_{\{1,1 \}},(X^x_t)_{\{1,i \}
}\bigr\rangle\\
&&\qquad= 2\sum
_{k=1}^r\sum_{l=1}^r (k_r)_{i-1-r,k}(c_r^{-1})_{k,l}(X^x_t)_{\{1
,l+1 \}}\,dt \\
&&\qquad= 2\sum_{l=1}^r(k_r c_r^{-1})_{i-1-r,l}(X^x_t)_{\{1,l+1
\}}\,dt\\
&&\qquad= 2(X^x_t)_{\{1,i \}} \,dt \qquad\mbox{if } i > r+1\qquad \mbox{by
Lemma \ref
{DemonstrationLemmaDesompo}}.
\end{eqnarray*}
Thus, we deduce that $L_{e^1_d}$ is the infinitesimal generator of
$(X^x_t)_{t \geq0}$.
%
\begin{lemma}\label{DemonstrationLemmaDesompo}
Let $y\in\posm$. We set $r=\Rg((y_{i,j})_{2\leq i,j \leq d})$,
$y_1^{r}=(y_{1,i+1})_{ 1\leq i\leq r}$ and $y_1^{r,d}=(y_{1,{i+1}})_{
r+1\leq
i\leq d}$. We assume that there are an invertible matrix $c_r$ and a
matrix $k_r$ defined on $\mathcal{M}_{d-r-1 \times r}(\mathbb R)$, such
that
\[
(y_{i,j})_{2\leq i,j \leq d} = \pmatrix{
c_r & 0\cr
k_r & 0
} \pmatrix{
c_r^T & k_r^T\cr
0 & 0
}.
\]
Then, we have $y_1^{r,d}= k_r c_r^{-1} y_1^r$.
\end{lemma}
\begin{pf}
We set
\[
p=\left(
\begin{array}{c|c@{\quad}c}
1&0&0\\
\hline
0&c_r & 0\\
0&k_r & I_{d-r-1}
\end{array}
\right)
\quad
\mbox{and have}\quad
p^{-1}=\left(
\begin{array}{c|c@{\quad}c}
1&0&0\\
\hline
0&c^{-1}_r & 0\\
0&-k_rc^{-1}_r & I_{d-r-1}
\end{array}
\right).
\]
Since the matrix
\[
p^{-1}y(p^{-1})^T= \left(
\begin{array}{c|c@{\quad}c}
y_{1,1}& (c^{-1}_r y_1^r)^T & (y_1^{r,d}-k_r c_r^{-1} y_1^r)^T \\
\hline
c^{-1}_r y_1^r&I_r & 0\\
y_1^{r,d}-k_r c_r^{-1} y_1^r&0 & 0
\end{array}
\right)
\]
is positive semidefinite, we necessarily have $y_1^{r,d}-k_r c_r^{-1} y_1^r=0$.
\end{pf}

\section{\texorpdfstring{Proofs of Section \lowercase{\protect\ref{sec_high}}}{Proofs of Section 3}}\label{AppsecHigh}

\subsection{\texorpdfstring{Proof of Proposition \protect\ref{Flow}}{Proof of Proposition 14}}\label{AppFlow}

%
\begin{lemma}\label{Remarkable}
Let $(X_t^x)_{t \geq0}\underset{\mathit{Law}}{\sim} \WIS_d(x,\alpha,b,a)$ and
$v=v_R+iv_I$ such that $v_R \in\mathcal{D}_{b,a;t}$ and $v_I\in\symm
$. We denote by
$\phi(t,\alpha,x,v)$ the Laplace transform of $X_t^x$ given by (\ref
{CarWishart}), the other parameters
$a$, $b$ being fixed. Then, the derivative w.r.t. $x_{\{k,l\}}$
satisfies the equality
\[
\partial_{\{ k,l \}} \phi(t,\alpha,x,v) = \phi(t,\alpha+2,x,v)
p_t^{\{ k,l\}}(v),
\]
where $p_t^{\{ k,l\}}$ is a polynomial function of the
matrix elements of degree $d$ defined by
\begin{eqnarray*}
p_t^{\{ k,l\}}(v)&=& \Tr[v \adj(I_d-2q_tv)
m_t (e^{k,l}_d+ \mathbh{1}_{k \not= l}e^{l,k}_d) m_t^T]\\
&=&\!:\sum
_{\gamma\in\mathbb
N^{{d(d+1)}/{2}}, |\gamma|\leq d}a_t^{\gamma,{\{ k,l\}
}}\overline{v}^{\gamma},
\end{eqnarray*}
where
\[
\overline{v}^{\gamma}=
\prod_{\{i,j\}}v_{\{i,j\}}^{\gamma_{\{i,j\}}}.
\]

Moreover, its coefficients are bounded uniformly in time,
\[
\exists K_t>0, \forall s\in[0,t]\qquad \max_{\gamma\in\mathbb
N^{{d(d+1)}/{2}},|\gamma|\leq d}\bigl(\bigl|a_s^{\gamma,{\{ k,l
\}}}\bigr|\bigr)\leq K_t.
\]
\end{lemma}
\begin{pf}
We get from (\ref{CarWishart})
\begin{eqnarray*}
\partial_{\{ k,l \}}\phi(t,\alpha,x,v)&=&\frac{\Tr
[v \adj(I_d-2q_tv)
m_t(e^{k,l}_d+ \mathbh{1}_{k \not= l}e^{l,k}_d) m_t^T]}{\det
(I_d-2q_tv)}\\
&&{}\times\frac{\exp(\operatorname{Tr}
[v(I_d-2q_tv)^{-1}m_txm_t^T])}{\det(I_d-2q_tv)^{
{\alpha}/{2}}} \\
&=&\phi(t,\alpha+2,x,v)\operatorname{Tr}[v \operatorname
{adj}(I_d-2q_tv)
m_t(e^{k,l}_d+ \mathbh{1}_{k \not= l}e^{l,k}_d) m_t^T].
\end{eqnarray*}
Since $s \mapsto\|m_s\|$ and $s \mapsto\|q_s\|$ are continuous
functions on
$[0,t]$, we obtain the bounds on the polynomial coefficients.
\end{pf}
\begin{pf*}{Proof of Proposition~\ref{Flow}}
Let $f \in\Cpolde{\symm}$. First, let us observe that (\ref
{derflowform}) is obvious
when $l=|n|=0$. Since we have $\forall l\in\mathbb
N, L^lf \in\Cpolde{\symm}$, and
$\partial_t^l \tilde{u}(t,x)=\E(L^lf(X_t^x))$, it is sufficient to
prove (\ref{derflowform}) only for the derivatives
w.r.t.~$x$.

We first focus on the case $|n|=1$ and want to show that $ \partial
_{\{ k,l \}}
\tilde u (t,x)$ satisfies~(\ref{derflowform}).
The sketch of this proof is to write $f$ as the inverse Fourier
transform of
its Fourier transform and then use Lemma~\ref{Remarkable}.
Unfortunately, $f$
has not a priori the required integrability to do that, and we have to
introduce an auxiliary function $f_\rho$.

\textit{Definition of the new function $f_{\rho}$.}
Since $\mathcal{D}_{b,a;T}$ given by (\ref{DomLapWis}) is an open
set and
$0\in\mathcal{D}_{b,a;T}$, there is $\rho>0$ such that $\rho I_d
\in\mathcal{D}_{b,a;T}$.
Let
$\mu\dvtx\mathbb R \rightarrow\mathbb R $ be the function such that
$ \mu(x)=0 $ if $ x\leq-1$ or $x\geq0$, $\mu(x)=\exp(\frac
{1}{x(x+1)})$ if $-1<x<0$. We have $\mu\in\mathcal{C}^{\infty
}(\mathbb R)$.

Then we consider he cutoff function $\zeta\dvtx\mathbb R \rightarrow
\mathbb R \in\mathcal{C}^{\infty}(\mathbb R)$ defined as
$\forall x \in\mathbb R$, $\zeta(x)=
\frac{\int_{-\infty}^x\mu(y)\,dy}{\int_{\mathbb R}\mu(y)\,dy}$. It is
nondecreasing, such that $0\le\zeta(x)\le1$, $\zeta(x)=0$ if $x\le
-1$ and
$\zeta(x)=1$ if $x\ge0$. Besides, we have $\zeta\in\Cpolde{\R}$
since all
its derivatives have a compact support. Now, we define a $\vartheta
\in\Cpolde{\symm} $ as
\[
\vartheta\dvtx\symm\rightarrow\mathbb R,
\qquad x\mapsto\prod_{i=1}^d\zeta\bigl(x_{\{ i,i \}}\bigr) \prod
_{i \neq
j} \zeta\bigl(x_{\{ j,j \}}x_{\{ i,i \}}-x_{\{
i,j \}}^2\bigr).
\]
It is important to notice that $0 \le\vartheta\le1$, $\vartheta(x)
=1$ if
$x \in\posm$ and $\vartheta(x) =0$ if there is $i \in\{
1,\ldots,d\}$ such
that $x_{\{ i,i \}}<-1$ or $i<j \in\{1,\ldots
,d\}$ such that\vspace*{1pt} $x_{\{ i,j \}}^2>1+x_{\{ i,i
\}}x_{\{ i,i \}}$.
Let $\gamma\in\N^{d(d-1)/2}$. Since $f \in\Cpolde{\symm}$, there
are constants $K,E>0$ and $K',E'>0$ such that, $\forall x \in\symm$
\begin{eqnarray*}
|\partial^\gamma(\vartheta f)(x)|&\leq& K(1+\|x\|^E)\prod
_{i=1}^d \bigl(
1_{\{|x_{\{ i,i \}}|>-1\}}\bigr)\prod_{1\leq
i < j\leq
d}\bigl(\mathbh1_{\{ x_{\{ i,j \}}^2\leq1+
x_{\{ i,i \}}x_{\{ j,j \}}\}}\bigr) \\[-2pt]
&\leq& K'\bigl(1+\bigl\|\bigl(x_{\{ i,i \}}\bigr)_{1\leq i\leq d}\bigr\|
^{E_1}\bigr)\\[-2pt]
&&{}\times\prod
_{i=1}^d\bigl( 1_{\{|x_{\{ i,i \}}|>-1\}
}\bigr)\prod
_{1\leq i < j\leq d}\bigl(\mathbh1_{\{ x_{\{ i,j \}
}^2\leq1+
x_{\{ i,i \}}x_{\{ j,j \}}\}}\bigr).
\end{eqnarray*}
Here, the upper bound only involves the diagonal coefficients. We define
\[
x \in\symm,\qquad f_{\rho}(x):= \vartheta(x)
{f}(x)\exp(-\operatorname{Tr}(\rho x))
\]
and obtain from the last inequality that $f_\rho$ belongs to the Schwartz
space of rapidly decreasing functions since $\rho>0$. Thus, its
Fourier transform also belongs
to the Schwartz space and we have
\[
f_\rho(x)=\frac{1}{(2\pi)^{{d(d+1)}/{2}}}\int_{\mathbb
R^{{d(d+1)}/{2}}}\exp(-\Tr(ivx))\mathcal{F}({f}_{\rho})(v)\,dv,
\]
where
\[
\mathcal{F}({f}_{\rho})(v)= \int_{\mathbb
R^{{d(d+1)}/{2}}}\exp(\Tr(ivx)) f_\rho(x)\,dx
\]
and, in particular, $f_\rho, \mathcal{F}(f_\rho) \in L^1(\symm)
\cap L^\infty(\symm)$.

\textit{A new representation of $\tilde{u}(t,x)$.}
We have $f(x)=\exp( \rho\Tr( x)) f_\rho(x)$ for $x \in\posm$, and
therefore
\begin{eqnarray*}
\tilde u(t,x) &=&\mathbb E[\exp(\Tr(\rho X_t^{x})){f}_{\rho
}(X_t^{x})]\\[-2pt]
&=& \frac{1}{(2\pi)^{{d(d+1)}/{2}}}\mathbb E \biggl[\int
_{\mathbb
R^{{d(d+1)}/{2}}}\exp\bigl(\operatorname{Tr}[(-iv+\rho I_d)X_t^{x}]\bigr)
\mathcal{F}({f}_{\rho})(v)\,dv
\biggr]\\[-2pt]
&=&\frac{1}{(2\pi)^{{d(d+1)}/{2}}}\int_{\mathbb
R^{{d(d+1)}/{2}}}\mathbb
E\bigl[\exp\bigl(\Tr[(-iv+\rho I_d)X_t^{x}]\bigr)\bigr]
\mathcal{F}({f}_{\rho})(v)\,dv.
\end{eqnarray*}
The last equality holds since
\begin{eqnarray*}
&&
\int_{\mathbb R^{{d(d+1)}/{2}}}
\bigl| \E\bigl[\exp\bigl(\Tr[(-iv+\rho I_d)X_t^x]\bigr)\bigr]\bigr|
|\mathcal{F}({f}_{\rho
})(v)|\,dv \\[-2pt]
&&\qquad\le
\phi(t,\alpha,x,\rho I_d)\|\mathcal{F}({f}_{\rho})\|_{1} <\infty.
\end{eqnarray*}
Here we
have used that $\rho I_d \in\mathcal{D}_{b,a;T}$ to get
$\phi(t,\alpha,x,\rho I_d)< \infty$.

\textit{Derivation with respect to $x_{\{ k,l\}}$, $ k,l
\in\{1,\ldots,d \}$.}
From Lemma~\ref{Remarkable}, we have by Lebesgue's theorem
%
\begin{eqnarray}\label{derutil}
\partial_{\{ k,l \}} \tilde u (t,x)&=& \frac{1}{(2\pi)^{{d(d+1)}/{2}}}\int_{\mathbb
R^{{d(d+1)}/{2}}}\phi(t,\alpha+2,x,-iv+\rho I_d)\nonumber\\[-4pt]\\[-12pt]
&&\hspace*{132pt}{}\times p_t^{\{ k,l\}}(\rho I_d-iv) \mathcal{F}({f}_{\rho})(v)\,dv\nonumber
\end{eqnarray}
since
$
| \partial_{\{ k,l \}}^x\phi(t,\alpha,x,-iv+\rho I_d)
\mathcal
{F}({f}_{\rho})(v)|
\leq|\phi(t,\alpha+2,x,\rho I_d)||p_t^{\{ k,l \}}(\rho
I_d-\break iv) \mathcal{F}({f}_{\rho})(v)|
$ and $p_t^{\{ k,l \}}(\rho I_d- iv) \mathcal{F}({f}_{\rho
})(v)$ is
a rapidly decreasing function.

Let $1 \le k',l' \le d$. An integration by part gives
$ \int_{\mathbb R} (\rho I_d-iv)_{\{ k',l' \}} \exp
(\Tr[x(iv-\rho
I_d)] ) \vartheta(x) f(x) \,dx_{\{ k',l' \}} =(\frac
{\mathbh1_{k'
\neq l'}}{2} + \mathbh1_{k'=l'})\int_{\mathbb R} \exp{( \Tr
[x(iv-\rho
I_d)])}\partial_{\{ k',l' \}}(\vartheta
(x)\*f(x))\,dx_{\{ k',l' \}}, $
and thus
\begin{eqnarray*}
&&(\rho I_d-iv)_{\{ k',l' \}}\mathcal{F}(\exp[-\rho
\Tr(x)]\vartheta(x)f(x))(v) \\
&&\qquad= \biggl(\frac{\mathbh1_{k' \neq l'}}{2} +
\mathbh
1_{k'=l'}\biggr)
\mathcal{F}\bigl(\exp[-\rho\Tr(x)]\partial_{\{ k',l'
\}}[\vartheta(x)f(x)]\bigr)(v).
\end{eqnarray*}
We set $\varphi(\gamma)=\prod_{1\leq k' \leq l' \leq
d}(\frac{\mathbh1_{k'\neq l'}}{2} + \mathbh1_{k'=l'})^{\gamma_{\{
k',l'\}}} $
for $\gamma\in\N^{d(d+1)/2}$ and get by iterating the argument that
%
\begin{eqnarray}\label{condl1}
&&
\prod_{1\leq
k'\leq l'\leq d}(\rho I_d-iv)_{\{ k',l' \}}^{\gamma
_{\{ k',l'
\}}} \mathcal{F}({f}_{\rho})(v) \nonumber\\[-8pt]\\[-8pt]
&&\qquad= \varphi(\gamma)
\mathcal{F}
\bigl(\exp[- \rho\Tr(x )] \partial_\gamma(\vartheta
\times f)
(x) \bigr)(v).\nonumber
\end{eqnarray}
Since $p_t^{\{ k,l\}}(\rho
I_d-iv)=\sum_{\gamma\in\mathbb N^{{d(d+1)}/{2}}, |\gamma|\leq
d}a_t^{\gamma,{\{ k,l\}}} \prod_{1\leq
k'\leq l'\leq d}(\rho I_d-iv)_{\{ k',l' \}}^{\gamma
_{\{ k',l'
\}}}$, we get from (\ref{derutil}) and (\ref{condl1})
%
\begin{eqnarray}\label{derut}
\partial_{\{ k,l \}}u(t,x) &=& \sum_{|\gamma|\leq
d}a_t^{\gamma,{\{
k,l\}
}} \varphi(\gamma)
\mathbb E\bigl(\partial_\gamma{(f \times
\vartheta)}(Y_t^{x})\bigr)\nonumber\\[-8pt]\\[-8pt]
&=& \sum_{|\gamma|\leq
d}a_t^{\gamma,{\{ k,l\}}} \varphi(\gamma) \E
(\partial_\gamma f(Y_t^{x})),\nonumber
\end{eqnarray}
where $(Y_t^{x})_{t \geq0}
\underset{\mathrm{Law}}{\sim}\WIS_d(x,\alpha+2,b,a)$. Here we have used
that $\partial_\gamma(\vartheta\times f)(y)=\partial_\gamma f(y)$ for
$y \in\posm$. From\vspace*{2pt} Lemma~\ref{Remarkable} $(a^{\gamma,\{
k,l\}}_t)_{\gamma\in\mathbb N^{{d(d+1)}/{2}},|\gamma|\leq d}$ is
bounded for $t\in[0,T]$, and we get (\ref{derflowform}) when $|n|=1$
since $\partial_\gamma f \in \Cpolde{\symm}$. Thanks to (\ref{derut}),
a derivative of order $|n|$, can be seen as a (bounded) linear
combination of derivatives of order $|n|-1$, and we easily get
(\ref{derflowform}) by an induction on $|n|$.

It remains to check that we have indeed $\partial_t\tilde
{u}(t,x)=Lu(t,x)$. Let $t,h>0$. By the Markov
property, we have $\tilde{u}(t+h,x)=\E[\tilde{u}(t,X^x_h)]$.
From (\ref{derflowform}) and It\^{o}'s formula, we get
$[\tilde{u}(t+h,x)-u(t,x)]/h \underset{h \rightarrow0^+}
{\rightarrow}
Lu(t,x)$.
\end{pf*}
%
\begin{lemma}\label{MBCAFFINEODE}
Let $\alpha,x\in\posm$, $B \in\mathcal{L}(\posm)$ that
satisfies (\ref{ReQuiredAssumption}), and $x(t)$ be the solution of
the ODE
%
\begin{equation}\label{ODEMom}
x(t) = x + \int_0^t\bigl(\alpha+ B(x(s))\bigr)\,ds.
\end{equation}
Then we have $x(t) \in\posm$ for $t\ge0$.
\end{lemma}
\begin{pf}
The ODE (\ref{ODEMom}) is affine and has unique solution on $\posm$
which is given by
%
\begin{equation}\label{ODEexplicitesol}
t \ge0,\qquad x(t)=\exp(tB)(x)+\int_0^t\exp(s B)(\alpha
)\,ds,
\end{equation}
where\vspace*{-2pt} $\forall t \in\mathbb R^+, \forall x \in\symm, \exp
(tB)(x)=\sum_{k=0}^{\infty}\frac{t^k B^k(x)}{k!}$, $B^k(x) =
\underbrace{B \circ\cdots\circ B}_{k\ \mathrm{times}}(x)$ such that
$B^0(x)=x$.

We first assume that $\alpha,x \in\dpos$
and consider $\tau= \inf\{ t \geq0, x(t) \notin\posm\}$,
with the convention $\inf\varnothing=+\infty$. We have $\tau> 0$.
Let us
assume by a way of contradiction that $\tau< \infty$. Then $x(\tau)$
cannot be invertible and there is $y \in\posm$ such that $y \not= 0$ and
$\Tr(yx(\tau))=0$. From (\ref{ODEexplicitesol})
and (\ref{ReQuiredAssumption}), we get
\[
\Tr(x'(\tau)y) =\Tr\bigl(
[B(x(\tau))+\alpha] y \bigr)>0,
\]
since $\alpha$ is positive definite. Therefore, there is
$\epsilon\in(0,\tau)$ such that $\Tr(y x(\tau-\epsilon))<0$. Let
us now recall
that $z \in\posm\iff\forall y \in\posm, \Tr(yz) \ge0$. Thus,
$x(\tau-\epsilon) \notin\posm$, which contradicts the definition
of $\tau$.

In the general case $\alpha,x \in\posm$, we observe that the
solution (\ref{ODEexplicitesol}) is continuous w.r.t. $x$ and
$\alpha$, and
thus $\forall t \ge0, x(t) \in\posm$ since $\posm$ is a closed set.
\end{pf}

\subsection{\texorpdfstring{Proof of Proposition \protect\ref{propweaknu}}{Proof of Proposition 17}}
\label{Apppropweaknu}
First, let us check that $\theta_t \in\nsing$ is well defined, such
that $q_t/t=\theta_t I^n_d \theta_t^T$ and satisfies
%
\begin{equation}\label{boundtheta}
\exists K, \eta>0, \forall t \in(0,\eta)\qquad \max(\|\theta_t \|,\|
\theta_t \|^{-1})
\le K.
\end{equation}
When $n=d$, $q_t/t$ is definite positive as a convex combination of definite
positive matrices and the usual Cholesky decomposition is well defined.
Moreover, (\ref{boundtheta}) holds since $q_t/t$ goes to $a^Ta$
which is invertible when $t\rightarrow0^+$. When
$n<d$, we have assumed, in addition, that $b$ and $a^Ta$ commute. Therefore,
$q_t=a^Ta (\int_0^t \exp(sb) \exp(sb^T)\,ds/t)$. Since $a^T a$ and
$(\int_0^t
\exp(sb) \exp(sb^T)\,ds/t)$ are positive semidefinite matrices that
commute, we
have
\[
q_t=\sqrt{\frac{1}{t}\int_0^t \exp(sb) \exp(sb^T) \,ds }\,a^Ta
\sqrt{\frac{1}{t}\int_0^t \exp(sb) \exp(sb^T) \,ds }.
\]
Once again, $\frac{1}{t}\int_0^t \exp(sb) \exp(sb^T) \,ds$ is
definite positive as a convex combination of definite
positive matrices and we get that $\theta_t\!=\!\sqrt{\frac{1}{t}\int
_0^t \!\exp(sb) \exp(sb^T) \,ds } \*p^{-1}
\bigl({c_n \atop k_n}\enskip{0 \atop I_{d-n}}\bigr)
\in\nsing$ satisfies $q_t/t=\theta_t I^n_d \theta_t^T$
by Lemma~\ref{OuterProdDec}. Similarly, (\ref{boundtheta}) holds
since $ p^{-1} \bigl({c_n \atop k_n}\enskip{0 \atop I_{d-n}}\bigr)$
does not depend on $t$ and $\sqrt{\frac{1}{t}\int_0^t
\exp(sb) \exp(sb^T) \,ds }$ goes to $I_d$ when $t\rightarrow0^+$.

Let $f \in\Cpolde{\posm}$. Let $X^x_t \sim \WIS_d(x,\alpha,b,a;t)$.
Since the
exact scheme is a potential $\nu$th-order scheme, there are constants
$C,E,\eta>0$ depending only on a good sequence of $f$ such that
%
\begin{equation}\label{Propschnustep1}
\forall t \in(0, \eta)\qquad \Biggl|\E[f(X^x_t)]-\sum
_{k=0}^{\nu} \frac{t^k}{k!}
L^k f(x)\Biggr|\le C t^{\nu+1}(1+\|x\|^E).
\end{equation}
On the other hand, we have from Proposition~\ref{Propidloiwis},
%
\begin{eqnarray}\label{Propschnustep2}
&&\E[f(\hat{X}^x_t)]-\E[f(X^x_t)]\nonumber\\[-8pt]\\[-8pt]
&&\qquad=\E\bigl[f\bigl(\theta_t \hat
{Y}^{\theta_t^{-1}
m_txm_t^T(\theta_t^{-1})^T}_t \theta_t^T\bigr)\bigr]-\E\bigl[f\bigl(\theta_t Y^{\theta_t^{-1}
m_txm_t^T(\theta_t^{-1})^T}_t \theta_t^T\bigr)\bigr].\nonumber
\end{eqnarray}
Let us introduce
$f_{\theta_t}(y):= f( \theta_t y \theta_t^T) \in\Cpolde{\posm}$.
By the chain
rule, we have $\partial_{\{i,j\}}f_{\theta_t}(y)=\Tr[ \theta_t (e^{i,j}_d
+\mathbh{1}_{i \not= j} e^{j,i}_d)
\theta_t^T \partial f (\theta_t y \theta_t^T)]$, where $(\partial f
(x))_{k,l}=\break(\mathbh{1}_{k=l} + \frac{1}{2} \mathbh{1}_{k \not
=l})\partial_{\{k,l\}} f (x)$ and
$e^{i,j}_d=(\mathbh{1}_{k=i,l=j})_{1\le k,l\le d}$. From (\ref
{boundtheta}), we see that
there is a good sequence $(C_\gamma,e_\gamma)_{\gamma\in\N
^{d(d+1)/2}}$ that
can be obtained from a good sequence of $f$ such that
\[
\forall t \in(0, \eta), \forall y \in\posm\qquad |\partial_\gamma
f_{\theta_t}(y)|\le C_\gamma
(1+ \|y \|^{e_\gamma}).
\]
Therefore, we get that there are constants still denoted by $C,E,\eta
>0$ such
that
%
\begin{eqnarray}\label{Propschnustep3}
&&
\forall t \in(0, \eta)\nonumber\\
&&\qquad\bigl| \E\bigl[f\bigl(\theta_t \hat {Y}^{\theta_t^{-1}
m_txm_t^T(\theta_t^{-1})^T}_t \theta_t^T\bigr)\bigr]-\E\bigl[f\bigl(\theta_t
Y^{\theta_t^{-1} m_txm_t^T(\theta_t^{-1})^T}_t \theta_t^T\bigr)\bigr]\bigr|\\
&&\qquad\qquad\le C
t^{\nu+1} \bigl(1+ \|\theta_t^{-1} m_txm_t^T(\theta_t^{-1})^T
\|^E\bigr).\nonumber
\end{eqnarray}
From (\ref{boundtheta}), we get that there is a constant $K'>0$ such
that $\|\theta_t^{-1}
m_tx\*m_t^T(\theta_t^{-1})^T \|^E \le K'\|x \|^E$ for $t \in(0, \eta)$.
Thus, we get the result by gathering (\ref{Propschnustep1}),
(\ref{Propschnustep2}) and (\ref{Propschnustep3}).

\subsection{\texorpdfstring{Proof of Proposition \protect\ref{squareOU}}{Proof of Proposition 21}}
\label{AppsquareOU}

We have, by using It\^{o} calculus, $dX^x_t=(c+W_tI^n_d)^T\,dW_tI^n_d+ I^n_d
\,dW_t^T(c+W_tI^n_d)+dI^n_d \,dt$. By using Lemma~\ref{Invariantvol}, the
quadratic covariation of $(X^x_t)_{i,j}$ and $(X^x_t)_{m,n}$ is given by
$d\langle(X^x_t)_{i,j},(X^x_t)_{m,n} \rangle=(X^x_t)_{i,m}
(I^n_d)_{j,n}+(X^x_t)_{i,n} (I^n_d)_{j,m}+(X^x_t)_{j,m}
(I^n_d)_{i,n}+(X^x_t)_{j,n} (I^n_d)_{i,m}$. Therefore, $(X^x_t)_{t \ge0}$
solves the same martingale problem as $\WIS_d(x,d,0,I^n_d)$, which is
known to
have a unique solution from Cuchiero et al.~\cite{Teichmann}.

Let us now show that $\hat{X}^x_t$ is a potential second-order scheme.
We can
see $c+\sqrt{t}\hat{G}I^n_d$ as the Ninomiya--Victoir scheme with
moment-matching variables (see~\cite{Alfonsi}, Theorem 1.18)
associated to
$\frac{1}{2}\sum_{i=1}^d \sum_{j=1}^n \partial_{i,j}^2$ on
$\mathcal{M}_d(\R)$. Let $f \in\Cpolde{\posm}$. Then, $x
\in\mathcal{M}_d(\R)
\mapsto f(x^Tx) \in\Cpolde{\mathcal{M}_d(\R)}$ and there are constants
$C,E,\eta>0$ depending only on a good sequence of $f$ such that
\begin{eqnarray*}
&&
\forall t \in(0, \eta)\\
&&\qquad
\bigl|\E\bigl[f\bigl(\bigl(c+\sqrt{t}\hat{G}I^n_d\bigr)^T\bigl(c+\sqrt{t}\hat{G}I^n_d\bigr)\bigr)\bigr]
-\E
\bigl[f\bigl((c+W_tI^n_d)^T(c+W_tI^n_d)\bigr)\bigr]\bigr|\\
&&\qquad\qquad\le
C t^{\nu+1} (1+\|c\|^E).
\end{eqnarray*}
Let us now observe that the Frobenius norm of $c$ is $\sqrt{\Tr(c^T
c)}=\sqrt{\Tr(x)}\le\sqrt{d+\Tr(x^2)} \le\sqrt{d}+\sqrt{\Tr
(x^2)} $. Therefore, for any norm, there is a constant $K>0$ such that
$\|c \| \le
K(1+ \|x\|)$, which gives the result.
\end{appendix}



\printaddresses


\begin{thebibliography}{30}

\bibitem{Alfonsi2}
\begin{barticle}[mr]
\bauthor{\bsnm{Alfonsi},~\bfnm{Aur{\'e}lien}\binits{A.}}
(\byear{2005}).
\btitle{On the discretization schemes for the {CIR} (and {B}essel squared)
  processes}.
\bjournal{Monte Carlo Methods Appl.}
\bvolume{11}
\bpages{355--384}.
\bid{doi={10.1163/156939605777438569}, issn={0929-9629}, mr={2186814}}
\bptok{imsref}%
\end{barticle}
\endbibitem

\bibitem{Alfonsi}
\begin{barticle}[mr]
\bauthor{\bsnm{Alfonsi},~\bfnm{Aur{\'e}lien}\binits{A.}}
(\byear{2010}).
\btitle{High order discretization schemes for the {CIR} process: Application to
  affine term structure and {H}eston models}.
\bjournal{Math. Comp.}
\bvolume{79}
\bpages{209--237}.
\bid{doi={10.1090/S0025-5718-09-02252-2}, issn={0025-5718}, mr={2552224}}
\bptok{imsref}%
\end{barticle}
\endbibitem

\bibitem{Harry}
\begin{bmisc}[auto:STB|2012/06/07|11:40:19]
\bauthor{\bsnm{{B}enabid},~\bfnm{A.}\binits{A.}},
  \bauthor{\bsnm{{B}ensusan},~\bfnm{H.}\binits{H.}} \AND
  \bauthor{\bparticle{{E}l} \bsnm{{K}aroui},~\bfnm{N.}\binits{N.}}
(\byear{2010}).
\bhowpublished{Wishart stochastic volatility: Asymptotic smile and numerical
  framework. Preprint}.
\bptok{imsref}%
\end{bmisc}
\endbibitem

\bibitem{Bruthesis}
\begin{bmisc}[auto:STB|2012/06/07|11:40:19]
\bauthor{\bsnm{Bru},~\bfnm{M.~F.}\binits{M.~F.}}
(\byear{1987}).
\bhowpublished{Th\`ese 3\`eme cycle. R\'esistence d'Escherichie
  coli aux antibiotiques. Sensibilit\'es des analyses en composantes
  principales aux perturbations Browniennes et simulation. Ph.D. thesis, Univ.
  Paris Nord}.
\bptok{imsref}%
\end{bmisc}
\endbibitem

\bibitem{Bru}
\begin{barticle}[mr]
\bauthor{\bsnm{Bru},~\bfnm{Marie-France}\binits{M.-F.}}
(\byear{1991}).
\btitle{Wishart processes}.
\bjournal{J. Theoret. Probab.}
\bvolume{4}
\bpages{725--751}.
\bid{doi={10.1007/BF01259552}, issn={0894-9840}, mr={1132135}}
\bptok{imsref}%
\end{barticle}
\endbibitem

\bibitem{CM}
\begin{barticle}[auto:STB|2012/06/07|11:40:19]
\bauthor{\bsnm{Carr},~\bfnm{P.}\binits{P.}} \AND
  \bauthor{\bsnm{Madan},~\bfnm{A.}\binits{A.}}
(\byear{1999}).
\btitle{Option pricing and the fast Fourier transform}.
\bjournal{J.~Comput. Finance}
\bvolume{2}
\bpages{61--73}.
\bptok{imsref}%
\end{barticle}
\endbibitem

\bibitem{Teichmann}
\begin{barticle}[mr]
\bauthor{\bsnm{Cuchiero},~\bfnm{Christa}\binits{C.}},
  \bauthor{\bsnm{Filipovi{\'c}},~\bfnm{Damir}\binits{D.}},
  \bauthor{\bsnm{Mayerhofer},~\bfnm{Eberhard}\binits{E.}} \AND
  \bauthor{\bsnm{Teichmann},~\bfnm{Josef}\binits{J.}}
(\byear{2011}).
\btitle{Affine processes on positive semidefinite matrices}.
\bjournal{Ann. Appl. Probab.}
\bvolume{21}
\bpages{397--463}.
\bid{doi={10.1214/10-AAP710}, issn={1050-5164}, mr={2807963}}
\bptnote{check year}%
\bptok{imsref}%
\end{barticle}
\endbibitem

\bibitem{Dafonseca}
\begin{barticle}[auto:STB|2012/06/07|11:40:19]
\bauthor{\bsnm{Da~Fonseca},~\bfnm{J.}\binits{J.}},
  \bauthor{\bsnm{Grasselli},~\bfnm{M.}\binits{M.}} \AND
  \bauthor{\bsnm{Tebaldi},~\bfnm{C.}\binits{C.}}
(\byear{2008}).
\btitle{Option pricing when correlations are stochastic: An analytical
  framework}.
\bjournal{Review of Derivatives Research}
\bvolume{10}
\bpages{151--180}.
\bptok{imsref}%
\end{barticle}
\endbibitem

\bibitem{Filipovic}
\begin{bbook}[mr]
\bauthor{\bsnm{Filipovi{\'c}},~\bfnm{Damir}\binits{D.}}
(\byear{2009}).
\btitle{Term-Structure Models: A Graduate Course}.
\bpublisher{Springer}, \baddress{Berlin}.
\bid{doi={10.1007/978-3-540-68015-4}, mr={2553163}}
\bptok{imsref}%
\end{bbook}
\endbibitem

\bibitem{GP}
\begin{bmisc}[auto:STB|2012/06/07|11:40:19]
\bauthor{\bsnm{Gauthier},~\bfnm{P.}\binits{P.}} \AND
  \bauthor{\bsnm{Possamai},~\bfnm{D.}\binits{D.}}
(\byear{2009}).
\bhowpublished{Efficient simulation of the Wishart model. SSRN eLibrary}.
\bptok{imsref}%
\end{bmisc}
\endbibitem

\bibitem{Glasserman}
\begin{bbook}[mr]
\bauthor{\bsnm{Glasserman},~\bfnm{Paul}\binits{P.}}
(\byear{2004}).
\btitle{Monte {C}arlo Methods in Financial Engineering: Stochastic Modelling
  and Applied Probability}.
\bseries{Applications of Mathematics (New York)}
\bvolume{53}.
\bpublisher{Springer}, \baddress{New York}.
\bid{mr={1999614}}
\bptok{imsref}%
\end{bbook}
\endbibitem

\bibitem{Gleser}
\begin{barticle}[mr]
\bauthor{\bsnm{Gleser},~\bfnm{Leon~Jay}\binits{L.~J.}}
(\byear{1976}).
\btitle{A canonical representation for the noncentral {W}ishart distribution
  useful for simulation}.
\bjournal{J. Amer. Statist. Assoc.}
\bvolume{71}
\bpages{690--695}.
\bid{issn={0162-1459}, mr={0433717}}
\bptok{imsref}%
\end{barticle}
\endbibitem

\bibitem{Golub}
\begin{bbook}[mr]
\bauthor{\bsnm{Golub},~\bfnm{Gene~H.}\binits{G.~H.}} \AND
  \bauthor{\bsnm{Van~Loan},~\bfnm{Charles~F.}\binits{C.~F.}}
(\byear{1996}).
\btitle{Matrix Computations}, \bedition{3rd} ed.
\bpublisher{Johns Hopkins Univ. Press}, \baddress{Baltimore, MD}.
\bid{mr={1417720}}
\bptok{imsref}%
\end{bbook}
\endbibitem

\bibitem{Gourieroux}
\begin{bmisc}[auto:STB|2012/06/07|11:40:19]
\bauthor{\bsnm{Gourieroux},~\bfnm{C.}\binits{C.}} \AND
  \bauthor{\bsnm{Sufana},~\bfnm{R.}\binits{R.}}
(\byear{2003}).
\bhowpublished{Wishart quadratic term structure models. Working paper.}
\bptok{imsref}%
\end{bmisc}
\endbibitem

\bibitem{Tebaldi}
\begin{barticle}[mr]
\bauthor{\bsnm{Grasselli},~\bfnm{Martino}\binits{M.}} \AND
  \bauthor{\bsnm{Tebaldi},~\bfnm{Claudio}\binits{C.}}
(\byear{2008}).
\btitle{Solvable affine term structure models}.
\bjournal{Math. Finance}
\bvolume{18}
\bpages{135--153}.
\bid{doi={10.1111/j.1467-9965.2007.00325.x}, issn={0960-1627}, mr={2380943}}
\bptok{imsref}%
\end{barticle}
\endbibitem

\bibitem{Heston}
\begin{barticle}[auto:STB|2012/06/07|11:40:19]
\bauthor{\bsnm{Heston},~\bfnm{S.}\binits{S.}}
(\byear{1993}).
\btitle{A closed-form solution for options with stochastic volatility with
  applications to bond and currency options}.
\bjournal{The Review of Financial Studies}
\bvolume{6}
\bpages{327--343}.
\bptok{imsref}%
\end{barticle}
\endbibitem

\bibitem{Kabe}
\begin{barticle}[mr]
\bauthor{\bsnm{Kabe},~\bfnm{D.~G.}\binits{D.~G.}}
(\byear{1964}).
\btitle{A note on the {B}artlett decomposition of a {W}ishart matrix}.
\bjournal{J.~Roy. Statist. Soc. Ser. B}
\bvolume{26}
\bpages{270--273}.
\bid{issn={0035-9246}, mr={0171358}}
\bptok{imsref}%
\end{barticle}
\endbibitem

\bibitem{Kshirsagar}
\begin{barticle}[mr]
\bauthor{\bsnm{Kshirsagar},~\bfnm{A.~M.}\binits{A.~M.}}
(\byear{1959}).
\btitle{Bartlett decomposition and {W}ishart distribution}.
\bjournal{Ann. Math. Statist.}
\bvolume{30}
\bpages{239--241}.
\bid{issn={0003-4851}, mr={0103552}}
\bptok{imsref}%
\end{barticle}
\endbibitem

\bibitem{Kusuoka}
\begin{bincollection}[mr]
\bauthor{\bsnm{Kusuoka},~\bfnm{Shigeo}\binits{S.}}
(\byear{2004}).
\btitle{Approximation of expectation of diffusion processes based on {L}ie
  algebra and {M}alliavin calculus}.
In \bbooktitle{Advances in Mathematical Economics. {V}ol. 6}.
\bseries{Adv. Math. Econ.}
\bvolume{6}
\bpages{69--83}.
\bpublisher{Springer}, \baddress{Tokyo}.
\bid{doi={10.1007/978-4-431-68450-3_4}, mr={2079333}}
\bptok{imsref}%
\end{bincollection}
\endbibitem

\bibitem{Levin}
\begin{barticle}[mr]
\bauthor{\bsnm{Levin},~\bfnm{J.~J.}\binits{J.~J.}}
(\byear{1959}).
\btitle{On the matrix {R}iccati equation}.
\bjournal{Proc. Amer. Math. Soc.}
\bvolume{10}
\bpages{519--524}.
\bid{issn={0002-9939}, mr={0108628}}
\bptok{imsref}%
\end{barticle}
\endbibitem

\bibitem{LyonsVictoir}
\begin{barticle}[mr]
\bauthor{\bsnm{Lyons},~\bfnm{Terry}\binits{T.}} \AND
  \bauthor{\bsnm{Victoir},~\bfnm{Nicolas}\binits{N.}}
(\byear{2004}).
\btitle{Cubature on {W}iener space}.
\bjournal{Proc. R. Soc. Lond. Ser. A Math. Phys. Eng. Sci.}
\bvolume{460}
\bpages{169--198}.
\bnote{Stochastic analysis with applications to mathematical finance}.
\bid{doi={10.1098/rspa.2003.1239}, issn={1364-5021}, mr={2052260}}
\bptok{imsref}%
\end{barticle}
\endbibitem

\bibitem{STELZER}
\begin{barticle}[mr]
\bauthor{\bsnm{Mayerhofer},~\bfnm{Eberhard}\binits{E.}},
  \bauthor{\bsnm{Pfaffel},~\bfnm{Oliver}\binits{O.}} \AND
  \bauthor{\bsnm{Stelzer},~\bfnm{Robert}\binits{R.}}
(\byear{2011}).
\btitle{On strong solutions for positive definite jump diffusions}.
\bjournal{Stochastic Process. Appl.}
\bvolume{121}
\bpages{2072--2086}.
\bid{doi={10.1016/j.spa.2011.05.006}, issn={0304-4149}, mr={2819242}}
\bptnote{check year}%
\bptok{imsref}%
\end{barticle}
\endbibitem

\bibitem{Ninomiya2}
\begin{barticle}[mr]
\bauthor{\bsnm{Ninomiya},~\bfnm{Mariko}\binits{M.}} \AND
  \bauthor{\bsnm{Ninomiya},~\bfnm{Syoiti}\binits{S.}}
(\byear{2009}).
\btitle{A new higher-order weak approximation scheme for stochastic
  differential equations and the {R}unge--{K}utta method}.
\bjournal{Finance Stoch.}
\bvolume{13}
\bpages{415--443}.
\bid{doi={10.1007/s00780-009-0101-4}, issn={0949-2984}, mr={2519839}}
\bptok{imsref}%
\end{barticle}
\endbibitem

\bibitem{NV}
\begin{barticle}[mr]
\bauthor{\bsnm{Ninomiya},~\bfnm{Syoiti}\binits{S.}} \AND
  \bauthor{\bsnm{Victoir},~\bfnm{Nicolas}\binits{N.}}
(\byear{2008}).
\btitle{Weak approximation of stochastic differential equations and application
  to derivative pricing}.
\bjournal{Appl. Math. Finance}
\bvolume{15}
\bpages{107--121}.
\bid{doi={10.1080/13504860701413958}, issn={1350-486X}, mr={2409419}}
\bptok{imsref}%
\end{barticle}
\endbibitem

\bibitem{Odell}
\begin{barticle}[mr]
\bauthor{\bsnm{Odell},~\bfnm{P.~L.}\binits{P.~L.}} \AND
  \bauthor{\bsnm{Feiveson},~\bfnm{A.~H.}\binits{A.~H.}}
(\byear{1966}).
\btitle{A numerical procedure to generate a sample covariance matrix}.
\bjournal{J. Amer. Statist. Assoc.}
\bvolume{61}
\bpages{199--203}.
\bid{issn={0162-1459}, mr={0192635}}
\bptok{imsref}%
\end{barticle}
\endbibitem

\bibitem{Hocking}
\begin{barticle}[auto:STB|2012/06/07|11:40:19]
\bauthor{\bsnm{Smith},~\bfnm{W.~B.}\binits{W.~B.}} \AND
  \bauthor{\bsnm{Hocking},~\bfnm{R.~R.}\binits{R.~R.}}
(\byear{1972}).
\btitle{Algorithm as 53: Wishart variate generator}.
\bjournal{J. R. Stat. Soc. Ser. C. Appl. Stat.}
\bvolume{21}
\bpages{341--345}.
\bptok{imsref}%
\end{barticle}
\endbibitem

\bibitem{Strang}
\begin{barticle}[mr]
\bauthor{\bsnm{Strang},~\bfnm{Gilbert}\binits{G.}}
(\byear{1968}).
\btitle{On the construction and comparison of difference schemes}.
\bjournal{SIAM J. Numer. Anal.}
\bvolume{5}
\bpages{506--517}.
\bid{issn={0036-1429}, mr={0235754}}
\bptok{imsref}%
\end{barticle}
\endbibitem

\bibitem{Talay}
\begin{barticle}[mr]
\bauthor{\bsnm{Talay},~\bfnm{Denis}\binits{D.}} \AND
  \bauthor{\bsnm{Tubaro},~\bfnm{Luciano}\binits{L.}}
(\byear{1990}).
\btitle{Expansion of the global error for numerical schemes solving stochastic
  differential equations}.
\bjournal{Stoch. Anal. Appl.}
\bvolume{8}
\bpages{483--509}.
\bid{doi={10.1080/07362999008809220}, issn={0736-2994}, mr={1091544}}
\bptok{imsref}%
\end{barticle}
\endbibitem

\bibitem{semBachelier}
\begin{bmisc}[auto:STB|2012/06/07|11:40:19]
\bauthor{\bsnm{Teichmann},~\bfnm{J.}\binits{J.}}
(\byear{2010}).
\bhowpublished{Covariance matrix valued affine processes structure and
  numerics. Talk at the Bachelier Seminar in Paris.}
\bptok{imsref}%
\end{bmisc}
\endbibitem

\end{thebibliography}
\end{document}